\documentclass[12pt,reqno,final]{amsart}

\usepackage[utf8]{inputenc}
\usepackage{geometry}
\usepackage{xcolor}
\usepackage[right]{lineno}
\usepackage[color,notref,notcite]{showkeys}
\definecolor{refkey}{rgb}{1,0,0}
\definecolor{labelkey}{rgb}{0,0,1}
\usepackage{setspace}
\usepackage{indentfirst}
\usepackage{tgtermes}
\usepackage{mathrsfs,mathtools,amssymb,stmaryrd,tikz-cd,MnSymbol}
\usepackage[all,cmtip]{xy}
\usetikzlibrary{decorations.pathmorphing}

\newtheorem{theorem}{Theorem}[section]
\newtheorem{lemma}[theorem]{Lemma}
\newtheorem{proposition}[theorem]{Proposition}
\newtheorem{corollary}[theorem]{Corollary}
\newtheorem{definition}[theorem]{Definition}
\newtheorem{remark}[theorem]{Remark}
\newtheorem{example}[theorem]{Example}

\newtheorem{introthm}{Theorem}

\newtheorem{intropro}[introthm]{Problem}
\newcommand{\mathcalOM}{\mathcal{C}^\infty_M}
\usepackage[backend=biber,bibencoding=utf8,style=numeric,
url=true,doi=false,eprint=true,isbn=false,
hyperref=auto,backref=false]{biblatex}
\usepackage[pagebackref=false,hidelinks,
unicode=true,bookmarks=true,linktoc=page,
pdfstartview={FitH},final=true]{hyperref}

\allowdisplaybreaks[3]
\binoppenalty=\maxdimen
\relpenalty=\maxdimen
\usepackage{abbreviations}

\title{\texorpdfstring{$A_\infty$}{Aoo}-algebras from Lie pairs}

\dedicatory{Dedicated to the memory of our colleague and friend Kirill C.\ H.\ MacKenzie}

\thanks{Research partially supported by NSF grants DMS-2001599 and 
DMS-2302447.}

\author{Mathieu Stiénon}
\address{Department of Mathematics, Pennsylvania State University}
\email{stienon@psu.edu}

\author{Luca Vitagliano}
\address{Dipartimento di Matematica, Università di Salerno}
\email{lvitagliano@unisa.it}

\author{Ping Xu}
\address{Department of Mathematics, Pennsylvania State University}
\email{ping@math.psu.edu}

\begin{document}

\begin{abstract}
Given an inclusion $A\hookrightarrow L$ of Lie algebroids
sharing the same base manifold $M$, i.e.\ a Lie pair,
we prove that the space $\Gamma(\Lambda^\bullet A^\vee)\otimes_{R}
\frac{\mathcal{U}(L)}{\mathcal{U}(L)\cdot\Gamma(A)}$, where $R=C^\infty(M)$,
admits an $A_\infty$-algebra structure, unique up to $A_\infty$-isomorphisms.
As a consequence, the Chevalley–Eilenberg cohomology
 $H^\bullet_{\CE} \big( A, \frac{\mathcal{U}(L)}{\mathcal{U}(L)\cdot\Gamma(A)} \big)$
admits a canonical associative algebra structure.
This $A_\infty$-algebra can be considered
as the universal enveloping algebra of the $L_\infty$-algebroid
$A[1]\times_M L/A$.
Our construction is based on the homotopy equivalence of 
the $L_\infty$-algebroid $A[1]\times_M L/A$ and the 
dg Lie algebroid corresponding to the comma double Lie algebroid
of Jotz-Mackenzie.
\end{abstract}

\maketitle
\tableofcontents


\section{Introduction}

A \emph{Lie pair} $(L,A)$ is a pair consisting of two Lie algebroids
(over a field $\KK$ of characteristic zero)
over the same manifold $M$ together with a Lie algebroid inclusion
$A \hookrightarrow L$ (covering the identity map $\id_M$).
Lie pairs appear naturally in a variety of places in
classical areas of mathematics such as
complex geometry, in the theory of foliations, 
and the theory of $\frakg$-manifolds (manifolds equipped with
an infinitesimal action of a Lie algebra $\frakg$).
For instance, the tangent distribution $T_\cF$ to a foliation $\cF$
of a manifold $M$, together with the tangent bundle $T_M$
and the natural inclusion $T_{\cF}\hookrightarrow T_M$ is a Lie pair.

Given a Lie pair $(L,A)$, the quotient vector bundle $L/A$ admits a canonical flat
$A$-connection: the \emph{Bott connection} $\nabla^{\mathrm{Bott}}$ defined by
\[ \nabla^{\mathrm{Bott}}_a q(l) = q ([a,l]),
\quad a \in \Gamma(A), l \in \Gamma(L), \]
where $q : L \to L/A$ is the projection.
As a result, $L/A$ is an $A$-module and
$\Gamma(\Lambda^\bullet A^\vee \otimes L/A)$ is 
naturally a dg module over the commutative differential 
graded algebra (\dga)
$\big(\Gamma(\Lambda^\bullet A^\vee), d_A\big)$.
For instance, when $A = T_\cF$ is the tangent distribution to a foliation
$\cF$ of $M$, the dg module
$\Gamma(\Lambda^\bullet T^\vee_{\cF} \otimes T_M/T_{\cF})$
can be regarded as the space of vector fields transverse to the foliation $\cF$ or
the space of formal vector fields on the differentiable stack
presented by the holonomy groupoid of $\cF$.
In~\cite{MR3277952},
one of the authors proved that the space 
$\Gamma(\Lambda^\bullet T^\vee_\cF \otimes T_M/T_\cF)$ is indeed equipped
with an $L_\infty$-algebra structure (actually a homotopy Lie-Rinehart
algebra structure over $\Gamma(\Lambda^\bullet T^\vee_\cF)$),
unique up to $L_\infty$-isomorphisms.
More generally, for a generic Lie pair $(L,A)$,
the space $\Gamma(\Lambda^\bullet A^\vee \otimes L/A)$
is known to admit an $L_\infty$-algebra structure,
which is in fact a structure of homotopy Lie-Rinehart
algebra over $\Gamma (\Lambda^\bullet A^\vee)$,
unique up to $L_\infty$-isomorphisms 
\cite{MR4325718,MR4091493}.
That is, $A[1]\times_M L/A \to A[1]$
is an $L_\infty$-algebroid.
This leads to the following natural

\begin{intropro}
What is the universal enveloping algebra
of the $L_\infty$-algebroid $A[1]\times_M L/A \to A[1]$?
\end{intropro}

Baranovsky proposed a construction of universal enveloping algebra for $L_\infty$-algebras based on the cobar construction and homological perturbation theory \cite{MR2470385}.
However, as is well-known, there is no canonical
construction of universal enveloping algebra for
$L_\infty$-algebras.
As for $L_\infty$-algebr\emph{oids}, little is known concerning
universal enveloping algebras.

Given a Lie pair $(L,A)$ over a base manifold $M$,
the Lie algebroid $A$ acts on the universal enveloping
algebra of the Lie-Rinehart algebra $(\Gamma(L),R)$
from both the left and the right --- we use the symbol $R$
to denote the algebra $C^\infty(M)$ of smooth functions on 
$M$.
Multiplication by element of $\Gamma(A)$ and $R$
from the left in $\cU(L)$ induces a canonical 
$A$-module structure on the quotient
$\cU_{L/A}=\frac{\cU(L)}{\cU(L)\cdot\Gamma(A)}$.
This $A$-module structure extends in a natural way
the Bott connection. It follows that
$\Gamma(\Lambda^\bullet A^\vee) \otimes_\cirm \cU_{L/A}$
is naturally a differential graded (dg) module over
the \dga\ $\big(\Gamma(\Lambda^\bullet A^\vee),d_A\big)$. 
In the special case of a Lie pair
$(L, A)=(T_M,T_{\cF})$ arising from a foliation $\cF$ of $M$,
this dg module
can be understood as the space of ``differential operators
transverse to the foliation.''

Here is the rationale.

The second author proved in~\cite{MR3313214}
that the dg module
$\Gamma(\Lambda^\bullet T^\vee_{\cF})\otimes_\cirm
\cU_{T_M/T_{\cF}}$ admits an $A_\infty$-algebra structure,
which is unique up to $A_\infty$-isomorphisms.
Hence, its cohomology is canonically an associative algebra.
The proof is based on (1) the construction of a contraction from the dg module $D(T_{\cF}[1])$ of differential operators on the dg manifold $T_{\cF}[1]$ to the dg module
$\Gamma(\Lambda^\bullet T^\vee_{\cF})\otimes_\cirm
\cU_{T_M/T_{\cF}}$ and (2) the homotopy transfer theorem
for $A_\infty$-algebras;
the differential graded algebra structure on $D(T_{\cF}[1])$
transfers to an $A_\infty$-algebra structure on
$\Gamma(\Lambda^\bullet T^\vee_{\cF})\otimes_\cirm \cU_{T_M/T_{\cF}}$.

Furthermore, the $L_\infty$-algebroid structure
on $T_{\cF}[1]\times_M T_M/T_{\cF}\to T_{\cF}[1]$
described by the second author in~\cite{MR3277952}
was obtained by homotopy transfer from the tangent dg Lie 
algebroid of the dg manifold $T_{\cF}[1]$.
Therefore, since the dg algebra $D(T_{\cF}[1])$
is the universal enveloping algebra of the tangent Lie algebroid to the dg manifold $T_{\cF}[1]$,
the $A_\infty$-algebra $\Gamma(\Lambda^\bullet T^\vee_{\cF})
\otimes_\cirm \cU_{T_M/T_{\cF}}$ can be considered
as the universal enveloping algebra of the $L_\infty$-algebroid $T_{\cF}[1]\times_M T_M/T_{\cF}\to T_{\cF}[1]$,
i.e.\ the differential operators transverse to the foliation $\cF$.

It is natural to expect that a similar result holds
for all Lie pairs. More precisely, the dg module
$\Gamma(\Lambda^\bullet A^\vee) \otimes_\cirm \cU_{L/A}$
ought to be the substrate of the universal enveloping algebra
of the $L_\infty$-algebroid $A[1]\times_R L/A \to A[1]$.
This interpretation is also supported by the
fact that the Chevalley–Eilenberg cohomology
$H^\bullet_{\CE} (A,\cU_{L/A})$,
i.e.\ the cohomology of
$\big(\Gamma(\Lambda^\bullet A)\otimes_\cirm
\cU_{L/A},d_A\big)$ is a Lie subalgebra
of a canonical graded Lie algebra \cite{MR4325718}.

The initial step of our strategy,
which is motivated by the constructions 
in~\cite{MR3313214,MR3277952}, consists in constructing,
for a generic Lie pair $(L,A)$, a dg Lie algebroid
that plays a role analogous to the tangent bundle $T_\cM\to\cM$ to the dg manifold $\cM=T_{\cF}[1]$
in the foliation case $(L,A)=(T_M,T_{\cF})$.

Our first result is the following general construction.

\begin{introthm}[See~Theorem~\ref{Zug}]
\label{introthmB}
Let $A$ and $L$ be Lie algebroids over $\KK$ with the same base manifold $M$,
and let $\phi:A\to L$ be a Lie algebroid morphism.
Then the pull-back Lie algebroid
$\pi^! L\xto{\breve{\varpi}} A[1]$ of the Lie algebroid $L\to M$ (see~\cite{MR2157566})
through the surjective submersion $\pi: A[1]\to M$
admits a structure of \emph{dg} Lie algebroid
induced, in a canonical way, by $\phi$.
\end{introthm}

In particular, given any Lie pair $(L,A)$ over $\KK$,
the associated pull-back Lie algebroid
$\pi^! L\xto{\breve{\varpi}} A[1]$
automatically inherits a \emph{dg} Lie algebroid structure,
namely the one induced by the inclusion morphism
$A \hookrightarrow L$.
This dg Lie algebroid is the main object of study
in this paper. Its space of sections $\Gamma(\pi^! L)$
is a dg module over the \dga\
$\big( \Gamma(\Lambda^\bullet A^\vee), d_A\big)$. 
We will use the symbol $\sQ$ to denote the differential
on $\Gamma(\pi^! L)$. 

Next, we build up a
$\Gamma(\Lambda^\bullet A^\vee)$-linear contraction 
\begin{equation*}
\begin{tikzcd}
\big(\Gamma(\pi^! L),\sQ\big)
\arrow[r, "P_0", shift left] \arrow["H_0", loop left]
& \big(\Gamma(\Lambda^\bullet A^\vee\otimes L/A),\dA\big)
\arrow[l, "I_0", shift left]
\end{tikzcd}
\end{equation*}
from a choice of a splitting
of the short exact sequence of vector bundles
\begin{equation}\label{eq:splitting0}
0 \to A \to L \to L/A \to 0
.\end{equation}
This contraction enables us to homotopy
transfer the dg Lie algebroid structure on $\pi^! L$
to an $L_\infty$-algebroid structure on
$A[1]\times_M L/A \to A[1]$.
We prove the following

\begin{introthm}[See~Theorem~\ref{prop:L_infty_alg}]
\label{introthmC}
The $L_\infty$-algebroid structure on 
$A[1]\times_M L/A\to A[1]$ inherited from the dg Lie 
algebroid $\pi^! L\xto{\breve{\varpi}}A[1]$
by way of the homotopy transfer theorem for $L_\infty$-algebroids (see~Theorem~\ref{theor:ht_L_infty_alg})
coincides with the one described in~\cite[Proposition~4.1]{MR4091493}.
\end{introthm}

Thus we expect that one possible
candidate for the universal enveloping algebra of
the $L_\infty$-algebroid $A[1]\times_R L/A \to A[1]$ 
is $\Gamma(\Lambda^\bullet A^\vee)\otimes_R\cU_{L/A}$
endowed with an $A_\infty$-algebra structure obtained
by homotopy transfer from the differential graded
associative algebra structure on the universal enveloping 
algebra $\big(\cU(\pi^! L),D_\cU\Big)$
of the dg Lie algebroid $\pi^! L\xto{\breve{\varpi}}A[1]$.

To construct the desired contraction
from $\big(\cU(\pi^! L),D_\cU\Big)$ to
$\big(\Gamma(\Lambda^\bullet A^\vee)\otimes_R\cU_{L/A},
\dA\big)$,
we use the PBW type isomorphisms established in
\cite{MR4271478, MR3319134, MR3910470} to identify 
$\cU(\pi^! L)$ and $\Gamma(\Lambda^\bullet A^\vee)\otimes_R\cU_{L/A}$
with $\Gamma(S(\pi^! L))$ and $\Gamma \big( \Lambda^\bullet A^\vee\otimes S (L/A) \big)$, respectively,
as dg modules over the \dga\ $\big(\Gamma(\Lambda^\bullet A^\vee), d_A \big)$,
and then use the tensor trick and the homological perturbation lemma.
The homotopy transfer theorem for $A_\infty$-algebras then yields the following 

\begin{introthm}[See~Theorem~\ref{theor:main}]
\label{introthmD}
Let $(L,A)$ be a Lie pair.
\begin{enumerate}
\item Each choice of splitting of the short exact sequence
of vector bundles \eqref{eq:splitting0}, together with
some additional geometric data,
determines an $A_\infty$-algebra structure
on $\Gamma(\Lambda^\bullet A^\vee)\otimes_R\cU_{L/A}$,
whose unary operation is precisely the differential $\dA$.
\item The $A_\infty$-algebras arising from all possible
choices of splitting (and additional geometric data)
are all mutually isomorphic.
\item In particular, the cohomology
$H\big(\Gamma(\Lambda^\bullet A^\vee)\otimes_R\cU_{L/A},\dA\big)$
carries a canonical graded associative algebra structure.
\end{enumerate}
\end{introthm}

Let us summarize. Every Lie pair $(L,A)$ yields
(according to Theorem~\ref{introthmB})
a dg Lie algebroid structure on
$\pi^!L\xto{\breve{\varpi}}A[1]$,
which in turn yields (according to Theorem~\ref{introthmC})
an $L_\infty$-algebroid structure on
$A[1]\times_M L/A\to A[1]$ by homotopy transfer.
Furthermore, the differential graded associative algebra
structure on the universal enveloping algebra
of the dg Lie algebroid $\pi^!L\xto{\breve{\varpi}}A[1]$
yields (according to Theorem~\ref{introthmD})
an $A_\infty$-algebra structure on
$\Gamma(\Lambda^\bullet A^\vee)\otimes_R\cU_{L/A}$
again by homotopy transfer.
Thus, the $A_\infty$-algebra
$\Gamma(\Lambda^\bullet A^\vee)\otimes_R\cU_{L/A}$
can be thought of as
the universal enveloping algebra
of the $L_\infty$-algebroid
$A[1]\times_M L/A\to A[1]$.

As a further motivation for considering the $A_\infty$-algebra structure on $\Gamma(\Lambda^\bullet A^\vee)\otimes_R\cU_{L/A}$ first notice that, in 
various specific cases, its $0$-th cohomology $H^0$ has precise geometric meaning. For instance, we can consider
the following four special cases:
\begin{enumerate}
\item Let $(L,A) = (T_M,T_{\cF})$ be the Lie pair arising
from a \emph{simple} foliation $\cF$
on a manifold $M$.
Then $H^0$ is the space of differential operators on the leaf-space $M/\cF$.
\item Let $(L,A) = (\frakg,\frakh)$, where $\frakg$
is the Lie algebra of a connected Lie group $G$ and
$\frakh$ is the Lie subalgebra of $\frakg$ corresponding
to a closed subgroup $H$ of $G$.
Then $H^0$ is the space of $G$-invariant differential operators on the homogeneous space $G/H$.
\item Let $(L,A) = (T_{\CC}X,T^{0,1}X)$ be the (complex) Lie pair
corresponding to a complex manifold $X$.
This means that $T_{\CC} X = TX \otimes \CC$
is the complexified tangent bundle
and $T^{0,1} X \subseteq T_{\CC} X$
is the bundle of antiholomorphic tangent vectors (i.e.\ $T^{0,1}X$
is the $-i$-eigenbundle of the almost complex structure). Then $H^0$ is the space of holomorphic differential operators on $X$.
\item Let $(L,A)=((\frakg\ltimes M)\bowtie TM,\frakg\ltimes M)$
be the Lie pair arising from a $\frakg$-manifold $M$.
That is, a Lie algebra $\frakg$ acts on a manifold $M$;
$\frakg\ltimes M$ is the action Lie algebroid;
$(\frakg\ltimes M,TM)$ is the corresponding matched pair of Lie algebroids; and we consider the resulting
inclusion of Lie algebroids
$\frakg\ltimes M\hookrightarrow(\frakg\ltimes M)\bowtie TM$
over the base manifold $M$.
In this case $H^0$ is the space of $\frakg$-invariant
differential operators on $M$.
\end{enumerate}

In all four cases above, the associative multiplication
in $H^0$ is exactly the one induced
by the $A_\infty$-algebra structure. The $A_\infty$-algebra 
does of course contain much more information than just
its $0$-th cohomology.
Moreover, in more singular situations,
e.g.\ when the foliation $\cF$ of the first special case
above is non-simple, the $A_\infty$-algebra can be seen
as a regular non-acyclic resolution of its ``singular''
$0$-th cohomology, and as an appropriate replacement
for the latter.

As is well-known, the universal enveloping algebra
of a finite dimensional Lie algebra $\frakg$ gives rise to
a deformation quantization of the linear Poisson structure
$\frakg^\vee$ \cite{MR1747916}.
Provided a connection is chosen, the same idea works
for the linear Poisson structure $A^\vee$ corresponding
to a Lie algebroid $A$ --- see~\cite{MR1687747}.
Analogously, we expect that the $A_\infty$-algebra structure on
$\Gamma(\Lambda^\bullet A^\vee)\otimes_R\cU_{L/A}$
would give rise to a flat deformation quantization of
the linear $P_\infty$-structure on $A [1]\times (L/A)^\vee$
corresponding to the $L_\infty$-algebroid
$A[1]\times_R L/A\to A[1]$. Note that it is still an open
question to determine when a given $P_\infty$-structure admits
a flat deformation quantization \cite{MR2304327}.

The geometric meaning of the dg Lie algebroid of Theorem~\ref{introthmB} is explored in Appendix~\ref{Luca}
where we link it with the comma double Lie algebroid.
Every Lie algebroid morphism $\phi: A\to L$ determines a comma \emph{double} Lie algebroid,
whose explicit construction is due to Jotz-Mackenzie~\cite{MR4322148} (see also \cite[Theorem~2.16]{MR4421028}).
If $\KK$ is the field $\RR$ of real numbers,
it is the infinitesimal counterpart of the comma double Lie groupoid
of Brown--Mackenzie~\cite[Example~1.8]{MR1170713}
(see also \cite[Example~2.5]{MR1174393}).
According to Voronov~\cite{MR2971727}, a double Lie algebroid
in the sense of Mackenzie~\cite{arXiv:math/9808081,MR1650045} gives rise naturally to a dg Lie algebroid. 
In Appendix, we establish the following

\begin{introthm}[See~Theorem~\ref{thm:Cortona}]
Let $A$ and $L$ be Lie algebroids over $\KK$ with the same base manifold $M$,
and let $\phi:A\to L$ be a Lie algebroid morphism.
Then the dg Lie algebroid corresponding to the comma double Lie algebroid coincides with the one in Theorem~\ref{introthmB}.
\end{introthm}

As a consequence of this connection to double Lie groupoids
when working over $\KK=\RR$,
the $A_\infty$-algebra structure on
$\Gamma(\Lambda^\bullet A^\vee)\otimes_R\cU_{L/A}$
arising from a Lie pair $(L,A)$ ought to admit
a geometric interpretation in terms of higher stacks,
which is worthy of further exploration in the future.

Finally, some remarks are in order.
After the initial version of the present paper was posted,
the constructions set forth in it have been applied
successfully to other undertakings.
Recently, Liao~\cite{MR4665716} proved that the Atiyah class 
\cite{MR3319134} of the dg Lie algebroid
$\pi^! L\xto{\breve{\varpi}}A[1]$ is indeed isomorphic
to the Atiyah class of the Lie pair $(L,A)$ \cite{MR3439229}.
In fact, one can prove \cite{BSSX} that the Kapranov $L_\infty[1]$-algebra $\Gamma(A[1];\pi^! L)$
associated to the dg Lie algebroid $\pi^! L\xto{\breve{\varpi}} A[1]$
is quasi-isomorphic to the Kapranov $L_\infty[1]$-algebra
$\Gamma \big(\Lambda^\bullet A^\vee \otimes L/A \big)$
of the Lie pair $(L,A)$.
In~\cite{MR4325718},
the authors considered the dg module
$\Gamma(\Lambda^\bullet A^\vee)\otimes_\cirm\cD_{\mathrm{poly}}^\bullet$
of formal polydifferential operators on the Lie pair $(L,A)$,
where $\cD_{\mathrm{poly}}^\bullet=\bigoplus_{k\geq 0}\cU_{L/A}^{\otimes k}$
(the tensor products are over $\cirm$),
and showed that its cohomology carries a canonical Gerstenhaber algebra structure.
The $L_\infty$-algebra structure on cochains inducing the Gerstenhaber algebra structure on cohomology
is unique up to $L_\infty$-isomorphisms.
For degree reasons, the direct summand
$\Gamma(\Lambda^\bullet A^\vee)\otimes_\cirm\cU_{L/A}$
of $\Gamma(\Lambda^\bullet A^\vee)\otimes_\cirm
\cD_{\mathrm{poly}}^\bullet$ is preserved by the $L_\infty$-brackets.

\section{Dg Lie algebroids associated to Lie algebroid morphisms}
\label{doudou}

In this paper, the symbol $\KK$ denotes either of
the fields $\RR$ and $\CC$, and $R=C^\infty (M)$.

A \emph{$\ZZ$-graded manifold} $\cM$ over $\KK$
consists of a smooth manifold $M$
together with a sheaf $\mathfrak{A}$ of $\ZZ$-graded commutative
$\mathcalOM$-algebras over $M$ such that there exist a covering
$\{U_i\}_{i\in I}$ of $M$ by open submanifolds $U_i\subset M$
and a family of isomorphisms
$\mathfrak{A}(U_i)\cong\mathcal{O}_M(U_i)\otimes_\KK{S}(V^\vee)$,
where $V$ is a fixed $\ZZ$-graded vector space over the field $\KK$,
$V^\vee$ denotes the $\KK$-dual of $V$,
and ${S}(V^\vee)$ denotes the $\KK$-algebra of polynomial functions on $V$.
The manifold $M$ is called the support of the graded manifold $\cM$.
By $C^\infty(\cM)$, we denote the $\ZZ$-graded $\KK$-algebra $\mathfrak{A}(M)$
of global sections of the sheaf $\mathfrak{A}$.
See~\cite{MR2709144,MR2275685,MR2819233}.

A \emph{dg manifold} is a $\ZZ$-graded manifold $\cM$ endowed with a
homological vector field, i.e.\ a degree $+1$ derivation $Q$ on
$C^\infty(\cM)$ satisfying $\lie{Q}{Q}=0$.

Dg Lie algebroids are Lie algebroid objects in the category of dg manifolds.
In the $C^\infty$-context,
they were first introduced by Mehta~\cite{MR2534186},
who called them $Q$-algebroids.
First, recall that a Lie algebroid object in the category of $\ZZ$-graded
manifolds consists of a vector bundle object $\cA\to\cM$ of $\ZZ$-graded
manifolds together with a degree $0$
bundle map $\rho:\cA\to T_{\cM}$, called anchor,
and a structure of graded Lie algebra on $\Gamma(\cA)$
with Lie bracket satisfying
\[ \lie{X}{f\cdot Y} = \big(\rho(X)f\big)\cdot Y
+ (-1)^{\degree{X}\degree{f}} f\cdot \lie{X}{Y} \]
for all homogeneous $X,Y\in\Gamma(\cA)$ and $f\in C^\infty(\cM)$.
According to a well known theorem of Va\u{\i}ntrob~\cite{MR1480150},
the Chevalley--Eilenberg differential
\[ d_{\cA}:\Gamma(\Lambda^\bullet\cA^\vee)
\to\Gamma(\Lambda^{\bullet+1}\cA^\vee) \]
of the Lie algebroid $\cA\to\cM$
can be seen as a homological vector field on $\cA[1]$
so that $( \cA[1], d_{\cA})$ is a dg manifold.

Now assume that $\cA\to\cM$ is a dg vector bundle
--- see~\cite{MR2709144} for the notion of dg vector bundle.
The homological vector field $Q_{\cA}$ on $\cA$
is \emph{linear} with respect to the vector bundle structure
on $\cA\to\cM$. In other words, the derivation $Q_{\cA}$
of $C^\infty(\cA)$ stabilizes the subspace
$\Gamma(\cA^\vee)$ of fiberwise linear functions
on $\cA$.
Therefore the shift functor $[1]$ makes sense \cite{MR2534186}.

By $\cQ$, we denote the homological vector field
on $\cA[1]$ obtained from $Q_{\cA}$ by applying the shift functor $[1]$.
The following definition is due to Mehta \cite{MR2534186}.

\begin{definition}
A dg Lie algebroid consists of a dg vector bundle
$\cA\to\cM$, with homological
vector fields on $\cA$ and
$\cM$ denoted $Q_\cA$ and $Q_\cM$, respectively, and a Lie algebroid structure
on the vector bundle $\cA\to\cM$ such that the dg
and the Lie structures are compatible in the sense
that the Chevalley--Eilenberg differential $d_\cA$
of the Lie algebroid structure on $\cA\to\cM$
and the homological vector field $\cQ$ on $\cA[1]$
induced by $Q_\cA$ ---
two derivations of the graded algebra
$\Gamma(\Lambda^\bullet\cA^\vee)\cong C^\infty(\cA[1])$
--- commute:
\[ \lie{d_\cA}{\cQ}=0 .\]
\end{definition}

As a consequence, the pair of differentials $d_{\cA}$ and $\cQ$
make the algebra $\Gamma(\Lambda^\bullet\cA^\vee)$
--- which is doubly graded by the `$\bullet$-degree'
and the `internal' degree --- a double complex.
While the Lie algebroid differential $d_\cA$ raises
the $\bullet$-degree by $1$,
the differential $\cQ$ raises the internal degree by $1$.
Following Mehta~\cite{MR2534186},
we call the total cohomology of this double complex
the dg Lie algebroid cohomology of $\cA$.

The following proposition describes an efficient way of
producing dg Lie algebroids.

\begin{proposition}\label{Prato}
Let $\cA\to\cM$ be a Lie algebroid object
in the category of $\ZZ$-graded manifolds
with anchor map $\rho:\cA\to T_\cM$.
If $s\in\Gamma(\cA)$ is a section of degree $+1$
of $\cA\to\cM$ satisfying $\lie{s}{s}=0$,
then $\cA\to\cM$ admits a structure of dg Lie algebroid.
\end{proposition}

\begin{proof}
From the assumption that $\lie{s}{s}=0$,
it follows that $Q_{\cM}=\rho (s)$ is a homological vector field on $\cM$.
Furthermore, it is routine to see that the operator of degree $+1$
\[ \cQ=\lie{s}{\argument} : \Gamma(\cA)\to \Gamma(\cA) \]
makes $\Gamma(\cA)$ into a dg module over the \dga
 $(C^\infty(\cM),Q_{\cM})$.
Hence $\cA$ acquires a dg manifold structure such that,
according to~\cite[Lemma~1.5]{MR3319134}, $\cA\to \cM$ is a dg vector bundle.
Indeed the algebra $C^\infty(\cA)$ is smoothly generated by its two subspaces $\pi^* C^\infty(\cM)$
and $\Gamma(\cA^\vee)\cong C^\infty_{\operatorname{linear}}(\cA)$ ---
the fiberwise constant and the fiberwise linear functions respectively --- and we have
\begin{gather}
Q_{\cA}(\pi^* f) = \pi^* (Q_{\cM} f) ,\qquad\forall f\in C^\infty(\cM); \label{eq:f} \\
\duality{Q_{\cA}(\varepsilon)}{e} = Q_{\cM}\big(\duality{\varepsilon}{e}\big)
-(-1)^{\degree{\varepsilon}}\duality{\varepsilon}{\lie{s}{e}}
,\qquad\forall\varepsilon\in \Gamma(\cA^\vee), \ e\in\Gamma(\cA) . \label{eq:linear}
\end{gather}
Equations~\eqref{eq:f} and~\eqref{eq:linear} imply that
the differential $\cQ$ on $\Gamma(\Lambda^\bullet\cA)$,
which raises the internal degree by $1$,
coincides with the Lie derivative $\liederivative{s}$ w.r.t.\ the section $s$
on the two subspaces $C^\infty(\cM)$ and $\Gamma(\cA^\vee)$.
It follows from the Leibniz rule that
\[ \cQ=\liederivative{s} \]
as derivations of the whole algebra $\Gamma(\Lambda^\bullet\cA^\vee)$.
Hence we have \[ [d_{\cA},\cQ]=[d_{\cA},\liederivative{s}]=0 \]
and $\cA\to\cM$ is a dg Lie algebroid.
This concludes our proof.
\end{proof}

Next we apply Proposition~\ref{Prato} to a special situation
arising from a Lie algebroid morphism.
Let $A \xto{\pi} M$ and $L \xto{\varpi} M$ be Lie algebroids over $\KK$
with the same base manifold $M$. Let $\phi:A\to L$ be a bundle map
over the identity map $\id_M:M\to M$.
Consider the pull-back of the Lie algebroid $L\to M$ (see~\cite{MR2157566})
through the surjective submersion $\pi: A[1]\to M$:
\[ \pi^! L\xto{\breve{\varpi}} A[1] .\]

As a manifold, $\pi^! L$ is canonically identified
to the fibered product $T_{A[1]}\times_{T_M}L$.
The projection $\breve{\varpi}$ is the composition
of canonical maps $T_{A[1]}\times_{T_M}L\to T_{A[1]}\to A[1]$.
The anchor is the canonical projection map
\[ \pi^! L \cong T_{A[1]}\times_{T_M}L\ni (X,\nu)\longmapsto X\in T_{A[1]} .\]

A section of $\pi^! L\xto{\breve{\varpi}} A[1]$ is a pair $(X,\nu)$,
called \emph{characterizing pair}, consisting of a vector field $X\in \XX(A[1])$
and a map $\nu:A[1]\to L$ covering the identity map $\id_M$ such that the 
diagram
\[ \begin{tikzcd}
A[1] \arrow[r, "X"] \arrow[d, "\nu"'] & T_{A[1]} \arrow[d, "\pi_*"] \\
L \arrow[r, "\rho_L"'] & T_M
\end{tikzcd} \]
commutes.

A section of the Lie algebroid $\pi^! L\xto{\breve{\varpi}}A[1]$
characterized by a pair $(X,\nu)$ is said to be \emph{$\pi$-related}
to a section $v$ of the Lie algebroid $L\xto{\varpi}M$ if
(1) $\nu=v\circ\pi$ and (2) the vector field $X\in\XX(A[1])$
is $\pi$-related to the vector field $\rho_L(v)\in\XX(M)$.
The $C^\infty(A[1])$-module
of sections of $\pi^! L\xto{\breve{\varpi}}A[1]$ is generated
by those sections which are $\pi$-related to sections of $L\xto{\varpi}M$.

A section of the Lie algebroid $\pi^! L\xto{\breve{\varpi}}A[1]$
is said to be \emph{null} if it is $\pi$-related to the zero section of
the Lie algebroid $L\xto{\varpi}M$. In other words, a null section
is characterized by a pair $(X,0)$ such that $\pi_* X=0$.

The Lie bracket on $\Gamma(\pi^! L\xto{\breve{\varpi}}A[1])$ is determined by the
Lie bracket on $\Gamma(L\xto{\varpi}M)$ through the relation
\begin{equation}\label{eq:PEK}
\big[(X, v\circ\pi),(X', v'\circ\pi)\big]=\big( \lie{X}{X'}, \ \lie{v}{v'}\circ\pi \big),
\end{equation}
which holds for all pairs of sections of $\pi^! L$
that are $\pi$-related to sections of $L$.

More generally,
according to \cite[Equation~(20) and Section~4.2]{MR2157566},
the Lie bracket of two homogeneous characterizing pairs
$(X,\nu)$ and $(X',\nu')$ is given by
\begin{equation}\label{eq:BRU}
\big[(X,\nu),(X',\nu')\big]=\big(\lie{X}{X'}, X(\nu')-(-1)^{\degree{X'}\degree{\nu}}X'(\nu)
+\lie{\nu}{\nu'}\big)
.\end{equation}
Here $X(\nu')-(-1)^{\degree{X'}\degree{\nu}}X'(\nu)
+\lie{\nu}{\nu'}$ is an abuse of notation and really means
\begin{equation}\label{eq:PEK12}
\sum_i X(\xi'_i)\otimes l'_i
-(-1)^{\degree{X'}\degree{\nu}} \sum_i X'(\xi_i)\otimes l_i
+\sum_{i,j}(-1)^{\degree{l_i}\degree{\xi'_j}} (\xi_i\wedge\xi'_j)
\otimes \lie{l_i}{l'_j}
\end{equation}
if $\nu=\sum_i \xi_i\otimes l_i$ and $\nu'=\sum_j \xi'_j\otimes l'_j$
with $\xi_i,\xi'_j\in\Gamma(\Lambda^\bullet A^\vee)$
and $l_i,l'_j\in\Gamma(L)$.

Let $\boldphi$ be the composition $A[1]\to A\xto{\phi}L$.
Essentially, $A[1]\xto{\boldphi}L$ is simply the bundle map $\phi$ regarded as a
bundle of two-term complexes over $M$ with $A$ placed in degree $-1$ and $L$ placed in degree $0$.
Now think of the Chevalley--Eilenberg differential $\dA$ as a vector field on
$A[1]$. In this way, we obtain a pair $(\dA,\boldphi)$.

\begin{lemma}\label{lem:s_phi}
The pair $(\dA,\boldphi)$ characterizes a section $s_{\boldphi}$ of $\pi^! L\xto{\breve{\varpi}} A[1]$
if and only if $\rho_L\circ\phi=\rho_A$.
Furthermore, this section $s_{\boldphi}$ satisfies
$\lie{s_{\boldphi}}{s_{\boldphi}}=0$
if and only if the bundle map $\phi: A\to L$ is a Lie algebroid morphism.
\end{lemma}

\begin{proof}
For any $f\in\cirm$, by definition of $d_A$, we have
\[ \duality{\pi_* \circ \dA}{df}=\rho^*_A (df)\in \Gamma(A^\vee) ,\]
and
\[ \duality{\rho_L \circ \phi}{df}= \big(\phi^*\circ \rho^*_L\big) (df) \in \Gamma(A^\vee) .\]
Therefore $s_{\boldphi}$ is a section of $\pi^! L\xto{\breve{\varpi}} A[1]$
if and only if $\rho_L \circ \phi=\rho_A$.

Choose local coordinates $(x^1,\cdots,x^n)$ on $M$
and a local frame $(e_1,\cdots,e_m)$ for $A$.
They induce a local chart $(x^1,\cdots,x^n;\xi^1,\cdots,\xi^m)$ on $A[1]$.

We have
\[ \rho_A (e_i) =\sum_j a_i^j (x)\field{x}{j} \qquad\text{and}\qquad
\lie{e_i}{e_j} =\sum_k c^k_{ij} (x) e_k ,\]
where the $a_i^j (x)$ and $c^k_{ij} (x)$ are functions of the local coordinates $x^1,\dots,x^n$.
Then, in the local chart $(x^1,\cdots,x^n,\xi^1,\cdots,\xi^m)$ on $A[1]$,
the homological vector field $\dA$ reads
\begin{equation}
\label{eq:dA}
 \dA=\sum_{i,j} a_i^j (x)\xi^i \field{x}{j}
-\frac{1}{2}\sum_{i,j,k} c^k_{ij}(x) \xi^i\xi^j
\frac{\partial}{\partial \xi^k}=\sum_i \xi^i X_i ,
\end{equation}
where
\begin{equation} 
\label{eq:Xi} 
X_i=\sum_j a_i^j (x)\field{x}{j}-\frac{1}{2}\sum_{j,k} c^k_{ij} (x)\xi^j
\frac{\partial}{\partial \xi^k} .
\end{equation}
On the other hand, we have $\boldphi=\sum\xi^i\cdot\phi(e_i)\circ\pi$ and thence
\begin{equation}
\label{eq:sphi}
s_{\boldphi}=\sum_i\xi^i\cdot(X_i,\phi(e_i)\circ\pi) .
\end{equation}
From $\rho_L \circ \phi=\rho_A$, it follows that, for each $i$,
the pair $(X_i,\phi(e_i)\circ\pi)$ characterizes a section of
$\pi^! L\xto{\breve{\varpi}}A[1]$.
Using Equation~\eqref{eq:PEK} and the Leibniz rule, we obtain
\begin{align*}
\lie{s_{\boldphi}}{s_{\boldphi}}
& =\lie{\sum_i\xi^i(X_i,\phi(e_i)\circ\pi)}{\sum_j\xi^j(X_j,\phi(e_j)\circ\pi)} \\
&=\sum_{i,j}\xi^i\xi^j(0,\big(\lie{\phi(e_i)}{\phi(e_j)}
-\sum_k c^k_{ij}(x)\phi(e_k)\big)\circ\pi)
\end{align*}
The conclusion thus follows.
\end{proof}

As an immediate consequence, we have the following

\begin{theorem}
\label{Zug}
Let $A$ and $L$ be Lie algebroids over $\KK$ with the same base manifold $M$,
and let $\phi:A\to L$ be a Lie algebroid morphism.
Then the pull-back Lie algebroid
$\pi^! L\xto{\breve{\varpi}} A[1]$ is canonically
a dg Lie algebroid.
\end{theorem}

In particular, as an application, we have the following

\begin{corollary}
\label{Zug1}
For any given Lie pair $(L,A)$ over $\KK$, the pull-back Lie algebroid
$\pi^! L\xto{\breve{\varpi}} A[1]$ is canonically
a dg Lie algebroid.
\end{corollary}

\begin{remark}
According to Voronov~\cite{MR2971727}, \emph{double Lie algebroids}
in the sense of Mackenzie~\cite{arXiv:math/9808081,MR1650045} naturally
give rise to dg Lie algebroids.
There is such a double Lie algebroid picture
behind the dg Lie algebroid of Theorem~\ref{Zug}.
Indeed, in the case that $\KK$ is $\RR$,
any Lie algebroid morphism arises from a morphism of local Lie groupoids.
Given any morphism between Lie groupoids with the same base manifolds,
a classical construction due to Brown--Mackenzie \cite[Example~1.8]{MR1170713}
(see also \cite[Example~2.5]{MR1174393})
produces a double Lie groupoid, called \emph{comma double Lie groupoid}.
The infinitesimal counterpart of this comma double Lie groupoid
is a double Lie algebroid in the sense of Mackenzie~\cite{arXiv:math/9808081,MR1650045,MR4322148}
and the dg Lie algebroid to which it gives rise (in the sense of Voronov's theorem
\cite{MR2971727}) is precisely the dg Lie algebroid of Theorem~\ref{Zug}.
Thus, in the case of a morphism of real Lie algebroids,
we obtain a geometric interpretation of the dg Lie algebroid of Theorem~\ref{Zug},
and, in particular, of our main construction
of the $A_\infty$-algebra in Section~\ref{sec:A_infty_LP}
in terms of differentiable stacks.
See Appendix~\ref{Luca} for more details.
Note that, when $\KK$ is the field $\CC$ of complex numbers,
the double Lie groupoid no longer exists but the associated dg Lie algebroid
(or double Lie algebroid) still makes sense.
Finally we note that comma double Lie algebroids also appeared
in \cite[Theorem~2.16]{MR4421028}.
\end{remark}

\section{\texorpdfstring{$L_\infty$}{Loo}-algebroids from Lie pairs}
Recall that a \emph{Lie pair} is a pair $(L,A)$ consisting of a Lie algebroid $\varpi : L\to M$
and a Lie subalgebroid $A\subseteq L$ with the same base manifold $M$.
The algebra $C^\infty(M)$ of smooth functions on $M$ will be denoted by the symbol $\cirm$.

First of all, recall that, given a Lie pair $(L,A)$, the quotient vector bundle $L/A$
is canonically equipped with a flat $A$-connection customarily called \emph{Bott connection}
\cite{MR0362335,MR3439229}.
The Bott connection $\nabla^{\mathsf{Bott}}$ is defined by the relation
\[ \nabla^{\mathsf{Bott}}_e \big(q(v)\big) = q\big([e,v]\big) ,
\qquad \forall e \in \Gamma (A), \ v \in \Gamma (L) ,\]
and induces a differential on $\Gamma(\Lambda^\bullet A^\vee\otimes L/A)$
abusively denoted $\dA$ as it turns
$\left(\Gamma(\Lambda^\bullet A^\vee\otimes L/A),\dA\right)$ into a dg module
over the \dga\ $\left(\Gamma(\Lambda^\bullet A^\vee),\dA\right)$.
Here the symbol $q$ denotes the canonical projection $L\to L/A$.
Notice that the $\Gamma(\Lambda^\bullet A^\vee)$-module $\Gamma(\Lambda^\bullet A^\vee\otimes L/A)$
is the module of sections of the pull-back vector bundle \[ A[1] \times_M L/A\to A[1] ,\]
and the differential \[ \dA:\Gamma(\Lambda^\bullet A^\vee\otimes L/A)
\to\Gamma(\Lambda^{\bullet+1} A^\vee\otimes L/A) \] determines a dg vector bundle structure
on $A[1] \times_M L/A \to A[1]$.

\subsection{A contraction}

Consider the pull-back Lie algebroid
$\breve{\varpi} : \pi^! L\to A[1]$ as in Corollary~\ref{Zug1}.
The inclusion $i : A \to L$ is a Lie algebroid morphism.
So, according to Lemma~\ref{lem:s_phi},
\begin{equation}\label{Arago}
s_{\boldi}=(\dA,\boldi), \qquad\text{with}\qquad \boldi:A[1]\to A\xto{i}L
,\end{equation}
is a section of $\pi^! L\to A[1]$ satisfying $[s_{\boldi},s_{\boldi}]=0$.
Hence, according to Proposition~\ref{Prato}, $\breve{\varpi} : \pi^! L \to A[1]$
is a dg Lie algebroid with $\sQ=[s_{\boldi},\argument]$
as the operator of degree $+1$ on $\Gamma(A[1];\pi^! L)$.

In the sequel, we will denote the quotient $L/A$ simply by $B$ and
the projection $L\to L/A$ by $q$.
Once and for all, we pick a splitting $j : B \to L$ of the short exact sequence
\begin{equation}\label{eq:splitting1}
\begin{tikzcd}
0 \arrow[r] & A \arrow[r, "i"] & L \arrow[r, "q"] & B \arrow[r] & 0 .
\end{tikzcd}
\end{equation}
It determines another short exact sequence of vector bundles
\begin{equation}\label{eq:splitting2}
\begin{tikzcd}
0 & A \arrow[l] & L \arrow[l, "p"] & B \arrow[l, "j"] & 0 \arrow[l]
\end{tikzcd}
\end{equation}
in the usual way, i.e.\ $p : L \to A$ is the projection with kernel $j(B)$.

Next, we introduce three bundle maps $P_0$, $I_0$, and $H_0$.
While the operator $P_0$ is canonical,
the operators $I_0$ and $H_0$ depend on the chosen splitting $j$.

\subsubsection{The surjective map $P_0$}

It is convenient to identify $\Gamma(\Lambda^\bullet A^\vee\otimes B)$
with the module of sections of the dg vector bundle $A[1]\times_M B\to A[1]$.
In turn, sections of $A[1]\times_M B\to A[1]$ can be identified
with smooth maps $A[1] \to B$ covering the identity map $\id_M$.
We will adopt this identification in what follows.

We define a bundle map
\[ P_0: \pi^! L\to A[1]\times_M B \]
over $A[1]$ by the relation
\begin{equation}
P_0(X,\nu) := (\alpha, q \circ \nu ), \qquad\forall(X,\nu)\in\pi^! L|_\alpha
\cong T_{A[1]}|_\alpha \times_{T_M}L|_m, \ \text{where } \alpha\in A[1]|_m .
\end{equation}
It induces a surjection on the spaces of sections, denoted by
the same symbol by abuse of notations, 
\[ P_0 : \Gamma(\pi^! L) \to \Gamma(\Lambda^\bullet A^\vee\otimes B) .\]
More explicitly, expanding $\nu \in \Gamma( \pi^! L)$ as a linear combination
$\nu=\sum_k\xi_k\otimes v_k$ with $\xi_k\in\Gamma(\Lambda^\bullet A^\vee)$
and $v_k\in\Gamma (L)$, we obtain
\[ P_0(X,\nu) := \sum_k\xi_k\otimes q (v_k) .\]

\subsubsection{The injective map $I_0$}
\label{defIo}

Using the splitting \eqref{eq:splitting2}, one can define a map
\[ \Delta:\Gamma(B)\times\Gamma(A)\to\Gamma(A) \]
by the relation \[ \Delta_{b} a=p \lie{j(b)}{i(a)} \]
--- see \cite[Section~3.4.1]{MR3439229}.
This map $\Delta$ satisfies the identities
\begin{equation}
\label{eq:Prague}
\Delta_{f b}a=f \Delta_{b}a
\qquad\text{and}\qquad
\Delta_{b}(fa)=\rho_{j(b)}(f)\cdot a+f\Delta_b a ,
\end{equation}
for all $f\in\cirm$, $b\in\Gamma(B)$, and $a\in\Gamma(A)$.

Fixing the element $b$ of $\Gamma(B)$ and then dualizing, we obtain an operator
$\Delta_{b}:\Gamma(A^\vee)\to\Gamma(A^\vee)$, which we extend,
by the Leibniz rule, to a derivation $\Delta_{b}$ of the algebra
$\Gamma(\Lambda^\bullet A^\vee)$.
Thus $\Delta_{b}$ is a vector field of degree $0$ on $A[1]$.

Indeed, every section $b\in\Gamma(B)$ determines
a derivation of $C^\infty(A[1])$:
\[ C^\infty(A[1])\xfrom{i^\top}
C^\infty(L[1])\xfrom{\liederivative{j(b)}}
C^\infty(L[1])\xfrom{p^\top}C^\infty(A[1]) ,\]
i.e.\ a vector field on $A[1]$,
and a straightforward computation yields
$i^\top\circ\liederivative{j(b)}
\circ p^\top=\Delta_b$.
Here, $i^\top$ and $p^\top$ are the morphisms of algebras
induced by the maps $i:A\to L$ and $p:L\to A$.

By virtue of its construction, the vector field $\Delta_b$
is $\pi_*$-projectable and
\[ \pi_*(\Delta_b)=\rho_{j(b)} .\]
Hence, the pair $\big(\Delta_b,j(b)\circ\pi\big)$ characterizes a section
of $\pi^!L\to A[1]$ which is $\pi$-related to the section $j(b)$ of $L\to M$.

The map $\Gamma(B)\ni b\mapsto\Delta_{b}\in\XX(A[1])$
can be upgraded to a $\Gamma(\Lambda^\bullet A^\vee)$-linear map
\[ \Gamma(\Lambda^\bullet A^\vee\otimes B)\ni\beta\longmapsto
X_\beta\in\XX(A[1]) \]
as follows.
To an element $\beta=\sum_k\xi_k\otimes b_k$ of
$\Gamma(\Lambda^\bullet A^\vee)\otimes_R\Gamma(B)
\cong\Gamma(\Lambda^\bullet A^\vee\otimes B)$,
with $\xi_k\in\Gamma(\Lambda^\bullet A^\vee)$ and $b_k\in\Gamma(B)$,
we assign the derivation $X_\beta$ of $\Gamma(\Lambda^\bullet A^\vee)$ defined by
\begin{equation}\label{eq:SFO}
X_{\beta}(\eta)=\sum_k\xi_k\wedge\Delta_{b_k}(\eta),
\quad\forall\eta\in\Gamma(\Lambda^\bullet A^\vee).
\end{equation}
It is easy to check that $\pi_*\circ X_\beta=\rho\circ\nu_\beta$,
where $\nu_\beta$ is the composition
\[ \nu_\beta:A[1]\xto{\beta}B\xto{j}L .\]
Hence, the pair $(X_\beta,\nu_\beta)$ characterizes
a section of $\pi^!L\to A[1]$.

We define an injective bundle map

\begin{equation}
\label{eq:I0}
I_0 : A[1]\times_M B\to \pi^! L \cong T_{A[1]}\times_{T_M}L
\end{equation}
over $A[1]$ by the relation
\begin{equation}
\label{eq:I00}
I_0 : (\alpha, b)\to (\nabla_b|_\alpha, j (b), ),\ \ \forall (\alpha, b)
\in A[1]\times_M B
\end{equation}
Equation~\eqref{eq:Prague} implies that this is indeed a well-defined
bundle map. It induces an embedding, on the space of sections,
denoted by the same symbol by abuse of notations,
\[ I_0 : \Gamma(\Lambda^\bullet A^\vee\otimes B)\to\Gamma(\pi^! L) \]
by the relation
\[ I_0(\beta) := (X_\beta,\nu_\beta) .\]
We note that $\nu_\beta=\sum_i\xi_i\otimes j(b_i)
\in\Gamma(\Lambda^\bullet A^\vee\otimes L)$ if $\beta=\sum_i\xi_i\otimes b_i$.
It follows immediately from $q\circ j=\id_B$ that $P_0 I_0=\id$.

\subsubsection{The homotopy operator $H_0$}

We define a degree $(-1)$-bundle map
\[ H_0: \pi^! L\to \pi^! L \]
over $A[1]$ by the relation
\begin{equation} 
H_0(X,\nu) := \big(\iota_{p\circ\nu},0\big),\qquad \forall(X,\nu)\in
T_{A[1]}\times_{T_M}L ,
\end{equation}
where $X\in T_{A[1]}|_{\alpha}$, $\alpha\in A[1]|_m$ and $\nu\in L|_m$,
with $m\in M$,
satisfying the condition $\pi_* (X)=\rho_L (\nu)$.
Here $p : L \to A$ is the projection map, and
$\iota_{p\circ\nu}\in T_{A[1]}|_{\alpha}$
is the degree $(-1)$-tangent vector
corresponding to $p\circ\nu$.
On the level of spaces of sections, it induces an operator
$H_0 : \Gamma(\pi^! L) \to \Gamma(\pi^! L)[-1]$, which we
will describe in more details below.

Consider a section $v$ of $\pi : A \to M$. The contraction
$\iota_v : \Gamma (\Lambda^\bullet A^\vee) \to \Gamma(\Lambda^\bullet A^\vee)[-1]$
is a degree $-1$ derivation, i.e.~a degree $-1$ vector field on $A[1]$.
The correspondence $v \mapsto \iota_v$ is a $\cirm$-linear map
that we denote simply $\iota : \Gamma (A) \to \XX (A[1])[-1]$.
Extending it $\Gamma(\Lambda^\bullet A^\vee)$-linearly, we obtain a map
\[ \Gamma(\pi^\ast A) \cong C^\infty(A[1]) \otimes_{R}
\Gamma(A) \to \XX(A[1])[-1] \]
that we also denote $\iota$.
The latter construction has the following geometric interpretation:
the pull-back bundle $\pi^\ast (A[1]) = A[1] \times_M A[1] \to A[1]$ embeds into
the shifted tangent bundle $T_{A[1]}$ as the bundle of $\pi$-vertical tangent vectors in the usual way.
Applying the shift functor $[-1]$ fiberwise, we get an embedding
\[ \pi^\ast A = A[1] \times_M A \hookrightarrow T_{A[1]}[-1] .\]
Given a section $V \in \Gamma (\pi^\ast A)$, its image $\iota_{V} \in \XX (A[1])[-1]$
is nothing but the composition
\[ A[1] \xto{V} A[1] \times_M A \hookrightarrow T_{A[1]}[-1] .\]
of $V$ followed by the inclusion.

For any given section $(X,\nu)\in\Gamma(\pi^! L)$, the second
component $\nu$ is a smooth map $\nu : A[1] \to L$ covering the identity map $\id_M$.
Composing it with the projection $p : L \to A$, we get a map
\[ p\circ\nu : A[1]\to A \]
covering the identity map $\id_M$ or, equivalently, a section of the pull-back bundle
$\pi^\ast A\to A[1]$. In particular, we can consider the vector field
$\iota_{p\circ\nu} \in \XX (A[1])$.

Thus, on the spaces of sections, $H_0$ induces an operator, denoted by
the same symbol by abuse of notations:
\[ H_0 : \Gamma(\pi^! L) \to \Gamma(\pi^! L)[-1] \]
given by
\begin{equation}\label{eq:H00}
H_0(X,\nu) := \big(\iota_{p\circ\nu},0\big)
,\qquad \forall(X,\nu)\in\Gamma(\pi^! L).
\end{equation}
Since $\iota_{p\circ\nu}$ is a vector field
on $A[1]$ vertical relatively to the projection $\pi : A[1]\to M$,
the pair $\big(\iota_{p\circ\nu},0\big)$ characterizes a (null) section of
$\pi^! L \to A[1]$, which shows that $H_0$
is indeed well-defined.
For a characterizing pair of the special type
$(X,v\circ\pi)$, where $X\in\XX(A[1])$ is
$\pi_*$-projectable onto the vector field
$\rho(v)\in\XX(M)$, image of $v\in\Gamma(L)$
under the anchor map, we have
\begin{equation}\label{eq:H000}
H_0(X,v\circ\pi)=\big(\iota_{p(v)},0\big).
\end{equation}
For a generic characterizing pair
$(X,\nu)\in\Gamma(\pi^! L)$, expanding $\nu$
as a linear combination
$\nu=\sum_k\xi_k\otimes v_k$
with $\xi_k\in\Gamma(\Lambda^\bullet A^\vee)$ and $v_k\in\Gamma(L)$, we obtain
\begin{equation}\label{eq:H0000}
H_0(X,\nu)=
\left(\sum_k(-1)^{\degree{\xi_k}}\xi_k\cdot\iota_{p(v_k)},0 \right).
\end{equation}

\begin{theorem}\label{lem:contraction_0}
The triple of operators $(P_0, I_0, H_0)$, denoted by the
same symbols by abuse of notations, 
on the spaces of sections,
induced by the bundle maps $(P_0, I_0, H_0)$ 
constitutes $\Gamma(\Lambda^\bullet A^\vee)$-linear contraction data:
\begin{equation}\label{eq:contraction_0}
\begin{tikzcd}
\big(\Gamma(\pi^! L),\sQ\big)
\arrow[r, "P_0", shift left] \arrow["H_0", loop left]
& \big(\Gamma(\Lambda^\bullet A^\vee\otimes B),\dA\big)
\arrow[l, "I_0", shift left]
.\end{tikzcd}
\end{equation}
Here $\sQ=[s_{\boldi},\argument]$ is the coboundary operator
induced by the section $s_{\boldi}$ defined by Equation~\eqref{Arago},
and $\dA$ is the Chevalley--Eilenberg differential for the Bott representation
of the Lie algebroid $A$ on $B$.
\end{theorem}

\subsection{Proof of Theorem~\ref{lem:contraction_0}}

The present section is devoted to
the proof of Theorem~\ref{lem:contraction_0}.

The following lemma is well known.

\begin{lemma}[\cite{MR0220273}]
\label{lem:contraction}
Let $C^\bullet$ be a cochain complex with coboundary operator
$\delta:C^\bullet\to C[1]^\bullet$.
If there exists a linear operator $h:C^\bullet\to C[-1]^\bullet$
satisfying
\[ h^2=0 \qquad\text{and}\qquad h\delta h=h ,\]
then
\begin{enumerate}
\item the operators $h\delta$, $\delta h$, and $[h,\delta]=h\delta+\delta h$
are projection operators;
\item $K^\bullet:=\im(\id-[h,\delta])=\ker([h,\delta])
=\ker(h\delta)\cap\ker(\delta h)$ is a subcomplex of $(C^\bullet,\delta)$;
\item the graded module $C^\bullet$ decomposes as the direct sum
$C^\bullet=K^\bullet\oplus\im(h\delta)\oplus\im(\delta h)$;
\item and we have a contraction
\[ \begin{tikzcd}
(C^\bullet,\delta) \arrow[r, "p", shift left] \arrow["h", loop left] &
(K^\bullet,\delta) \arrow[l, "i", shift left]
\end{tikzcd} ,\]
where $p$ is the projection of $C^\bullet$ onto its direct summand $K^\bullet$
and $i$ is the inclusion of $K^\bullet$ into $C^\bullet$
so that $pi=\id$ and $ip=\id-[h,\delta]$.
\end{enumerate}
\end{lemma}

\begin{lemma}\label{Leman}
Consider the cochain complex $\big(\Gamma(\pi^! L),\sQ\Big)$
with the coboundary operator $\sQ=[s_{\boldi},\argument]$ introduced earlier
--- see~Equation~\eqref{Arago}.
The operator $H_0$ defined by Equation~\eqref{eq:H00} satisfies
\[ H_0^2=0 \qquad\text{and}\qquad H_0\sQ H_0=H_0 .\]
\end{lemma}

\begin{proof}
It is obvious that $H^2_0=0$.
Unlike $I_0$, $P_0$, and $H_0$,
the operator $\sQ$ is \emph{not}
$C^\infty(A[1])$-linear: we have,
for all $f\in C^\infty(A[1])$ and $s\in\Gamma(\pi^!L)$,
\begin{equation}
\label{eq:H0}
\sQ(f\cdot s)=(\dA f)\cdot s + (-1)^{\degree{f}} f\cdot\sQ(s). 
\end{equation}

Hence
\begin{align*}
(H_0\sQ H_0) (f\cdot s)
&= H_0 \big( (\dA f)\cdot H_0(s) + (-1)^{\degree{f}} f\cdot\sQ(H_0(s))\big) \\
&= (\dA f)\cdot H_0^2 (s)+H_0\big((-1)^{\degree{f}} f\cdot\sQ(H_0(s)))\big) \\
&= f\cdot (H_0\sQ H_0) (s)
.\end{align*}
Therefore it follows
that the operator $H_0\sQ H_0-H_0$ is $\Gamma(\Lambda^\bullet A^\vee)$-linear.
Thus, to prove the identity $H_0\sQ H_0=H_0$,
it suffices to show that the operator $H_0\sQ H_0-H_0$
annihilates every characteristic pair of the special type $(X,v\circ\pi)$,
where the vector field $X\in\XX(A[1])$ is $\pi_*$-projectable
onto the vector field $\rho(v)\in\XX(M)$.
Now we have
\begin{align*}
(H_0\sQ H_0) (X,v\circ\pi)
& = H_0 \lie{(\dA,\boldi)}{(\iota_{p(v)},0)}\\
& = H_0 \big(\lie{\dA}{\iota_{p(v)}}, \ \iota_{p(v)}(\boldi)\big)\\
& = H_0 \big(\lie{\dA}{\iota_{p(v)}}, \ i p(v) \circ\pi\big)\\
& = \left( \iota_{p (v)}, 0 \right)\\
& = H_0(X,v\circ\pi).
\end{align*}
This concludes the proof.
\end{proof}

It follows from Lemmas~\ref{lem:contraction} and~\ref{Leman}
that \[ K^\bullet:=\im(\id-[H_0,\sQ])
=\ker(H_0\sQ)\cap\ker(\sQ H_0) \]
is a subcomplex of $\big(\Gamma(\pi^! L),\sQ\Big)$.

Next, we will identify this subcomplex $K^\bullet$ with
the cochain complex
$\big(\Gamma(\Lambda^\bullet A^\vee\otimes B),d_A\big)$.

\begin{lemma}
\label{lem:ROM}
Given any $b\in\Gamma(B)$, there exists a unique vector field
$X\in\XX(A[1])$, which is $\pi_*$-projectable onto the vector field
$\rho_{j(b)}\in\XX(M)$ and satisfies the condition:
\[ H_0\sQ\big(X,j(b)\circ\pi\big)=0 .\]
That is, $X=\Delta_b$.
\end{lemma}

\begin{proof}
According to Equation~\eqref{eq:BRU}, we have
\[ \sQ\big(X,\jb\circ\pi\big)
=\lie{(d_A,s_{\boldi})}{(X,\jb\circ\pi)}
=\big(\lie{\dA}{X}, \ \lie{\boldi}{\jb\circ\pi}-(-1)^{\degree{X}}X(\boldi)\big) .\]
Thus $H_0\sQ(X,j(b)\circ\pi)=0$ if and only if
\begin{equation}\label{eq:ORD}
p\big(\lie{\boldi}{j(b)\circ\pi}-(-1)^{\degree{X}}X(\boldi)\big)=0.
\end{equation}
Now, according to~\eqref{eq:PEK12}, we have
\[ \lie{\boldi}{j(b)\circ\pi} - (-1)^{\degree{X}} X(\boldi)
=\sum_k \varepsilon^k \otimes \lie{i(e_k)}{j(b)}
-\sum_k (-1)^{\degree{X}} X(\varepsilon^k) \otimes i(e_k) ,\]
where $e_1,e_2,\dots$ and $\varepsilon^1,\varepsilon^2,\dots$ are any pair
of dual local frames for $A$ and $A^\vee$ respectively.

Therefore, Equation~\eqref{eq:ORD} is equivalent to
\[ -\sum_k \varepsilon^k \otimes \Delta_b e_k
-(-1)^{\degree{X}} \sum_k X(\varepsilon^k)\otimes e_k=0 .\]
Pairing with $\varepsilon^l$, we deduce that
\[ (-1)^{\degree{X}}X(\varepsilon^l)
=-\sum_k \varepsilon^k \langle \Delta_b e_k \mathbin{|} \varepsilon^l \rangle
=\sum_k \langle \Delta_b \varepsilon^l \mathbin{|} e_k \rangle \varepsilon^k
= \Delta_b \varepsilon^l .\]
Hence it follows that $\degree{X}=0$ and $X(\varepsilon^l)=\Delta_b\varepsilon^l$,
for every $l$. Since $\pi_*(X)=\rho_{j(b)}$,
we can conclude that $X=\Delta_{b}$.
\end{proof}

\begin{corollary}
\label{cor:NAP}
Given any section $\beta\in\Gamma(\Lambda^\bullet A^\vee\otimes B)$,
consider the induced map \[ \nu_\beta:A[1]\xto{\beta}B\xto{j}L .\]
There exists a unique vector field $X\in\XX(A[1])$ satisfying
\[ \pi_*\circ X=\rho\circ\nu_\beta \qquad\text{and}\qquad
H_0\sQ(X,\nu_\beta)=0 ,\]
namely the vector field $X_\beta$ defined by Equation~\eqref{eq:SFO}.
\end{corollary}

\begin{proof}
Without loss of generality, we may assume that $\beta =\xi\otimes b$
with $\xi\in\Gamma(\Lambda^\bullet A^\vee)$ and $b\in\Gamma(B)$
--- note that $\Gamma(\Lambda^\bullet A^\vee\otimes B)
\cong\Gamma(\Lambda^\bullet A^\vee)\otimes_R\Gamma(B)$.
Then, according to Equation~\eqref{eq:SFO}, we have
$X_{\beta}=\xi\cdot\Delta_{b}$ and thence
$(X_\beta,\nu_\beta)=\xi\cdot(\Delta_{b},j(b)\circ\pi)$.
It follows from Equation~\eqref{eq:H000} and Lemma~\ref{lem:ROM} that
\begin{multline*} H_0\sQ(X_\beta,\nu_\beta)
=H_0\sQ\big(\xi\cdot(\Delta_{b},\jb\circ\pi)\big) \\
=H_0\big((\dA\xi)\cdot(\Delta_{b},\jb\circ\pi)
+(-1)^{\degree{\xi}}\xi\cdot\sQ(\Delta_{b},\jb\circ\pi)\big) \\
=(-1)^{\degree{\xi}+1}(\dA\xi)\cdot H_0(\Delta_{b},\jb\circ\pi)
+\xi\cdot H_0\sQ(\Delta_{b},\jb\circ\pi)=0
.\end{multline*}

Conversely, assume that $H_0\sQ(X,\nu_\beta)=0$.
It follows that
\[ H_0\sQ(X-X_\beta,0)=H_0\sQ(X,\nu_\beta)
-H_0\sQ(X_\beta,\nu_\beta)=0 .\]
According to Lemma~\ref{lem:ROM},
we have $X-X_\beta=\Delta_{0}=0$.
Thus $X=X_\beta$.
\end{proof}

Note that $H_0(X_\beta,\nu_\beta)=0$ for all $\beta\in\Gamma(\Lambda^\bullet
A^\vee\otimes B)$ since $p\circ j=0$. It follows that
\[ (X_\beta,\nu_\beta)\in\ker(\sQ H_0)\cap\ker(H_0\sQ)=K^\bullet .\]

\begin{proposition}
\label{cor:NAPbis}
The map
\[ \begin{tikzcd}
\big(\Gamma(\Lambda^\bullet A^\vee\otimes B),\dA\big) \arrow[r, "I_0"]
& \big(\Gamma(\pi^! L),\sQ\big)
\end{tikzcd} \]
is a cochain map.
\end{proposition}

\begin{proof}
We need to show that $\sQ\circ I_0=I_0\circ\dA$.

Given any $b\in\Gamma(B)$,
Equation~\eqref{eq:BRU} yields
\begin{multline*}
\sQ I_0(b)=\sQ\big(\Delta_b,j(b)\circ\pi\big)
=\lie{(\dA,\boldi)}{(\Delta_b,j(b)\circ\pi)} \\
=\big(\lie{\dA}{\Delta_b},\dA(1)\otimes j(b)
-\sum_k\Delta_b(\varepsilon^k)\otimes i(e_k)
+\sum_k\varepsilon^k\otimes\lie{i(e_k)}{j(b)}\big)
\end{multline*}
since $\boldi=(A[1]\to A\xto{i} L)=\sum_k\varepsilon^k\otimes i(e_k)$,
where $e_1,e_2,\dots$ and $\varepsilon^1,\varepsilon^2,\dots$ are any pair of
dual local frames for $A$ and $A^\vee$ respectively.
A straightforward computation gives
\begin{multline*}
\dA(1)\otimes j(b)-\sum_k\Delta_b(\varepsilon^k)\otimes i(e_k)
+\sum_k\varepsilon^k\otimes\lie{i(e_k)}{j(b)} \\
=-\sum_k\Delta_b(\varepsilon^k)\otimes i(e_k)
+\sum_k\varepsilon^k\otimes
\big(-i(\Delta_b e_k)+j(\nabla^{\Bott}_{e_k} b)\big)
=\sum_k\varepsilon^k\otimes j(\nabla^{\Bott}_{e_k} b)
=\nu_{\dA(b)}
.\end{multline*}
Hence, we obtain
\[ \sQ I_0(b)=\big(\lie{\dA}{\Delta_b},\nu_{\dA(b)}\big) .\]
Since $\sQ^2=0$, it follows that
\[ H_0\sQ\big(\lie{\dA}{\Delta_b},\nu_{\dA(b)}\big)
=H_0\sQ\sQ I_0(b)=0 .\]
According to Corollary~\ref{cor:NAP}, we can therefore conclude that
$\lie{\dA}{\Delta_b}=X_{\dA(b)}$ and thence that
\[ \sQ I_0(b)=\big(X_{\dA(b)},\nu_{\dA(b)}\big)=I_0\dA(b) .\]

Consequently, for any $\alpha\otimes b\in\Gamma(\Lambda^\bullet A^\vee\otimes B)$
with $\alpha\in\Gamma(\Lambda^\bullet A^\vee)$ and $b\in\Gamma(B)$, we obtain
\begin{multline*}
\sQ I_0(\alpha\otimes b)=\sQ\big(\alpha\cdot I_0(b)\big)
=(\dA\alpha)\cdot I_0(b)+(-1)^{\degree{\alpha}}\alpha\cdot\sQ I_0(b) \\
=(\dA\alpha)\cdot I_0(b)+(-1)^{\degree{\alpha}}\alpha\cdot I_0\big(\dA(b)\big)
=I_0\big( \dA(\alpha)\cdot b + (-1)^{\degree{\alpha}} \alpha\cdot\dA(b) \big)
=I_0\dA(\alpha\otimes b)
\end{multline*}
since the map $I_0$ is $\Gamma(\Lambda^\bullet A^\vee)$-linear.

This concludes the proof.
\end{proof}

\begin{proof}[Proof of Theorem~\ref{lem:contraction_0}]
First, we will show that the cochain map $I_0$ identifies
$\big(\Gamma(\Lambda^\bullet A^\vee\otimes B),\dA\big)$
to the subcomplex \[ K^\bullet=\ker(H_0\sQ)\cap\ker(\sQ H_0)
=\im\big(\id-\lie{H_0}{\sQ}\big) \]
of the cochain complex $\big(\Gamma(\pi^! L),\sQ\big)$.
Since $I_0$ is injective, it suffices to show that
$\im(I_0)=K^\bullet$.

Since $p\circ j=0$, we have $H_0 I_0=0$ and thence
$\im I_0\subset\ker(\sQ H_0)$.
Furthermore, since $I_0$ is a cochain map, we have
$H_0\sQ I_0=H_0 I_0\dA=0$ and thence
$\im I_0\subset\ker(H_0\sQ)$.
Thus, we obtain $\im(I_0)\subset K^\bullet$.

Let $\breve{\pi}$ denote the map
\[ \pi^! L\cong T_{A[1]}\times_{T_M}L\ni(X,\nu)\longmapsto\nu\in L .\]
Obviously, we have $\breve{\pi}\circ H_0=0$.
Consequently, for all $(X,\nu)\in\Gamma(\pi^! L)$, we have
\[ \begin{split}
\breve{\pi}\circ\big(\id-\lie{H_0}{\sQ}\big)(X,\nu)
&= \breve{\pi}\big((X,\nu)-H_0\sQ(X,\nu)-\sQ H_0(X,\nu)\big) \\
&= \nu-\breve{\pi}\big(\lie{(\dA,\boldi)}{(\iota_{p\circ\nu},0)}\big) \\
&= \nu-\breve{\pi}\big(\lie{\dA}{\iota_{p\circ\nu}},\iota_{p\circ\nu}(\boldi)\big) \\
&= \nu-\iota_{p\circ\nu}(\boldi) \\
&= \nu-i\circ p\circ\nu \\
&= j\circ q\circ\nu
.\end{split} \]
Therefore, we have
\[ \big(\id-\lie{H_0}{\sQ}\big)(X,\nu)=(Y,j\circ q\circ\nu) \]
for some vector field $Y\in\XX(A[1])$
such that $\pi_*\circ Y=\rho\circ j\circ q_b\circ\nu$.
Since \[ \im(\id-\lie{H_0}{\sQ})
=\ker(H_0\sQ)\cap\ker(\sQ H_0)
\subset\ker(H_0\sQ) ,\]
it follows from Corollary~\ref{cor:NAP} that $Y=X_{q\circ\nu}$.
Thus, we obtain
\begin{equation}\label{eq:entretien}
\big(\id-\lie{H_0}{\sQ}\big)(X,\nu)
=(X_{q\circ\nu},j\circ q\circ\nu)=I_0(q\circ\nu)
,\end{equation}
which shows that
\[ K^\bullet=\im(\id-\lie{H_0}{\sQ})\subset\im(I_0) .\]

From Proposition~\ref{cor:NAPbis}, we can now conclude
that $I_0$ does indeed identify the cochain complex
$\big(\Gamma(\Lambda^\bullet A^\vee \otimes B),\dA\big)$
to the subcomplex $K^\bullet$ of
$\big(\Gamma(\pi^! L),\sQ\big)$ isomorphically.

It is immediate that $P_0 I_0=\id$ and it follows from Equation~\eqref{eq:entretien}
that, for all $(X,\nu)\in\Gamma(\pi^! L)$,
\[ \big(\id-\lie{H_0}{\sQ}\big)(X,\nu)
=I_0(q\circ\nu)=I_0 P_0(X,\nu) .\]
Hence $\id-\lie{H_0}{\sQ}=I_0 P_0$.
The conclusion thus follows from Lemma~\ref{lem:contraction}.
\end{proof}

\begin{remark}\label{flageolet}
Every section $a\in\Gamma(A)$ determines
a pair of vector fields on $A[1]$: namely the
Lie derivative $\liederivative{a}$
and the interior product $\interior{a}$.
It is easily seen that the pairs
$(\liederivative{a},i\circ a\circ\pi)$
and $(\interior{a},0)$
characterize a pair of sections of
$\pi^!L\to A[1]$.

On the other hand, every section
$b\in\Gamma(B)$ determines a vector field
on $A[1]$: namely the derivation
$\Delta_b=i^\top\circ\liederivative{j(b)}
\circ p^\top$ of the algebra
$C^\infty(A[1])=\Gamma(\Lambda^\bullet A^\vee)$
--- see Section~\ref{defIo}.
It is easily seen that the pair
$(\Delta_b,j\circ b\circ\pi)$
characterizes a section of $\pi^!L\to A[1]$.

It turns out that every section of $\pi^!L\to A[1]$
can be written (locally) as a linear combination with 
coefficients in $C^\infty(A[1])$ of sections
of the three aforementioned types:
\[ (\liederivative{a},i\circ a\circ\pi),
\qquad (\interior{a},0), \qquad\text{and}\qquad
(\Delta_b,j\circ b\circ\pi) .\]

For instance, the canonical section $s_{\boldi}=(\dA,\boldi)$
of $\pi^!L\to A[1]$ determined by the morphism of Lie
algebroids $i:A\to L$ arises locally as the linear
combination with coefficients in $C^\infty(A[1])$
\begin{equation}\label{eq:Darwin}
s_{\boldi} = (\dA,\boldi) = \sum_k \left\{
\xi^k\cdot(\liederivative{e_k},i\circ e_k\circ\pi)
-(\dA\xi^k)\cdot(\interior{e_k},0) \right\}
\end{equation}
for any pair $\{e_k\}$ and $\{\xi_k\}$
of dual local frames for $A$ and $A^\vee$
respectively --- see Equation~\eqref{eq:italo}
in Appendix~\ref{trenitalia}.

Furthermore, we note that, for all $f\in\cirm$,
$a\in\Gamma(A)$, and $b\in\Gamma(B)$, we have
\[ (\liederivative{f\cdot a},i\circ f\cdot a\circ\pi)=
\pi^*(f)\cdot(\liederivative{a},i\circ a\circ\pi)
+\dA f\cdot(\interior{a},0) \]
and
\[ (\Delta_{f\cdot b},j\circ f\cdot b\circ\pi)
=\pi^*(f)\cdot(\Delta_b,j\circ b\circ\pi) .\]

The following nine relations are immediate consequences
of the definitions of the operators $P_0$, $I_0$, and $H_0$:
\[ \begin{gathered}
P_0(\Delta_b,j\circ b\circ\pi)=1\otimes b \\
P_0(\liederivative{a},i\circ a\circ\pi)=0 \\
P_0(\interior{a},0)=0 \\
P_0(\dA,\boldi)=0
\end{gathered}
\qquad\quad
I_0(1\otimes b)=(\Delta_b,j\circ b\circ\pi)
\qquad\quad
\begin{gathered}
H_0(\Delta_b,j\circ b\circ\pi)=(0,0) \\
H_0(\liederivative{a},i\circ a\circ\pi)=(\interior{a},0) \\
H_0(\interior{a},0)=(0,0) \\
H_0(\dA,\boldi)=\sum_k\xi^k
\cdot(\interior{e_k},0)
\end{gathered} \]

The operator $\sQ$ satisfies the following four relations:
\begin{gather*}
\sQ(\Delta_b,j\circ b\circ\pi)=
\sum_k\xi^k\cdot
\big(\Delta_{\nabla^{\Bott}_{e_k}b},j\circ
(\nabla^{\Bott}_{e_k}b)\circ\pi\big) \\
\sQ(\liederivative{a},i\circ a\circ\pi)=(0,0) \\
\sQ(\interior{a},0)=
(\liederivative{a},i\circ a\circ\pi) \\
\sQ(\dA,\boldi)=(0,0)
\end{gather*}

It follows immediately from the relations above
that, in the direct sum decomposition
\[ C^\bullet=K^\bullet\oplus
\im(H_0\sQ)\oplus
\im(\sQ H_0) ,\]
\begin{enumerate}
\item the direct summand
$K^\bullet=\ker(H_0\sQ)\cap\ker(\sQ H_0)$
is the $C^\infty(A[1])$-submodule
of $C^\bullet=\Gamma(\pi^! L)$ generated by
the sections of type $(\Delta_b,j\circ b\circ\pi)$
for all $b\in\Gamma(B)$;
\item the direct summand $\im(H_0\sQ)$
is comprised of all linear combinations
\[ \sum_k \phi_k\cdot\big(\liederivative{a_k},
i\circ a_k\circ\pi\big)+
\sum_l \psi_l\cdot(\interior{a_l},0) \]
where the coefficients $\phi_k$ are homogeneous
functions on $A[1]$ of odd degree
and the coefficients $\psi_l$ are homogeneous
functions on $A[1]$ of even degree; and
\item the direct summand $\im(\sQ H_0)$
is comprised of all linear combinations
\[ \sum_k \phi_k\cdot\big(\liederivative{a_k},
i\circ a_k\circ\pi\big)+
\sum_l \psi_l\cdot(\interior{a_l},0) \]
where the coefficients $\phi_k$ are homogeneous
functions on $A[1]$ of even degree
and the coefficients $\psi_l$ are homogeneous
functions on $A[1]$ of odd degree.
\end{enumerate}
\end{remark}

\subsection{Homotopy transfer of \texorpdfstring{$L_\infty$}{Loo}-algebroids}

It was proved in~\cite[Proposition~4.1]{MR4091493} that a splitting
of the short exact sequence \eqref{eq:splitting1} determines
an $L_\infty$-algebroid structure on the graded vector bundle
$\cL: =A[1] \times_M L/A \to A[1]$. 
Its multi-anchor maps and multi-brackets can be described explicitly
as follows \cite[Lemma~4.5]{MR4091493}:
\begin{enumerate}
\item The nullary anchor $\rho_0$ is the Chevalley--Eilenberg differential
$\dA:\Gamma(\Lambda^{\bullet}A^\vee)\to\Gamma(\Lambda^{\bullet+1}A^\vee)$
of the Lie algebroid $A$.
\item The unary anchor $\rho_1: \cL\to T_{\LADA[1]}$ is determined by
\[ \rho_1(b)=\Delta_b,\quad \forall b\in \sections{B} .\]
\item The binary anchor $\rho_2: \Lambda^2\cL\to T_{\LADA[1]}$ is determined by
\[ \rho_2(b_1,b_2)= \iota_{p\baL{b_1}{b_2}} ,\qquad \forall b_1,b_2\in \sections{B} .\]
\item The unary bracket $l_1$ coincides with the Chevalley--Eilenberg differential
\[ \dAB:\Gamma(\Lambda^{\bullet}A\otimes B)\to\Gamma(\Lambda^{\bullet+1}A\otimes B) .\]
\item The binary bracket $l_2: \Lambda^2\sections{\cL}\to \sections{\cL}$ is determined by
\[ l_2(b_1,b_2)= \quotientmapLB\baL{ {b_1}}{ {b_2}}
\qquad \forall b_1,b_2\in \sections{B} .\]
\item The ternary bracket $l_3: \Lambda^3\sections{\cL}\to \sections{\cL}$ is determined by
\[ l_3(b_1,b_2,b_3)=0,\qquad \forall b_1,b_2,b_3\in \sections{B} .\]
\item All higher brackets $l_i$ ($i\geqslant 4$), and anchors $\rho_{j}$ ($j\geqslant 3$) vanish.
\end{enumerate}

As a preparation for next section, we prove here that this 
$L_\infty$-algebroid is actually induced by the dg Lie 
algebroid structure on $\pi^! L \to A[1]$
via \emph{homotopy transfer of dg Lie algebroids}.
In order to do this, we begin explaining
how an $L_\infty$-algebroid structure can be transferred 
along appropriate contraction data.
We will use an appropriate tensor trick
for $L_\infty$-algebroids, combined with the 
\emph{homological perturbation lemma}.
To the best of our knowledge, this result is new
and should be of independent interest.

We first introduce the following

\begin{definition}
Let $\mathcal{L}\to\mathcal{M}$ and $\mathcal{B}\to\mathcal{M}$
be dg vector bundles over a dg manifold
$(\mathcal{M},d)$.
A contraction of dg vector bundles is a triple
$(P, I, H)$ of dg vector bundle maps covering the identity of $\mathcal{M}$
such that its induced maps on sections, also denoted $(P, I, H)$ by abuse of notations,
is a contraction of cochain complexes
\begin{equation}\label{eq:contr_dg_vb}
\begin{tikzcd}
(\mathcal{L},d_\mathcal{L}) \arrow[r, "P", shift left] \arrow["H", loop left] &
(\mathcal{B},d_\mathcal{B}). \arrow[l, "I", shift left]
\end{tikzcd} 
\end{equation}
\end{definition}

\begin{theorem}[Homotopy transfer for $L_\infty$-algebroids]
\label{theor:ht_L_infty_alg}
Let $\mathcal{L}\to\mathcal{M}$
be an $L_\infty$-algebroid with finite multi-anchor maps and
multi-brackets, and $\mathcal{B}\to\mathcal{M}$
be a dg vector bundle.
A contraction of dg vector bundles \eqref{eq:contr_dg_vb}
induces an $L_\infty$-algebroid structure on $\mathcal{B}$
extending the dg vector bundle structure.
\end{theorem} 

Before providing the proof of Theorem~\ref{theor:ht_L_infty_alg},
we recall an important lemma, called \emph{the tensor trick},
which will also play an important role in the next section.

\begin{lemma}[{\cite{MR4485797} and \cite[Lemma~A.7]{arXiv:2103.08096}}]
\label{lem:tensor-trick}
Let $(\cR,d)$ be a \dga.
Given a contraction of dg $\cR$-modules
\[ \begin{tikzcd}
(M,d_M) \arrow[r, "P", shift left] \arrow["H", loop left] &
(N,d_N), \arrow[l, "I", shift left]
\end{tikzcd} \]
there exists an induced contraction on the symmetric tensor modules:
\[ \begin{tikzcd}
(S_{\cR}M,d_M) \arrow[r, "P_S", shift left] \arrow["H_S", loop left] &
(S_{\cR}N,d_N), \arrow[l, "I_S", shift left]
\end{tikzcd}\]
where
\begin{align}
P_S &= \bigoplus_{n\geq 0} P^{\odot n} ,&
I_S &= \bigoplus_{n\geq 0} I^{\odot n} ,&
H_S &= \bigoplus_{n\geq 0} \sum_{i=1}^{n}
(I\circ P)^{\odot (i-1)}\odot H\odot\id_M^{\odot (n-i)}. \label{eq:sym}
\end{align}
Here it is understood that
$(S^0_\cR M,d_M) \cong (S^0_\cR N,d_N) \cong (\cR,d)$.
\end{lemma}

\begin{remark}\label{rem:H_S_der}
Let $M, N, P, I, H$ be as in Lemma~\ref{lem:tensor-trick}.
First of all, note that $P_S, I_S$ are morphisms of cdgas.
Now, put $\mathcal{K}= \ker P $ and $\mathcal{I}= \operatorname{im} I$,
so that $M = K \oplus I$, and
$S_{\cR} M \cong S_\cR \mathcal{K}\otimes_\cR S_\cR \mathcal{I}$.
It is easy to see, using the explicit formulas \eqref{eq:sym}, that
\begin{enumerate}
\item $H_S$ is $S_\cR \mathcal{I}$-linear, and
\item for all $k_1, \cdots, k_p \in \mathcal{K}$,
\[ H_S (k_1 \odot \cdots \odot k_p) = \frac{1}{p} \sum_{i = 1}^p (-1)^{|k_1| + \cdots + |k_{i-1}|} k_1 \odot \cdots \odot H (k_i) \odot \cdots \odot k_p .\]
\end{enumerate}
In other words, on elements of $S_\cR^p \mathcal{K} \otimes S_\cR \mathcal{I} \subseteq S_\cR M$, the homotopy $H_S$ acts as $1/p$ times the unique $S_\cR\mathcal{I}$-linear derivation of $S_\cR M$ extending $H$. This simple remark will be useful later.
\end{remark}

\begin{proof}[Proof of Theorem~\ref{theor:ht_L_infty_alg}]
Consider the dual dg vector bundles $(\mathcal{L}^\vee, d_{\mathcal{L}^\vee}), (\mathcal{B}^\vee, d_{\mathcal{B}^\vee})$ of $(\mathcal{L}, d_\mathcal{L}), (\mathcal{B}, d_\mathcal{B})$, respectively. 
The contraction \eqref{eq:contr_dg_vb} induces a \emph{transpose contraction} of dg vector bundles
\[ \begin{tikzcd}
(\mathcal{L}^\vee[-1],d_{\mathcal{L}^\vee}) \arrow[r, "I^\vee", shift left]
\arrow["H^\vee", loop left] &
(\mathcal{B}^\vee[-1],d_{\mathcal{B}^\vee}). \arrow[l, "P^\vee", shift left]
\end{tikzcd} \]

From Lemma~\ref{lem:tensor-trick},
 we obtain a contraction
\[ \begin{tikzcd}
\big(\Gamma (S (\mathcal{L}^\vee[-1]),d_{\mathcal{L}^\vee}\big) \arrow[r, "I^\vee_S", shift left] \arrow["H^\vee_S", loop left] &
\big(\Gamma (S (\mathcal{B}^\vee[-1]),d_{\mathcal{B}^\vee}\big). \arrow[l, "P^\vee_S", shift left]
\end{tikzcd} \]

It is well-known \cite[Proposition~A.1]{MR4091493} 
that a $\ZZ$-vector bundle $\cL\to\cM$
is an $L_\infty$-algebroid if and only if $\cL[1]$ is a dg manifold
whose homological vector field $D_\mathcal{L}$ is tangent to the zero section
$\cM\xto{0}\cL[1]$.
When $\mathcal{L}\to\mathcal{M}$
an $L_\infty$-algebroid with finite multi-anchor maps and
multi-brackets,
$\Gamma(S(\mathcal{L}^\vee[-1]))$ is a differential graded subalgebra
of $(C^\infty(\mathcal{L}[1]),D_\mathcal{L})$. That is, it is preserved
by the homological vector field $D_\mathcal{L}$. We denote by the same symbol
$D_\mathcal{L}:\Gamma(S(\mathcal{L}^\vee[-1]))
\to\Gamma(S(\mathcal{L}^\vee[-1]))$ the differential
induced by the restriction and interpret it
as a perturbation of $d_{\mathcal{L}^\vee}$.
The operator
$H_S^\vee\circ(D_\mathcal{L}-d_{\mathcal{L}^\vee})$ is locally nilpotent.
Indeed, from Remark~\ref{rem:H_S_der}, $H_S^\vee$, hence
$H_S^\vee\circ(D_\mathcal{L}-d_\mathcal{L^\vee})$,
lowers by $1$ the order with respect to the filtration
\[ \cdots\subseteq\bigoplus_{i \leq p} S^i \mathcal{K} \otimes S \mathcal{I}
\subseteq \bigoplus_{i \leq p+1} S^i \mathcal{K} \otimes S \mathcal{I}
\subseteq \cdots \subseteq \Gamma (S \mathcal{L}^\vee[-1]) .\]
It follows that the classical homological perturbation lemma is applicable and we get a new contraction diagram:
\[ \begin{tikzcd}
\big(\Gamma (S \mathcal{L}^\vee[-1]),D_\mathcal{L}\big)
\arrow[r, "I^\mathsf{hp}_S", shift left] \arrow["H^\mathsf{hp}_S", loop left] &
\big(\Gamma (S \mathcal{B}^\vee[-1]),d_\mathcal{B}^\mathsf{hp}\big).
\arrow[l, "P^\mathsf{hp}_S", shift left]
\end{tikzcd} \]
We claim that the perturbation $d_\mathcal{B}^\mathsf{hp}$
is a derivation of $\Gamma (S (\mathcal{B}^\vee[-1]))$.
To see this, note that $d_\mathcal{B}^\mathsf{hp}$
is actually given by the following formula:
\begin{equation}\label{eq:d_B^hp}
d_\mathcal{B}^\mathsf{hp} = d_{\mathcal{B}^\vee} + I_S^\vee\circ
\sum_{i=0}^\infty \big((d_{\mathcal{L}^\vee} - D_\mathcal{L}) \circ H_S^\vee \big)^k
\circ (d_{\mathcal{L}^\vee} - D_\mathcal{L}) \circ P_S^\vee
.\end{equation}
Hence it suffices to show that
\begin{equation}\label{eq:d_hp_k}
I_S^\vee \circ \big((d_{\mathcal{L}^\vee} - D_\mathcal{L}) \circ H_S^\vee \big)^k
\circ (d_{\mathcal{L}^\vee} - D_\mathcal{L}) \circ P_S^\vee
\end{equation}
is a derivation for all $k=0,1,\dots$. For every $k$, let 
\[
\Delta_k := \big((d_{\mathcal{L}^\vee} - D_\mathcal{L}) \circ H_S^\vee \big)^k
\circ (d_{\mathcal{L}^\vee} - D_\mathcal{L}) \circ P_S^\vee .
\]
We prove, by induction on $k$, that, for any
$Q_1,Q_2\in\Gamma(S(\mathcal{L}^\vee[-1]))$,
\begin{equation}\label{eq:Delta_k}
\Delta_k (Q_1 \odot Q_2) = \Delta_k (Q_1) \odot P_S^\vee(Q_2) + (-1)^{|Q_1|} P_S^\vee (Q_1) \odot \Delta_k (Q_2) + \Sigma ,
\end{equation}
where $\Sigma \in \Gamma ((S \mathcal{L}^\vee[-1]))$ is a sum of terms of the form 
\begin{equation}\label{eq:sum_terms}
H_S^\vee (Q) \odot Q', \quad Q,Q' \in \Gamma (S \mathcal{L}^\vee[-1]).
\end{equation}
For $k = 0$, Equation~\eqref{eq:Delta_k} is clearly satisfied (with $\Sigma = 0$).
Now, assume that \eqref{eq:Delta_k} is satisfied for $k = n$,
and compute $\Delta_{n+1} (Q_1 \odot Q_2)$. It suffices to choose
$Q_i \in S^{p_i} \mathcal{K} \otimes S \mathcal{I}$ for some $p_i \in \NN_0$.
In this case, using Remark~\ref{rem:H_S_der}, the fact that
$P_S^\vee (Q_i ) \in S \mathcal{I}$ and $H_S^\vee \circ P_S^\vee = 0$, we 
have 
\[
\begin{aligned}
\Delta_{n+1} (Q_1 \odot Q_2) & = \big((d_{\mathcal{L}^\vee} - D_\mathcal{L}) \circ H_S^\vee \circ \Delta_n\big)(Q_1 \odot Q_2) \\
& = \big((d_{\mathcal{L}^\vee} - D_\mathcal{L}) \circ H_S^\vee \big)\big(\Delta_n (Q_1) \odot P_S^\vee(Q_2) + (-1)^{|Q_1|} P_S^\vee (Q_1) \odot \Delta_n (Q_2) + \Sigma
\big)\\
& = \Delta_{n+1}(Q_1) \odot P_S^\vee (Q_2) + (-1)^{|Q_1|} P_S^\vee (Q_1) \odot \Delta_{n+1} (Q_2) + \Sigma'
\end{aligned}
\]
where $\Sigma'$ is again a sum of terms of the form \eqref{eq:sum_terms}. The claim that \eqref{eq:d_hp_k} is a derivation, now immediately follows
from the fact that $I_S^\vee$ is a graded algebra morphism
and $I_S^\vee \circ H_S^\vee = 0$.

Next we show that $d_\mathcal{B}^{\mathrm{hp}}$ preserves
$\Gamma (S^0 \mathcal{B}^\vee [-1] ) = C^\infty(\mathcal{M})$,
where it acts as $d$ (the homological vector field on $\mathcal{M}$).
But this follows immediately from~\eqref{eq:d_B^hp}, using the fact
that $P^\vee_S, I^\vee_S, d_{\mathcal{L}^\vee}, D_\mathcal{L}, d_{\mathcal{B}^\vee}, H_S^\vee$ all preserve $C^\infty(\mathcal{M})$, and specifically, $d_{\mathcal{B}^\vee}$ acts as $d$, and $H^\vee_S$ vanishes on $C^\infty(\mathcal{M})$.
Hence $d_{\mathcal{B}^\vee}^{\mathrm{hp}}$ 
is indeed a homological vector field on $\mathcal{B}[1]$ tangent to
the zero section of the vector bundle $\mathcal{B}[1] \to \mathcal{M}$.
Therefore $(\mathcal{B}[1], d_{\mathcal{B}^\vee}^{\mathrm{hp}})$ is the dg
manifold associated to an $L_\infty$-algebroid structure on
$\mathcal{B}\to\mathcal{M}$ as desired.
\end{proof}

\begin{theorem}\label{prop:L_infty_alg}
Consider the contraction \eqref{eq:contraction_0}
as in Theorem~\ref{lem:contraction_0}.
The $L_\infty$-algebroid structure on the graded vector bundle
$A[1] \times_M L/A \to A[1]$ induced by the homotopy contraction
in Theorem~\ref{theor:ht_L_infty_alg}
coincides with the one as in \cite[Proposition 4.1]{MR4091493}.
\end{theorem}

\begin{proof}
Note that, shifting by $1$ the fiber degree of the vector bundle
$A[1] \times_M L/A \to A[1]$, we get the graded manifold
$A[1] \times_M L/A[1] = A[1] \times_M B[1] \cong L[1]$,
and recall from~\cite[Proposition~4.1]{MR4091493} that the homological vector field on
$L[1]$ corresponding to the $L_\infty$-algebroid structure on
$A[1] \times L/A \to A[1]$ is just the Chevalley--Eilenberg differential $d_L$.

Consider the contraction \eqref{eq:contraction_0}
as in Theorem~\ref{lem:contraction_0}.
As $\pi^! L \to A[1]$ is a dg Lie algebroid, in particular
an $L_\infty$-algebroid,
Theorem~\ref{theor:ht_L_infty_alg}
implies that there is an induced
$L_\infty$-algebroid structure
on $A[1] \times L/A \to A[1]$.
To prove Theorem~\ref{prop:L_infty_alg}, it suffices to 
show that the corresponding homological
vector field $d^{\mathrm{hp}}$ on $A[1] \times L/A[1] \cong L[1]$
given by the homological perturbation lemma as in the proof
of Theorem~\ref{theor:ht_L_infty_alg} coincides with $d_L$.
This can be done using the formulas in Remark~\ref{flageolet}.
Here we provide a simpler proof in local coordinates.

Choose local coordinates $(x^1, \cdots, x^n)$ on $M$
and a local frame $(e_1, \cdots , e_m)$ for $A$. Complete $(e_1, \cdots, e_m)$
to a local frame $(e_1, \cdots, e_m, j\circ g_1, \cdots, j \circ g_p)$ of $L$
by adding the image $(j\circ g_1, \cdots, j \circ g_p)$ via the splitting
$j : B \to L$ of some local frame $(g_1, \cdots, g_p)$ of $B$.
The coordinates $(x^1, \cdots, x^n)$, together with the local frame
$(e_1, \cdots, e_m, j \circ g_1, \cdots, j \circ g_p)$ 
on $L$ induce
a local chart $(x^1, \cdots, x^n, \xi^1, \cdots, \xi^m, \eta^1, \cdots, \eta^p)$
on $A[1] \times_M B[1] \cong L[1]$ where $(x^1, \cdots, x^n, \xi^1, \cdots, \xi^m)$
are coordinates on $A[1]$ and $(x^1, \cdots, x^n, \eta^1, \cdots, \eta^p)$
are coordinates on $B[1]$. According to Equation
\eqref{eq:dA}, 
the homological vector field $d_A$ on $A[1]$ reads
\[
d_A = \sum_{i, j} a_i^j (x) \xi^i \frac{\partial}{\partial x^j} - \frac{1}{2} \sum_{i,j,k} c_{ij}^k (x) \xi^i \xi^j \frac{\partial}{\partial \xi^k}.
\]
Similarly, the homological vector field $d_L$ on $L[1]$ reads
\begin{equation}\label{eq:d_L}
\begin{aligned}
d_L &= d_A - \sum_{\mu, j} B^j_\mu (x) \eta^\mu \frac{\partial}{\partial x^j} \\
& \quad - \sum_k\left(\sum_{j, \mu}C^k_{j\mu}(x) \xi^j \eta^\mu + \frac{1}{2} \sum_{\mu, \nu} D_{\mu\nu}^k (x) \eta^\mu \eta^\nu \right) \frac{\partial}{\partial \xi^k} \\
& \quad - \sum_\sigma\left(\sum_{j, \mu}E^\sigma_{j\mu}(x) \xi^j \eta^\mu + \frac{1}{2} \sum_{\mu, \nu} F_{\mu\nu}^\sigma (x) \eta^\mu \eta^\nu \right) \frac{\partial}{\partial\eta^\sigma}.
\end{aligned}
\end{equation}
The coordinates $(x^1, \cdots, x^n)$, together with the local frame 
 $(e_1, \cdots, e_m, j \circ g_1, \cdots, j \circ g_p)$
on $L$ do also determine a local frame
on
$\pi^! L\cong T_{A[1]}\times_{T_M}L$:
\[
(\dot{\eta}_1, \cdots, \dot{\eta}_m, A_1, \cdots, A_m, B_1, \cdots, B_p)
\]
given by
\[
\dot{\eta}_j = \left( \frac{\partial}{\partial \eta^j}, 0\right), \quad A_i = \left( a^j_i (x) \frac{\partial}{\partial x^j}, e_i \right), \quad B_\mu = \left( B_\mu^j \frac{\partial}{\partial x^j}, j \circ g_\mu \right),
\]
which in turn determines a local chart
\[
(x^1, \cdots, x^n, \xi^1, \cdots, \xi^n, \dot{\xi}{}^1, \cdots, \dot{\xi}{}^m, A^1, \cdots, A^m, B^1, \cdots, B^p)
\]
on $\pi^! L [1]$.

Next, we write all the relevant operators appearing in the perturbation formula \eqref{eq:d_B^hp}
in these coordinates. In the present situation, $\mathcal{B}=A[1]\times_M B$
and $d_{\mathcal{B}^\vee}$ is the homological vector field $d_A^{\mathsf{Bott}}$
on $A[1]\times_M B[1]$ corresponding to the Bott connection which, in the above coordinates, is given by
\begin{equation}\label{eq:d_A^Bott}
d_A^{\mathsf{Bott}} = d_A - \sum_{j, \mu, \sigma}E^\sigma_{j\mu}(x) \xi^j \eta^\mu \frac{\partial}{\partial\eta^\sigma}.
\end{equation}
Moreover $\mathcal{L}= \pi^! L$,
and the differentials $d_{\mathcal{L}^\vee},D_\mathcal{L}$ are
\begin{itemize}
\item the homological vector field on $\pi^! L [1]$ corresponding to the coboundary operator $\sQ=[s_{\boldi},\argument]$,
\item the homological vector field on $\pi^! L [1]$ corresponding to the dg Lie algebroid structure,
\end{itemize}
respectively. So the difference $D_\mathcal{L} - d_{\mathcal{L}^\vee}$ is the homological vector field $d_{\mathsf{Lie}}$ on $\pi^! L [1]$ corresponding to the graded Lie algebroid structure on $\pi^! L \to A[1]$, and it is locally given by
\begin{equation}\label{eq:d_Lie}
\begin{aligned}
d_{\mathsf{Lie}} & = \sum_j\dot \xi{}^j \frac{\partial}{\partial \xi^j} + \sum_i\left(\sum_j a^i_j(x) A^j + \sum_\mu B^i_\mu (x) B^\mu\right) \frac{\partial}{\partial x^i} \\
& \quad - \sum_k \left( \frac{1}{2} \sum_{i,j}c_{ij}^kA^iA^j + \sum_{j, \mu} C_{j\mu}^k(x) A^j B^\mu + \frac{1}{2} \sum_{\mu, \nu} D^k_{\mu\nu}(x) B^\mu B^\nu\right) \frac{\partial}{\partial A^k}\\
& \quad - \sum_\sigma \left( \sum_{j, \mu}E^\sigma_{j\mu}(x) A^j B^\mu + \frac{1}{2} \sum_{\mu, \nu} F_{\mu\nu}^\sigma (x) B^\mu B^\nu\right) \frac{\partial}{\partial B^\sigma}.
\end{aligned}
\end{equation}
The inclusion $P_S^\vee$ is the pull-back along the vector bundle projection $P_0 : \pi^! L[1] \to A[1] \times_M B[1]$ and it is locally given by
\begin{equation}\label{eq:P_S^vee}
P_S^\vee (x^i) = x^i, \quad P_S^\vee (\xi^j) = \xi^j, \quad P_S^\vee (\eta^\mu) = B^\mu.
\end{equation}
The projection $I_S^\vee$ is the pull-back along the vector bundle inclusion $I_0 : A[1] \times_M B[1] \to \pi^! L[1]$ and it is locally given by
\begin{equation}\label{eq:I_S^vee}
I_S^\vee (x^i) = x^i, \quad I_S^\vee (\xi^j) = \xi^j, \quad I_S^\vee (\dot\xi{}^k) = C_{j\mu}^k \xi^j \eta^\mu, \quad I^\vee_S (A^l) = 0, \quad I^\vee_S (B^\nu) = \eta^\mu.
\end{equation}
Finally, $H_S^\vee$ acts on coordinates as the transpose of the homotopy $H_0 : \pi^! L[1] \to \pi^! L[1]$, i.e.
\begin{equation}\label{eq:H_S^vee}
H_S^\vee (x^i) = 0, \quad H_S^\vee (\xi^j) = 0, \quad H_S^\vee (\dot\xi{}^k) = A^k, \quad H^\vee_S (A^l) = 0, \quad H^\vee_S (B^\nu) = 0.
\end{equation}

It is now easy to see in coordinates, combining \eqref{eq:d_A^Bott}, \eqref{eq:d_Lie}, \eqref{eq:P_S^vee}, \eqref{eq:I_S^vee}, \eqref{eq:H_S^vee}, and the 
fact that $d^{\mathrm{hp}}$ is a derivation, that
\begin{itemize}
\item the terms with $i > 1$ in the perturbation expansion \eqref{eq:d_B^hp}
 of $d^{\mathrm{hp}}$ vanish, while
\item the rest amounts exactly to $d_L$
(as locally given by~\eqref{eq:d_L}).
\end{itemize}
This concludes the proof.
\end{proof}

\section{\texorpdfstring{$A_\infty$}{Aoo}-algebras from Lie pairs}\label{sec:A_infty_LP}
\subsection{Statement of the main theorem}

In this section, we show how an $A_\infty$-algebra
can be constructed from a Lie pair $(L,A)$
equipped with some appropriate additional data.
In order to state the main theorem, some preliminary
remarks are in order.

The universal enveloping algebra $\cU(L)$ of the Lie algebroid $L$
--- see~\cite{MR2653938,MR1687747} --- is naturally a module over the Lie--Rinehart algebra
$\big(\Gamma(A),R\big)$: the sections of $A$ and the functions on $M$ act by
multiplication from the left in the universal enveloping algebra $\cU(L)$.
This module structure descends to the quotient
\[ \cU_{L/A} := \frac{\cU(L)}{\cU(L) \cdot \Gamma (A)} .\]

It follows that
$\Gamma (\Lambda^\bullet A^\vee) \otimes_{R} \cU_{L/A}$
is also equipped with a differential, again abusively denoted $\dA$,
which turns $\left( \Gamma (\Lambda^\bullet A) \otimes_R \cU_{L/A}, \dA\right)$
into a dg module over the \dga\ $\left( \Gamma (\Lambda^\bullet A^\vee), \dA \right)$.

We are now ready to state our main result of the paper.

\begin{theorem}\label{theor:main}
Let $(L,A)$ be a Lie pair.
\begin{enumerate}
\item Each choice of splitting of the short exact sequence of vector bundles
\[ 0 \to A \to L \to L/A \to 0 ,\]
(and of some additional geometric data to be specified later)
determines an $A_\infty$-algebra structure
on $\Gamma(\Lambda^\bullet A^\vee)\otimes_R\cU_{L/A}$,
which extends the differential $\dA$.
\item The $A_\infty$-algebras arising from all possible
choices of splitting (and of additional geometric data)
are all mutually isomorphic.
\item In particular, the cohomology
$H\big(\Gamma(\Lambda^\bullet A^\vee)\otimes_R\cU_{L/A},\dA\big)$
carries a canonical graded associative algebra structure.
\end{enumerate}
\end{theorem}

\begin{remark}
Part~(1) of Theorem~\ref{theor:main} was proved in~\cite{MR3313214} for the case where
$L = T_M$ is the tangent bundle of $M$ and $A \subseteq L$ is an involutive distribution on $M$.
In that case, $\cU(L) = \cU(T_M)$ consists of linear differential operators
on $M$, and $\cU_{L/A}$ can be interpreted as differential operators on $M$
\emph{transverse to the foliation $\cF$} integrating the distribution $A$.
In turn, the cochain complex
$\left( \Gamma (\Lambda^\bullet A^\vee) \otimes_R \cU_{L/A}, \dA \right)$
can be considered as the space of
 \emph{differential operators on the leaf space
$M/\cF$ of the foliation $\cF$}.
It inherits an algebraic structure from $\cU_{L/A}$ (the differential operators on $M$
transverse to $\cF$): namely, the $A_\infty$-algebra mentioned in Theorem~\ref{theor:main}.
\end{remark}

The proof of Theorem~\ref{theor:main} consists of various steps, which we present in the form
of a list of propositions and lemmas as they might be of independent interest.

\subsection{Induced contraction of symmetric algebras}

We will need to consider the symmetric algebra $\Gamma (S \pi^! L)$
of the dg Lie algebroid $\pi^! L \to A[1]$.
The algebra $\Gamma (S \pi^! L)$ of sections of $S \pi^! L$
is a graded Poisson algebra with degree $0$ Poisson bracket $\{-,-\}$
induced by the Lie algebroid $\pi^! L$ in the
usual way, i.e.\ the Lie-Poisson structure
on $(\pi^! L)^\vee$.

More importantly for our purposes, the differential
$\sQ=[s_{\boldi},\argument]$ in $\Gamma (\pi^! L)$
extends uniquely to a \emph{homological Poisson derivation}
\[ D_S = \{ s_{\boldi},\argument \} :
\Gamma (S \pi^! L) \to \Gamma (S \pi^! L) ,\]
i.e.\ a derivation of degree $+1$ of both the associative product and the Poisson bracket,
endowing $\Gamma (S \pi^! L)$ with a structure of \emph{dg Poisson algebra}.
In particular, $\big(\Gamma(S\pi^! L),D_S\big)$ is a dg module over the 
\dga\ $(\Gamma(\Lambda^\bullet A^\vee),\dA)$.
We have an obvious isomorphism of dg modules:
\begin{equation}\label{eq:Delta1}
\big(\Gamma(S\pi^! L),D_S\big) \xto{\simeq}
\big(S_{\Gamma(\Lambda^\bullet A^\vee)} \Gamma(A[1];\pi^! L),D_S\big)
.\end{equation}

Now consider the bundle of symmetric algebras $S B \to M$
and the graded $\Gamma(\Lambda^\bullet A^\vee)$-module
$\Gamma(\Lambda^\bullet A^\vee\otimes S B)$.
The Bott $A$-connection on $B$ extends naturally to a flat $A$-connection on $SB$
with associated Chevalley--Eilenberg differential
\[ \dA : \Gamma (\Lambda^\bullet A^\vee \otimes S B)
\to \Gamma (\Lambda^{\bullet + 1} A^\vee \otimes S B) .\]
It is clear that both
$\big(\Gamma(\Lambda^\bullet A^\vee\otimes B),\dA\big)$
and $\big(\Gamma(\Lambda^\bullet A^\vee\otimes S B),\dA\big)$ are dg modules over the \dga\ $(\Gamma(\Lambda^\bullet A^\vee),\dA)$,
and we have an obvious isomorphism of dg modules
\begin{equation}\label{eq:Delta2}
\big(\Gamma(\Lambda^\bullet A^\vee\otimes S B),\dA\big) \xto{\simeq}
\big(S_{\Gamma(\Lambda^\bullet A^\vee)}\Gamma(\Lambda^\bullet A^\vee\otimes B),\dA\big)
\end{equation}

Now, as an immediate consequence, Lemma~\ref{lem:tensor-trick} together
with~\eqref{eq:Delta1}, and~\eqref{eq:Delta2} implies the following

\begin{proposition}
The contraction data \eqref{eq:contraction_0} extends naturally
to contraction data $(P_S, I_S, H_S)$:
\begin{equation}\label{eq:contraction:S}
\begin{tikzcd}
\big(\Gamma(S\pi^! L),D_S\big)
\arrow[r, "P_S", shift left] \arrow["H_S", loop left] &
\big(\Gamma(\Lambda^\bullet A^\vee\otimes S B),\dA\big)
\arrow[l, "I_S", shift left]
\end{tikzcd}
.\end{equation}
\end{proposition}

\begin{remark}
It is simple to check that
$(P_S, I_S, H_S)$ is indeed a semifull algebra contraction
of dg commutative algebras.
Thus, according to \cite[Theorem 2.16]{MR4091493},
via homotopy transfer, it gives rise to
$L_\infty$-brackets on $\Gamma(\Lambda^\bullet A^\vee\otimes SB)$,
which make it into a $P_\infty$-algebra \cite{MR2304327,MR2180451}.
One can also check that its $l$th bracket, for any $l\geq 1$,
is of the weight $(1 -l)$. Therefore, according to
\cite[Proposition 3.8]{MR4091493}, such a $P_\infty$-algebra
is equivalent to an $L_\infty$-algebroid on $A[1]\times_M B\to A[1]$.
We can show that this coincides with the one obtained
from Theorem~\ref{prop:L_infty_alg}.
See~\cite{MR4091493,MR3313214,MR3300319} for details.
\end{remark}

\subsection{Projection map for contraction on \texorpdfstring{$\cU(\pi^! L)$}{U(pi! L)}}

In the next step towards the proof of Theorem~\ref{theor:main}
we consider the universal enveloping algebra $\cU(\pi^! L)$ of the dg
Lie algebroid $\pi^! L \to A[1]$. It is a graded associative algebra. Additionally,
the differential $\sQ=\lie{s_{\boldi}}{\argument}$ in
$\Gamma (\pi^! L)$ induces a differential
$D_\cU : \cU(\pi^! L) \to \cU(\pi^! L)$.
Namely, the homological section $s_{\boldi} \in \Gamma (\pi^! L)$
can also be seen as an element in $\cU(\pi^! L)$ via the canonical
embedding $\Gamma (\pi^! L) \hookrightarrow \cU(\pi^! L)$, and we put
\begin{equation}\label{eq:Dour}
D_\cU \Upsilon = [s_{\boldi},\Upsilon]
= s_{\boldi} \circc \Upsilon
- (-1)^{|\Upsilon|} \Upsilon \circc s_{\boldi}
,\end{equation}
for all $\Upsilon \in \cU (\pi^! L)$.
The differential $D_\cU$ gives to $\cU(\pi^! L)$ the
structure of a dg associative algebra, which is also
a dg module over $\big(\Gamma (\Lambda^\bullet A^\vee), \dA\big)$.

\begin{lemma}\label{lem:P_U}
There is a unique morphism of dg modules
\[ P_\cU : \cU(\pi^! L)
\to \Gamma(\Lambda^\bullet A^\vee)\otimes_R\cU_{L/A} \]
over the dg algebra $\big( \Gamma(\Lambda^\bullet A^\vee), \dA \big)$
satisfying
\begin{equation}\label{eq:P_U}
P_\cU\big((Y_1,v_1\circ\pi) \circc\cdots\circc (Y_n,v_n\circ\pi)\big)
=\class{v_1 \circc \cdots \circc v_n}
\end{equation}
for all pairs of the special type $(Y_1,v_1\circ\pi),\cdots,(Y_n,v_n\circ\pi)$
characterizing sections of the Lie algebroid $\pi^! L\to A[1]$.
The notation $\overline{u}$ refers to the equivalence class
in $\cU_{L/A} = \frac{\cU(L)}{\cU(L)\cdot\Gamma(A)}$
of an element $u\in\cU(L)$.
\end{lemma}
\begin{proof}

The statement is a consequence of the following general fact.
Given $\ZZ$-graded Lie algebroids $\cL_\cN \to \cN$ and $\cL_\cM \to \cM$
and a Lie algebroid morphism $F : \cL_\cN \to \cL_\cM$ covering a smooth map
$f:\cN\to\cM$,
there is a unique $C^\infty(\cN)$-linear map of filtered modules
\begin{equation}\label{eq:F_U}
F_\cU : \cU(\cL_\cN) \to C^\infty(\cN)
\underset{C^\infty(\cM)}{\otimes} \cU(\cL_\cM)
\end{equation}
such that
\[ F_\cU (V_1 \circc \cdots \circc V_n) = v_1 \circc \cdots \circc v_n \]
for any $n$-tuple of pairs $(V_i, v_i)$, where
$V_i \in \Gamma (\cL_\cN)$
and $v_{i} \in \Gamma (\cL_\cM)$ are $F$-related sections, i.e.~$F \circ V_i =
v_i \circ f$, $i = 1, \cdots, n$.
In~\eqref{eq:F_U}, $C^\infty(\cN)$ is equipped with the structure of a
$C^\infty(\cM)$-module via the pull-back $f^\ast : C^\infty(\cM) \to
C^\infty(\cN)$.

Applying the discussion above to the case that
$\cN = A[1]$, $\cL_\cN = \pi^! L$, 
$f = \pi : A[1] \to M$,
and $F=\pr_2: \pi^! L\cong A[1]\times_{TM}L\to L$, 
we obtain a degree $0$ $\Gamma (\Lambda^\bullet A^\vee) \cong
C^\infty(A[1])$-linear map of filtered modules
\[ \pi_\cU : \cU(\pi^! L) \to \Gamma (\Lambda^\bullet A^\vee)
\underset{R}{\otimes} \cU(L) \]
such that
\[ \pi_\cU \big((Y_1,v_1\circ\pi) \circc\cdots\circc (Y_n,v_n\circ\pi)\big)
= v_1 \circc\cdots\circc v_n \]
for all characterizing pairs of the special type
$(Y_i,v_i\circ\pi) \in \Gamma (\pi^! L)$. Composing with the projection
\[ \Gamma (\Lambda^\bullet A^\vee) \otimes_R \cU(L)
\to \Gamma (\Lambda^\bullet A^\vee) \otimes_R \cU_{L/A} \]
we obtain the desired homomorphism $P_\cU$.
The uniqueness follows from the fact that compatible sections generate
the $\Gamma (\Lambda^\bullet A^\vee) $-module $\Gamma (\pi^! L)$.

It remains to check that $P_\cU$ is a cochain map.
It is simple to see that the operator
\[ P_\cU\circ D_\cU -\dA\circ P_\cU:
\quad\cU(\pi^! L)\to\Gamma(\Lambda^{\bullet+1} A^\vee)
\otimes_R\cU_{L/A} \]
is $\Gamma(\Lambda^\bullet A^\vee)$-linear.
Therefore, it suffices to check that
$P_\cU\circ D_\cU$ and $\dA\circ P_\cU$
agree on elements of the form
$(Y_1,v_1\circ\pi) \circc \cdots \circc (Y_n,v_n\circ\pi)$,
where $(Y_i,v_i\circ\pi) \in \Gamma(\pi^! L)$.

It follows from Equations~\eqref{eq:Dour}, \eqref{eq:Darwin}, and~\eqref{eq:P_U} that
\[ \begin{split}
& P_\cU \circ D_\cU \big( (Y_1, v_1 \circ \pi) \circc
\cdots \circc (Y_n, v_n \circ \pi) \big) \\
& = P_\cU \left(s_{\boldi} \circc
(Y_1, v_1 \circ \pi) \circc \cdots \circc (Y_n, v_n \circ \pi)
- (-1)^{|Y_1| + \cdots + |Y_n|} (Y_1, v_1 \circ \pi) \circc \cdots \circc
(Y_n, v_n \circ \pi) \circc s_{\boldi} \right) \\
& = \sum_k P_\cU \big( \xi^k\cdot(\liederivative{e_k}, i\circ e_k \circ \pi)
\circc (Y_1, v_1 \circ \pi) \circc \cdots \circc (Y_n, v_n \circ \pi) \big) \\
& \quad - \sum_k P_\cU \big( (d_A \xi^k\cdot\interior{e_k}, 0)
\circc (Y_1, v_1 \circ \pi) \circc \cdots \circc (Y_n, v_n \circ \pi) \big) \\
& \quad - (-1)^{|Y_1|+\cdots+|Y_n|} \sum_k P_\cU \big( (Y_1, v_1 \circ \pi)
\circc \cdots \circc (Y_n, v_n \circ \pi) \circc \xi^k\cdot
(\liederivative{e_k}, i\circ e_k \circ \pi) \big) \\
& \quad + (-1)^{|Y_1|+\cdots+|Y_n|} \sum_k P_\cU \big( (Y_1, v_1 \circ \pi)
\circc \cdots \circc (Y_n, v_n \circ \pi) \circc
(d_A\xi^k\cdot\interior{e_k},0) \big) \\
& = \sum_k \xi^k\cdot \class{i(e_k)\cdot v_1 \cdot \cdots\cdot v_n}
- \sum_k \class{0\cdot v_1 \cdot \cdots\cdot v_n} \\
& \quad - (-1)^{|Y_1|+\cdots+|Y_n|} \sum_k P_\cU \big( (Y_1, v_1 \circ \pi)
\circc \cdots \circc (Y_n, v_n \circ \pi) \circc \xi^k\cdot
(\liederivative{e_k}, i\circ e_k \circ \pi) \big) \\
& \quad + (-1)^{|Y_1|+\cdots+|Y_n|} \sum_k
\class{v_1 \cdot \cdots\cdot v_n\cdot 0}
.\end{split} \]

Using the identity $(Y,v\circ\pi)\cdot f=Y(f)+f\cdot(Y,v\circ\pi)$ repeatedly,
each product of the form
\[ (Y_1, v_1 \circ \pi)
\circc \cdots \circc (Y_n, v_n \circ \pi) \circc \xi^k\cdot
(\liederivative{e_k}, i\circ e_k \circ \pi) \]
can be rewritten as a linear combination of the form
\[ \sum_{j=1}^n f_j\cdot (Y_1,v_1\circ\pi)\cdot\cdots\cdot(Y_j,v_j\circ\pi)
\cdot(\liederivative{e_k}, i\circ e_k \circ \pi) \]
with coefficients $f_j$ in $\Gamma(\Lambda^\bullet A^\vee)$.
It follows that
\[ P_\cU \big((Y_1, v_1 \circ \pi)
\circc \cdots \circc (Y_n, v_n \circ \pi) \circc \xi^k\cdot
(\liederivative{e_k}, i\circ e_k \circ \pi)\big)
=\sum_{j=1}^n f_j\cdot\class{v_1\cdot\cdots\cdot v_j\cdot i(e_k)}=0, \]
since $e_k\in\Gamma(A)$.

Hence, we obtain
\begin{multline*}
P_\cU \circ D_\cU \big( (Y_1, v_1 \circ \pi) \circ
\cdots \circ (Y_n, v_n \circ \pi) \big)
= \sum_k \xi^k\cdot \class{i(e_k)\cdot v_1 \cdot \cdots\cdot v_n} \\
= \sum_k \xi^k\cdot e_k \triangleright \overline{v_1 \circc \cdots \circc v_n}
= \dA \big( \overline{v_1 \circc \cdots \circc v_n}\big)
= \dA\circ P_\cU \big( (Y_1, v_1 \circ \pi) \circc \cdots \circc
(Y_n, v_n \circ \pi) \big)
.\qedhere\end{multline*}
This concludes the proof of the lemma.
\end{proof}

\subsection{Contraction of universal enveloping algebras}
We are finally ready to present our strategy to prove Theorem~\ref{theor:main}.
The idea is to complete the cochain map $P_\cU$ to a whole set of contraction 
data
$(P_\cU, I_\cU, H_\cU)$:

\begin{theorem}
\label{thm:Genova}
A set of contraction data
\begin{equation}
\label{eq:contraction_S}
\begin{tikzcd}
\Big(\cU(\pi^! L),D_\cU\Big)
\arrow[r, "P_\cU", shift left] \arrow["H_\cU", loop left] &
\Big(\Gamma(\Lambda^\bullet A^\vee)\otimes_R\cU_{L/A},\dA\Big)
\arrow[l, "I_\cU", shift left]
\end{tikzcd}
\end{equation}
can be constructed from a splitting of~\eqref{eq:splitting1}
and appropriate $L$-connections $\nabla^A$, $\nabla^B$, and $\nabla^L$ on $A$, $B$, and $L$, and a linear connection on $A$.

Moreover, the projection map is precisely the cochain map $P_\cU$ from Lemma~\ref{lem:P_U}.
\end{theorem}

Then using homotopy transfer,
we produce the desired $A_\infty$-algebra structure
on $\Gamma (\Lambda^\bullet A^\vee )\otimes_R \cU_{L/A}$
from the dg associative algebra structure on
$\cU(\pi^! L)$. 

To prove Theorem~\ref{thm:Genova}, We proceed as follows.

First, using an $L$-connection $\nabla^B$ on $B$ extending the Bott
$A$-connection
and a $\pi^! L$-connection $\nabla$ on $\pi^! L$,
we construct a pair of isomorphisms of Poincaré-Birkhoff-Witt type:
\[ \PBW : \Gamma(S\pi^! L) \xto{\simeq} \cU(\pi^! L)
\qquad\text{and}\qquad
\pbw : \Gamma(S B) \xto{\simeq} \cU_{L/A} .\]
Second, we push-forward the contraction diagram \eqref{eq:contraction:S}
via $\PBW$ and $\id\otimes\pbw$
so as to obtain the isomorphic contraction diagram
\begin{equation}
\label{eq:Roma}
\begin{tikzcd}
\Big(\cU(\pi^! L),D^{\pbw}_S\Big)
\arrow[r, "P^{\pbw}_S", shift left] \arrow["H^{\pbw}_S", loop left] &
\Big(\Gamma(\Lambda^\bullet A^\vee)\otimes_R\cU_{L/A},\dA^{\pbw}\Big)
\arrow[l, "I^{\pbw}_S", shift left]
\end{tikzcd}
\end{equation}
with
\[ \begin{gathered}
D^{\pbw}_S=\PBW\circ D_S\circ\PBW^{-1} \\
\dA^{\pbw}=(\id\otimes\pbw)\otimes
\dA\circ(\id\otimes\pbw)^{-1}
\end{gathered}
\qquad\text{and}\qquad
\begin{gathered}
P^{\pbw}_S=(\id\otimes\pbw)\circ P_S\circ\PBW^{-1} \\
I^{\pbw}_S=\PBW\circ I_S\circ(\id\otimes\pbw)^{-1} \\
H^{\pbw}_S=\PBW\circ H_S\circ\PBW^{-1}
\end{gathered} .\]

Third, we interpret $D_\cU$ as a perturbation of $D^{\pbw}_S$
and use homological perturbation to produce a new contraction diagram
\begin{equation}\label{eq:contr_hp}
\begin{tikzcd}
\Big(\cU(\pi^! L),D_\cU\Big)
\arrow[r, "P^{\hp}_S", shift left] \arrow["H^{\hp}_S", loop left] &
\Big(\Gamma(\Lambda^\bullet A^\vee)\otimes_R\cU_{L/A},\dA^{\hp}\Big)
\arrow[l, "I^{\hp}_S", shift left]
\end{tikzcd}
.\end{equation}
Fourth and last step, we show that the connections $\nabla^B$ and $\nabla$
can be chosen in such a way that $P^{\hp}_S=P_\cU$ and, consequently,
$\dA^{\hp}=\dA$.

This rest of the section is devoted
to the proof of Theorem~\ref{thm:Genova},
by carrying out, in detail, the outline described above.

\subsubsection{Poincaré--Birkhoff--Witt type isomorphisms}

The next two propositions provide us with the pair of isomorphisms
of Poincaré--Birkhoff--Witt type $\PBW$ and $\pbw$.

\begin{proposition}[{\cite[Lemma~4.2]{MR3319134} and
\cite[Proposition~4.2]{MR3910470}}]
Given a $\pi^! L$-connection $\nabla$ on $\pi^! L$,
there exists a unique isomorphism of filtered $\Gamma(\Lambda^\bullet A^\vee)$-modules
\[ \PBW : \Gamma(S\pi^! L) \to \cU(\pi^! L) \]
satisfying
\[ \PBW(f)=f \qquad\text{and}\qquad \PBW(s)=s ,\]
for all $f\in C^\infty(A[1])\cong\Gamma(\Lambda^\bullet A^\vee)$
and $s\in\Gamma(\pi^! L)$, and, moreover,
\begin{equation}\label{eq:PBW}
\PBW (s_1 \odot \cdots \odot s_{n}) = \frac{1}{n} \sum_{i = 1}^n
(-1)^{|s_i| (|s_1|+ \cdots + |s_{i-1}|)} \left(s_i\cdot \PBW(s^{\{i\}})
-\PBW(\nabla_{s_i}s^{\{i\}}) \right),
\end{equation}
for all $s_1,\cdots,s_n\in\Gamma(\pi^! L)$,
where $s^{\{i\}} := s_1 \odot\cdots\odot \widehat{s_i} \odot\cdots\odot s_n$.
\end{proposition}

\begin{proposition}\cite[Theorem~2.1]{MR4271478}
Given an $L$-connection $\nabla^B$ on $B$ extending the Bott $A$-connection,
there exists a unique induced isomorphism of filtered $\cirm$-modules
\[ \pbw : \Gamma(S B) \to \cU_{L/A} \]
satisfying
\[ \pbw(f) = f \qquad\text{and}\qquad \pbw(\betaa) = \betaa ,\]
for all $f\in\cirm$ and $\betaa\in\Gamma(B)$, and, moreover,
\[ \pbw(\betaa_1\odot\cdots\odot\betaa_n)
= \frac{1}{n} \sum_{i=1}^n \left( j(\betaa_i) \circc \pbw(\betaa^{\{i\}})
- \pbw(\nabla^B_{j(\betaa_i)}\betaa^{\{i\}}) \right) ,\]
for all $\betaa_1,\dots,\betaa_n\in\Gamma(B)$, where, in the r.h.s,
we used the obvious left $\cU(L)$-module structure on $\cU_{L/A}$
(also denoted by $\circc$), and
$\betaa^{\{i\}}$ is an abbreviation for
$\betaa_1 \odot\cdots\odot \widehat{\betaa_i} \odot\cdots\odot \betaa_n$.
\end{proposition}

The second step described above is straightforward.
The third step requires some explanations.
In order for the classical homological perturbation lemma to be applicable,
we need $H_{S}^{\pbw}\circ(D_\cU-D_{S}^{\pbw})$
to be a locally nilpotent operator,
i.e.\ for every $\Upsilon\in\cU(\pi^! L)$ there must exist $k\in\NN$
such that $\big(H_{S}^{\pbw}\circ(D_\cU-D_{S}^{\pbw})\big)^k(\Upsilon)=0$.
This is indeed the case.
To show it, notice that the operator
$D_\cU-D_{S}^{\pbw}$ lowers by $1$ the order
with respect to the standard filtration of $\cU(\pi^! L)$
and that the operator $H_{S}^{\pbw}$ respects that same filtration.
This follows from the definitions of the operators $D_\cU$, $D_S$, and $H_S$
and the simple fact that
\[ \PBW(s_1 \odot \cdots \odot s_n) = s_1 \cdot \cdots \cdot s_n
+ \text{lower order terms} \]
for all $s_1,\cdots,s_n\in\Gamma(\pi^! L)$.
Since the filtration in $\cU(\pi^! L)$ is bounded from below
(the first non-trivial order in the filtration is the $0$-th order),
the operator $H^{\pbw}_S \circ(D_\cU-D_{S}^{\pbw})$
is indeed locally nilpotent and the homological perturbation lemma applies.

It remains to construct connections $\nabla^B$ and $\nabla$
such that $P_{S}^{\hp}=P_\cU$.

\subsubsection{A key proposition}

We need a key proposition.

\begin{proposition}\label{pro:nabla}
There exists an $L$-connection $\nabla^B$ on $B$ extending the Bott $A$-connection
and a $\pi^! L$-connection $\nabla$ on $\pi^! L$ such that
\begin{equation}\label{eq:P_PBW_I}
P_\cU \circ I_S^{\pbw} = \id
\end{equation}
and
\begin{equation}\label{eq:P_PBW_H}
P_\cU \circ H_S^{\pbw} = 0
.\end{equation}
\end{proposition}

First, we need a few lemmas.

\begin{lemma}\label{lem:FRA}
There exists an $L$-connection $\nabla^A$ on $A$,
an $L$-connection $\nabla^B$ on $B$,
and an $L$-connection $\nabla^L$ on $L$ satisfying
the relations
\[ \nabla^L_X \big(i(a)\big) = i\big(\nabla^A_X a\big)
\qquad\text{and}\qquad
\nabla^L_X \big(j(b)\big) = j\big(\nabla^B_X b\big) \]
as well as
\[ \nabla^A_{j(b)} a = \Delta_b a
\qquad\text{and}\qquad
\nabla^B_{i(a)} b = \nabla^{\Bott}_a b \]
for all $X\in\Gamma(L)$, $a\in\Gamma(A)$, and $b\in\Gamma(B)$.
\end{lemma}

\begin{proof}
Choose an (auxiliary) $L$-connection $\nabla'$ on $L$.
Then define an $L$-connection $\nabla^L$ on $L$ by the relation
\[ \nabla^L_X Y=ip\big(\nabla'_{ip(X)}ip(Y)+\lie{jq(X)}{ip(Y)}\big)
+jq\big(\nabla'_{jq(X)}jq(Y)+\lie{ip(X)}{jq(Y)}\big) ,\]
an $L$-connection $\nabla^A$ on $A$ by the relation
\[ \nabla^A_X a=p\big(\nabla'_{ip(X)}i(a)+\lie{jq(X)}{i(a)}\big) ,\]
and an $L$-connection $\nabla^B$ on $B$ by the relation
\[ \nabla^B_X b=q\big(\nabla'_{jq(X)}j(b)+\lie{ip(X)}{j(b)}\big) .\]
It is straightforward to check that $\nabla^L$, $\nabla^A$, and $\nabla^B$
enjoy the desired properties.
\end{proof}

Now we proceed to construct the $\pi^! L$-connection
$\nabla$ on $\pi^! L$.

Choose a linear connection, i.e. a 
$T_M$-connection on the vector bundle $A\to M$.
It induces a linear connection on the vector bundle $A[1]\xto{\pi}M$,
which, in turn, induces a trivialization \cite{MR2581370}
of the tangent bundle $T_{A[1]}$:
\[ T_{A[1]}\xto{\simeq}A[1]\times_M (A[1]\oplus T_M)\cong\pi^*(A[1]\oplus T_M) .\]
Consequently, we obtain the following isomorphism of vector bundles over $A[1]$:
\[ \pi^! L=T_{A[1]} \times_{T_M} L\xrightarrow[\simeq]{\Phi}
A[1]\times_M(A[1]\oplus L) \cong \pi^* (A[1]\oplus L) ,\]
which maps $(\nabla^A_v,v\circ\pi)\in\Gamma(\pi^! L)$
to $\pi^*(v)\in\Gamma\big(\pi^*(A[1]\oplus L)\big)$
and maps $(\iota_a,0)\in\Gamma(\pi^! L)$
to $\pi^*(a[1])\in\Gamma\big(\pi^*(A[1]\oplus L)\big)$.

The $L$-connections $\nabla^A$ and $\nabla^L$ on $A$ and $L$
induce an $L$-connection on the graded vector bundle $A[1]\oplus L\to M$.
Pulling back the Lie algebroid $L$, the graded vector bundle $A[1]\oplus L$,
and this $L$-connection on $A[1]\oplus L\to M$
through the surjective submersion $\pi: A[1]\to M$,
we obtain a $\pi^! L$-connection $\nabla$ on the vector bundle
$\pi^*(A[1]\oplus L)\cong \pi^! L$.

\begin{lemma}\label{eq:BOS}
For all $v,v'\in\Gamma(L)$ and $a,a'\in\Gamma(A)$, we have
\[ \begin{gathered}
\nabla_{(\nabla^A_v,v\circ\pi)}(\nabla^A_{v'},v'\circ\pi)
= \big(\nabla^A_{\nabla^L_v v'},(\nabla^L_v v')\circ\pi\big), \\
\nabla_{(\interior{a},0)}(\nabla^A_{v'},v'\circ\pi) = (0,0), \\
\nabla_{(\nabla^A_v,v\circ\pi)}(\interior{a},0)
= (\interior{\nabla^A_v a},0), \\
\nabla_{(\interior{a},0)}(\interior{a'},0) = (0,0).
\end{gathered} \]
\end{lemma}

\begin{proof}
Given a Lie algebroid $L\to M$ and
a vector bundle $E\to M$ over the same
base manifold $M$, an $L$-connection $\nabla^E$
on $E$, and a surjective submersion
$\pi:N\to M$, we can consider the pullback
Lie algebroid $\pi^! L\to N$, the pullback
vector bundle $\pi^* E\to N$ and
the pullback $\pi^! L$-connection $\nabla$
on $\pi^* E\to N$.
For any section $s$ of the vector bundle $E\to M$,
any section $v$ of the Lie algebroid $L\to M$,
and any $\pi$-related section $(X,v\circ\pi)$
of the pullback Lie algebroid $\pi^! L$, we have
\begin{equation}\label{eq:edmond}
\nabla_{(X,v\circ\pi)}\pi^* s=\pi^*(\nabla^E_v s)
.\end{equation}

In the special case at hand, we have
$N=A[1]$, $E=A[1]\oplus L$,
and $\nabla^E=\nabla^A+\nabla^L$.
Furthermore, the pullback Lie algebroid is
identified to the pullback vector bundle through
the map $\Phi:\pi^! L\to\pi^*(A[1]\oplus L)$.

The desired identities follow immediately from
Equation~\eqref{eq:edmond} and the identities
\begin{gather*}
\pi^*(v)=\Phi\big((\nabla^A_v,v\circ\pi)\big),
\quad\forall v\in\Gamma(L); \\ \intertext{and}
\pi^*(a[1])=\Phi\big((\interior{a},0)\big),
\quad\forall a\in\Gamma(A)
.\qedhere\end{gather*}
\end{proof}

\begin{corollary}\label{Omacron}
\strut
\begin{enumerate}
\item If $s\in\Gamma(\pi^! L)$ is $\pi$-related to $v\in\Gamma(L)$
and $s'\in\Gamma(\pi^! L)$ is $\pi$-related to $v'\in\Gamma(L)$,
then $\nabla_s s'\in\Gamma(\pi^! L)$ is $\pi$-related
to $\nabla^L_v v'\in\Gamma(L)$.
\item In particular, if either $s$ or $s'$ is null, then so is $\nabla_s s'$.
\end{enumerate}
\end{corollary}

\begin{proof}
Since the section $s$ of $\pi^! L$ is
$\pi$-related to the section $v$ of $L$,
we have $s=(X,v\circ\pi)$ for some vector field
$X$ on $A[1]$ $\pi$-related to the vector field
$\rho(v)$ on $M$.
Since the pair $(\nabla^A_v,v\circ\pi)$
characterizes another section of $\pi^! L$
$\pi$-related to the section $v$ of $L$,
the vector field $X-\nabla^A_v\in\XX(A[1])$
is tangent to the fibers of $\pi:A[1]\to M$
and can therefore be written as a linear
combination \[ X-\nabla^A_v=
\sum_k\alpha^k\cdot\interior{a_k} \]
where the $a_k$'s are sections of $A$
and the coefficients $\alpha^k$ are
elements of the algebra $C^\infty(A[1])$.
It follows that
\[ s=(X,v\circ\pi)=(\nabla^A_v,v\circ\pi)
+\sum_k\alpha^k\cdot(\interior{a_k},0) .\]
Indeed, every section of $\pi^! L$ that is
$\pi$-related to a section $v$ of $L$ arises as
the sum of $(\nabla^A_{v},v\circ\pi)$ and
a linear combination of sections of type
$(\interior{a},0)$.
The conclusion now follows immediately from
Lemma~\ref{eq:BOS}.
\end{proof}

\begin{corollary}\label{lem:3.20}
Consider the injection $I_0:\Gamma(\Lambda^\bullet A^\vee\otimes B)
\to\Gamma(\pi^! L)$ defined in Section~\ref{defIo}.
For all $b,b'\in\Gamma(B)$, we have
\[ \nabla_{I_0(1\otimes b)}\big(I_0(1\otimes b')\big)
=I_0\big(1\otimes\nabla^B_{j(b)}b'\big) .\]
\end{corollary}

\begin{proof}
According to Remark~\ref{flageolet} and Lemma~\ref{lem:FRA}, we have
\[ I_0(1\otimes b)=(\Delta_b,j\circ b\circ\pi)
=\big(\nabla^A_{j(b)},j(b)\circ\pi\big) .\]

It follows from Lemmas~\ref{eq:BOS} and~\ref{lem:FRA} that
\begin{multline*}
\nabla_{I_0(1\otimes b)}\big(I_0(1\otimes b')\big)
=\nabla_{\big(\nabla^A_{j(b)},j(b)\circ\pi\big)}
\big(\nabla^A_{j(b')},j(b')\circ\pi\big)
=\left(\nabla^A_{\nabla^L_{j(b)}j(b')},
\big(\nabla^L_{j(b)}j(b')\big)\circ\pi\right) \\
=\left(\nabla^A_{j(\nabla^B_{j(b)}b')},j(\nabla^B_{j(b)}b')\circ\pi\right)
=I_0\big(1\otimes\nabla^B_{j(b)}b'\big)
.\qedhere\end{multline*}
\end{proof}

\begin{lemma}\label{lem:3.21}
Let $s_1,\cdots,s_n$ be sections of $\pi^! L\to A[1]$.
Suppose that each section $s_k$ is $\pi$-related to a section $v_k$ of $L\to M$.
Provided at least one of the sections $s_1,\cdots,s_n$ is null, we have
\[ P_\cU\circ\PBW (s_1\odot\cdots\odot s_n) = 0 .\]
\end{lemma}

\begin{proof}
It follows immediately from Lemma~\ref{eq:BOS} that
$\nabla_{s_i}(s_1\odot\cdots\odot s_{i-1}\odot s_{i+1}\odot\cdots\odot s_n)$
is a sum of products $t_1\odot\cdots\odot t_{n-1}$
of sections of $\pi^! L\to A[1]$ in which each $t_j$ is $\pi$-related
to a section $w_j$ of $L\to M$.
It then follows from Equation~\eqref{eq:PBW} that
$\PBW(s_1\odot\cdots\odot s_n)$ is also a sum of products
of sections of $\pi^! L\to A[1]$,
each of which is $\pi$-related to a section of $L\to M$.
Furthermore, if any of the sections $s_1,\cdots,s_n$ is null,
then $\PBW(s_1\odot\cdots\odot s_n)$ is a sum of products
of sections of $\pi^! L\to A[1]$ all $\pi$-related to sections of $L\to M$
and at least one of them is null.
The conclusion now follows from Equation~\eqref{eq:P_U}.
\end{proof}

\subsubsection{Proof of Proposition~\ref{pro:nabla}}

We are now ready to prove Proposition~\ref{pro:nabla}.

\begin{proof}[Proof of Proposition~\ref{pro:nabla} ]
First we prove by induction the identity
\[ P_\cU\circ\PBW\circ I_S = \id\otimes\pbw ,\]
which is equivalent to Equation~\eqref{eq:P_PBW_I}.
Since both $P_\cU\circ\PBW\circ I_S$ and
$\id\otimes\pbw$ are $\Gamma(\Lambda^\bullet A^\vee)$-linear,
it suffices to verify that
\[ P_\cU\circ\PBW\circ I_S(b_1\odot\cdots\odot b_n)
= \pbw(b_1\odot\cdots\odot b_n) \]
for all $n\in\NN$ and $b_1,\dots,b_n\in\Gamma(B)$.
It is clear that the identity holds for $n=0$ and $n=1$.
Assuming that the identity holds for $n\leq N-1$,
it remains to prove that it holds for $n=N$ as well.
According to ~\eqref{eq:sym}, we have
\begin{align*}
P_\cU\circ\PBW\circ I_S(b_0\odot b_1\odot\cdots\odot b_n)
&= P_\cU\circ\PBW(I_0(b_0)\odot I_0(b_1)\odot\cdots\odot I_0(b_n)) \\
&= \frac{1}{n+1}\sum_{i=0}^n
P_\cU\big(I_0(b_i)\circc\PBW(I_S(b^{\{i\}}))
-\PBW(\nabla_{I_0(b_i)}I_S(b^{\{i\}})) \big)
\end{align*}
for all $b_0,b_1,\dots,b_n\in\Gamma(B)$,

Since the section $I_0(b_i)$ of $\pi^! L\to A[1]$ is $\pi$-related to
the section $b_i$ of $L\to M$ for each $i\in\{0,1,\cdots,n\}$,
it follows from Corollary~\ref{Omacron} and Equation~\eqref{eq:PBW} that
$\PBW(I_S b^{\{i\}})\in\cU(\pi^! L)$ can be rewritten as a sum
of products of sections of $\pi^! L\to A[1]$ that are $\pi$-related to sections
of $L\to M$.
Therefore, it follows from Equation~\eqref{eq:P_U}
and our induction assumption that
\[ P_\cU \big(I_0(b_i)\circc\PBW(I_S (b^{\{i\}})\big)
= j(b_i)\cdot(P_\cU\circ\PBW\circ I_S)(b^{\{i\}})
= j(b_i)\cdot\pbw(b^{\{i\}}) .\]

On the other hand, from Lemma~\ref{lem:3.20} and our induction assumption,
 it follows that
\[ P_\cU\circ\PBW\big(\nabla_{I_0(b_i)}I_S (b^{\{i\}})\big)
= P_\cU\circ\PBW\circ I_S \big(\nabla^B_{j(b_i)}(b^{\{i\}})\big)
= \pbw\big(\nabla^B_{j(b_i)} (b^{\{i\}})\big) .\]
Hence we obtain
\begin{multline*}
P_\cU\circ\PBW\circ I_S(b_0\odot b_1\odot\cdots\odot b_n)
= \frac{1}{n+1}\sum_{i=0}^{n}
\left(j(b_i)\cdot\pbw(b^{\{i\}})
-\pbw\big(\nabla^B_{j(b_i)} (b^{\{i\}})\big)\right) \\
= \pbw(b_0\odot b_1\odot\cdots\odot b_n)
.\end{multline*}

To establish Equation~\eqref{eq:P_PBW_H}, it suffices to prove that
\[ P_\cU\circ\PBW\circ H_S = 0 \]
since $H_S^{\pbw}\circ\PBW=\PBW\circ H_S$ and $\PBW$ is invertible.
Furthermore, since the map $P_\cU\circ\PBW\circ H_S$
is $\Gamma(\Lambda^\bullet A^\vee)$-linear, it suffices to show that
\begin{equation}\label{eq:PUPBWHS}
P_\cU\circ\PBW\circ H_S (s_1\odot\cdots\odot s_n) = 0
\end{equation}
for sections $s_1,\cdots,s_n$ of $\pi^! L\to A[1]$
that are $\pi$-related to sections of $L\to M$.

According to ~\eqref{eq:sym}, we have
\[ H_S(s_1\odot\cdots\odot s_n)
= \sum_{i=1}^n (I_0\circ P_0)(s_1)\odot\cdots\odot(I_0\circ P_0)(s_{i-1})\odot
H_0(s_i)\odot s_{i+1}\odot\cdots\odot s_n .\]
Since $s_1,\dots,s_n$ are sections of $\pi^! L\to A[1]$
that are $\pi$-related to sections of $L\to M$,
it is easy to see that $(I_0\circ P_0)(s_1),\dots,(I_0\circ P_0)(s_{i-1})$ as well
are sections of $\pi^! L\to A[1]$ that are $\pi$-related to sections of $L\to M$.
Moreover, $H_0(s_i)$ is clearly a null section.
Thus Equation~\eqref{eq:PUPBWHS} follows from Lemma~\ref{lem:3.21}.
\end{proof}

\subsubsection{Proof of Theorem~\ref{thm:Genova}}

We are now ready to prove Theorem~\ref{thm:Genova}.

\begin{proof}[Proof of Theorem~\ref{thm:Genova}]
Recall that, according to the homological perturbation lemma,
the perturbations $I_S^{\hp}$ and $H_S^{\hp}$ of $I_S^{\pbw}$
and $H_S^{\pbw}$ are given by the following formulas:
\begin{gather}
I_S^{\hp} = I_{S}^{\pbw} + \sum_{k=1}^\infty \left(H_S^{\pbw} \circ
(D_S^{\pbw}- D_\cU)\right)^k\circ I_{S}^{\pbw} , \nonumber\\
H_S^{\hp} = H_S^{\pbw} + \sum_{k=1}^\infty H_S^{\pbw} \circ
\left( (D_S^{\pbw}-D_\cU) \circ H_S^{\pbw}\right)^k \label{eq:H^hp}
.\end{gather}
Hence the identities $P_\cU \circ I_S^{\pbw} = \id$
and $P_\cU \circ H_S^{\pbw} = 0$ imply that
\begin{equation}\label{eq:Php}
P_\cU \circ I_S^{\hp} = \id
\qquad\text{and}\qquad
P_\cU \circ H_S^{\hp} = 0.
\end{equation}
Finally, we obtain
\begin{multline*}
P_\cU = P_\cU \circ \left(H_S^{\hp} \circ D_\cU
+ D_\cU \circ H_S^{\hp} + I_S^{\hp} \circ P_S^{\hp} \right) \\
= (P_\cU\circ H_S^{\hp})\circ D_\cU
+\dA\circ(P_\cU\circ H_S^{\hp})
+(P_\cU\circ I_S^{\hp})\circ P_S^{\hp} = P_S^{\hp}
,\end{multline*}
where we used Equations~\eqref{eq:Php} together with the facts that Diagram~\eqref{eq:contr_hp} is a contraction
and $P_\cU$ is a cochain map.
Since $P_\cU = P_S^{\hp}$ is a surjective map,
it follows from the identities
\[ P_S^{\hp}\circ D_\cU = \dA^{\hp}\circ P_S^{\hp}
\qquad\text{and}\qquad
P_\cU\circ D_\cU = \dA\circ P_\cU \]
that $\dA = \dA^{\hp}$.

This concludes the proof of Theorem~\ref{thm:Genova}.
\end{proof}

\subsection{Proof of Theorem~\ref{theor:main}}

Finally we are ready to complete the proof of Theorem~\ref{theor:main}.

Let us first recall the following standard homotopy transfer theorem.

\begin{theorem}[{\cite[Theorem~1.3]{MR3276839}},
see Section~12 in loc.~cit.\ for
the $A_\infty$-algebra case]\label{theor:berglund}
\label{thm:Atransfer}
Let $(A,d)$ and $(K,d_K)$ be cochain complexes and let
\[ \begin{tikzcd}
(A,d) \arrow[r, "P", shift left] \arrow["H", loop left] &
(K,d_K) \arrow[l, "I", shift left]
\end{tikzcd},\]
be a contraction of $(A,d)$ onto $(K,d_K)$. Given an $A_\infty$-algebra
structure on $A$ (with $d$ as unary operation), there exists
a `transferred' $A_\infty$-algebra structure on $K$ and a pair
of $A_\infty$-quasi-isomorphisms $\mathbb{I}:K\rightsquigarrow A$
and $\mathbb{P}:A\rightsquigarrow K$ having the cochain maps $I$
and $P$ as respective first Taylor coefficients.
\end{theorem}

\begin{proof}[Proof of Theorem~\ref{theor:main}] 
Consider the contraction \eqref{eq:contraction_S}
in~Theorem~\ref{thm:Genova}.
According to homotopy transfer Theorem~\ref{thm:Atransfer}, we can equip
$\Gamma (\Lambda^\bullet A^\vee )\otimes_R \cU_{L/A}$ with an $A_\infty$-algebra
from the dg associative algebra structure on $\cU(\pi^! L)$.

In order to complete the proof of Theorem~\ref{theor:main},
it remains to show that the $A_\infty$-algebra structures obtained
from different choices of splitting $j : B \to L$ of the short exact
sequence \eqref{eq:splitting1} and different choices
of the $L$-connections $\nabla^A, \nabla^B, \nabla^L$
of Lemma~\ref{lem:FRA} and $T_M$-connection on $A$ 
are all $A_\infty$-isomorphic.

Now let $(P_S^{\mathsf{hp}}, I_S^{\mathsf{hp}}, H_S^{\mathsf{hp}})$
and $({}'\!P_S^{\mathsf{hp}}, {}'\!I_S^{\mathsf{hp}}, {}'\!H_S^{\mathsf{hp}})$
be a pair of contractions of $(\cU(\pi^! L), D_\cU)$
onto $(\Gamma(\Lambda^\bullet A^\vee) \otimes_R \cU_{L/A}, d_A)$
induced by two different choices of splitting and connections.
Denote by $\mathbb{A}$ and $\mathbb{A}'$ the space
$\Gamma(\Lambda^\bullet A^\vee) \otimes_R \cU_{L/A}$
equipped with the two transferred $A_\infty$-algebra
structures respectively.
Use Theorem~\ref{theor:berglund} to promote $I_S^{\mathsf{hp}}$
and ${}'\!P_S^{\mathsf{hp}}$ to $A_\infty$-quasi-isomorphisms
$\mathbb{I} : \mathbb{A} \rightsquigarrow (\cU(\pi^! L),D_\cU)$
and $\mathbb{P}' : (\cU(\pi^! L),D_\cU) \rightsquigarrow
\mathbb{A}'$. The composition $\mathbb{P}' \circ \mathbb{I} : \mathbb{A}
\rightsquigarrow \mathbb{A}'$ is an $A_\infty$-quasi-isomorphism
whose first Taylor coefficient is given by
${}'\!P_S^{\mathsf{hp}} \circ I_S^{\mathsf{hp}}$.
But ${}'\!P_S^{\mathsf{hp}} = P_\cU = P_S^{\mathsf{hp}}$.
Hence the first Taylor coefficient of $\mathbb{P}' \circ \mathbb{I}$
is $P_S^{\mathsf{hp}}\circ I_S^{\mathsf{hp}}=\id$.
As an $A_\infty$-morphism with invertible first Taylor coefficient
is invertible, we conclude that $\mathbb{P}'\circ\mathbb{I}$ is actually an $A_\infty$-isomorphism.

This concludes the proof of Theorem~\ref{theor:main}.
\end{proof}

\section{The dga of a matched pair of Lie algebroids}

\subsection{Universal enveloping algebra of the dg Lie algebroid}

Let $(L, A)$ be a Lie pair. In this section we study
in details the case when the splitting $j : L/A \to L$
of the short exact sequence
\[ 0 \to A \to L \to L/A \to 0 \]
is \emph{integrable}, i.e.~it identifies $L/A$
to a Lie subalgebroid $j(L/A)$ of $L$. In this case,
$B := L/A$ inherits from $j(B)$
the structure of a Lie algebroid, and $(A, B)$ is a 
\emph{matched pair of Lie algebroid}. Recall from 
\cite{MR1460632, MR1716681} that a matched pair of Lie 
algebroids is a pair $(A,B)$, where $A$ and
$B$ are Lie algebroids over the same manifold $M$, equipped with a Lie algebroid structure on their direct sum $A \oplus B$, also denoted $A \bowtie B$, 
such that both the natural inclusions
$i:A\hookrightarrow A\oplus B$, 
$\alpha \mapsto (\alpha, 0)$
and $j:B\hookrightarrow A \oplus B$,
$\beta\mapsto(0,\beta)$ are Lie algebroid morphisms.
In this case, $(A,A\bowtie B)$ is a Lie pair,
and the natural inclusion $j : B \to A \oplus B$ is an 
integrable splitting. Clearly, every Lie pair with an 
integrable splitting arises in this way.

\begin{example}\label{exmpl:action}
Let $\frakg$ be a finite dimensional Lie algebra,
and let $M$ be a $\frakg$-manifold, i.e.~$M$ is equipped
with a $\frakg$-action. Consider the action Lie algebroid
$\frakg \ltimes M$. Then $(\frakg \ltimes M, TM)$
is a matched pair of Lie algebroids in a natural way \cite{MR3650387}.
\end{example}

\begin{example}\label{exmpl:compl_man}
Let $X$ be a complex manifold. As usual, the complexified tangent bundle $T_{\CC} X$ splits as the direct sum $T^{0,1} X \oplus T^{1,0} X$ of the antiholomorphic and the holomorphic tangent bundles (the $\pm i$-eingenbundles of the 
almost complex structure). Then $(T^{0,1} X, T^{1,0} X)$ is a matched pair of complex Lie algebroids and $T^{0,1} X \bowtie T^{1,0} X \cong T_{\CC} X$.
\end{example}

Next we explore the $L_\infty$-algebroid structure
on $A[1] \times_M B \to A[1]$ in the matched pair case.
Let $(A,B)$ be a matched pair of Lie algebroids over a manifold $M$.
Since $A$ and $B$ play symmetric roles as
a pair of complementary Lie subalgebroids of the Lie algebroid $L=A\bowtie B$,
we have a pair of Bott connections: the Bott $A$-connection on
$B$ and the Bott $B$-connection on $A$, both denoted by $\nabla^{\Bott}$
by abuse of notations. The Lie algebroid structure on $L=A\bowtie B$
can be explicitly expressed by those of $A$ and $B$ together with
the Bott-connections:
\begin{equation}
\label{eq:anchor-AB}
\rho_L (\alpha, \beta ) = \rho_A (\alpha) + \rho_B (\beta) 
\end{equation}
and
\begin{equation}
\label{eq:bracket-AB}
[(\alpha_1, \beta_1), (\alpha_2, \beta_2)]_L
= \Big( [\alpha_1, \alpha_2]_A + \nabla^{\mathsf{Bott}}_{\beta_1} \alpha_2
- \nabla^{\mathsf{Bott}}_{\beta_2} \alpha_1, [\beta_1, \beta_2]_B
+ \nabla^{\mathsf{Bott}}_{\alpha_1} \beta_2
- \nabla^{\mathsf{Bott}}_{\alpha_2} \beta_1\Big) ,
\end{equation}
for all $\alpha, \alpha_1,\alpha_2 \in \Gamma (A)$,
and $\beta, \beta_1, \beta_2 \in \Gamma (B)$.

In addition, according to Mackenzie~\cite{MR2831518,}
\begin{equation}
\label{eq:double}
\begin{tikzcd}[column sep=small,row sep=small]
A\oplus B \arrow{r} \arrow{d} & B \arrow{d} \\
A \arrow{r} & M
\end{tikzcd} 
\end{equation}
is a double Lie algebroid.
As a consequence, by a theorem of Voronov~\cite{MR2971727}, 
$(A[1]\oplus B, d^{\Bott}_A )$ is a dg Lie algebroid over $(A[1],\dA)$.
See~\cite[Corollary 4.2]{MR4325718}.

\begin{lemma}
Let $(A, B)$ be a matched pair of Lie algebroid.
Fix the canonical integrable splitting $j : B \to A \bowtie B$
of the short exact sequence $0 \to A \to A \bowtie B \to B \to 0$
of the Lie pair $(A, A \bowtie B)$. Then the $L_\infty$-algebroid structure
on $A[1] \times_M B\to A[1]$ as in Theorem~\ref{prop:L_infty_alg}
agrees with the dg Lie algebroid structure
arising from the double Lie algebroid
\eqref{eq:double} as above.
\end{lemma}
\begin{proof}
We denote by $d_A^{\mathsf{Bott}}$ the homological
vector field on $A[1] \times_M B$ for the dg vector
bundle $A[1] \times_M B\to A[1]$.
Denote, by $\widetilde{d_A^{\mathsf{Bott}}}$, the homological
vector field on $A[1] \times_M B [1]$ by applying the shifting $[1]$-functor
\cite{MR2709144}.
By $d_{\mathsf{Lie}}$, we denote the homological vector field on
$A[1]\oplus B[1]$ induced from the $\ZZ$-graded Lie algebroid
$A[1]\oplus B\to A[1]$, i.e., its Chevalley--Eilenberg differential.
In view of Theorem~\ref{prop:L_infty_alg},
it suffices to show that 
\[ \widetilde{d_A^{\mathsf{Bott}}} + d_{\mathsf{Lie}} = d_L ,\]
where $d_L$ is the homological vector field
on $A[1]\oplus B[1]\cong(A\bowtie B)[1]$
corresponding to the Lie algebroid structure
on $L=A\bowtie B$.
This can be verified directly \cite[Lemma~4.3]{MR4325718}.
\end{proof}

Next we prove the following:

\begin{lemma}\label{lem:Phi_B}
Let $(A,B)$ be a matched pair of Lie algebroids.
Denote $L := A \bowtie B$ and consider the Lie pair $(A,L)$.
The map
\begin{equation}\label{eq:U(B)=U(L/A)}
\Phi : \mathcal{U}(B) \to \mathcal{U}_{L/A},
\quad \beta_1 \cdot \cdots \cdot \beta_k \mapsto \overline{(0,\beta_1)
\cdot \cdots \cdot (0,\beta_k)}, 
\end{equation}
is a well-defined $R$-module isomorphism (here, as usual, the notation $\overline{u}$ 
refers to the equivalence class in
$\mathcal{U}_{L/A}=\frac{\mathcal{U}(L)}{\mathcal{U}(L)\cdot\Gamma(A)}$
of an element $u\in\mathcal{U}(L)$). 
\end{lemma}
\begin{proof}
First of all, $\Phi$ is the composition of the morphism
$\tilde{j}:\mathcal{U}(B)\to\mathcal{U}(L)$ of universal enveloping algebras
determined by the Lie algebroid morphism $j:B\to L$, followed by the projection
$\pr:\mathcal{U}(L)\to\mathcal{U}_{L/A}$, so it is indeed well-defined.
Since both $\tilde{j}$ and $\pr$ are morphisms of $R$-modules, so is $\Phi$.
From~\eqref{eq:bracket-AB}, it is simple to check by induction
that $\mathcal{U}(L)\cong \mathcal{U}(B)\cdot \mathcal{U}(A)$, where $\cdot$ stands for the 
multiplication in $\mathcal{U}(L)$. Thus it follows
that $\Phi$ is surjective and its kernel is isomorphic
to $\mathcal{U}(L) \cdot \Gamma (A)$. Thus the conclusion follows.
\end{proof}

We are now ready to prove the following:

\begin{proposition}\label{theor:dga_MP}
Let $(A,B)$ be a matched pair of Lie algebroids.
Then $\Gamma(\Lambda^\bullet A^\vee)\otimes_R\mathcal{U}(B)$
is a \dga\ in a natural way.
\end{proposition}
\begin{proof}
As $\pi^\ast B = A[1] \times_M B \to A[1]$ is a dg Lie algebroid,
its universal enveloping algebra $\mathcal{U}(\pi^\ast B)$ is a dga.
Clearly, $\mathcal{U}(\pi^\ast B)$ is generated, as a graded algebra,
by $C^\infty (A[1]) = \Gamma (\Lambda^\bullet A^\vee)$ and $\Gamma (B)$.
Also, note that $ \mathcal{U}(B)$ is naturally an $A$-module.
In order to conclude the proof, we will show that there is an isomorphism of
dg modules over $\big(\Gamma (\Lambda^\bullet A^\vee), \ d_A\big)$:
\[ \Psi : \mathcal{U}(\pi^\ast B)
\cong \Gamma (\Lambda^\bullet A^\vee) \otimes_R \mathcal{U}(B) .\]
We define $\Psi$ on generators as
\begin{equation}\label{eq:Psi_gener}
\Psi (\omega) = \omega \quad \text{and} \quad \Psi (\beta) = \beta
\end{equation}
for all $\omega \in \Gamma (\Lambda^\bullet A^\vee)$ and $\beta \in \Gamma (B)$.
It is clear that $\Psi$ is well-defined by~\eqref{eq:Psi_gener}.
It is also clear that $\Psi$ intertwines the differentials.
It remains to check that $\Psi$ is bijective. To do this we argue, first notice that $\Psi$ preserves the filtrations
\[ \cdots \subseteq \mathcal{U}^{k-1} (\pi^\ast B)
\subseteq \mathcal{U}^k (\pi^\ast B) \subseteq \cdots
\subseteq \mathcal{U} (\pi^\ast B) \]
and
\[
\cdots \subseteq \Gamma (\Lambda^\bullet A^\vee) \otimes_R \mathcal{U}^{k-1}(B)
\subseteq \Gamma (\Lambda^\bullet A^\vee) \otimes_R \mathcal{U}^{k}(B)
\subseteq \cdots \subseteq \Gamma (\Lambda^\bullet A^\vee) \otimes_R \mathcal{U}(B)
\]
The rest follows from the following simple remarks:
\begin{itemize}
\item The quotients $\mathcal{U}^k (\pi^\ast B)/ \mathcal{U}^{k-1} (\pi^\ast B)$
and $\Gamma (\Lambda^\bullet A^\vee) \otimes_R \mathcal{U}^{k}(B) /
\Gamma (\Lambda^\bullet A^\vee) \otimes_R \mathcal{U}^{k-1}(B)$
are both canonically isomorphic to $\Gamma (\Lambda^\bullet A^\vee \otimes S^k B)$,
\item the diagrams
\[
\begin{tikzcd}
0 \arrow[r] & \mathcal{U}^{k-1}(\pi^\ast B) \arrow[r] \arrow[d, "\Psi"']
& \mathcal{U}^{k}(\pi^\ast B) \arrow[r] \arrow[d, "\Psi"']
& \Gamma (\Lambda^\bullet A^\vee \otimes S^k B) \arrow[d, equal] \arrow[r] & 0 \\
0 \arrow[r] & \Gamma (\Lambda^\bullet A^\vee) \otimes_R \mathcal{U}^{k-1}(B) \arrow[r]
& \Gamma (\Lambda^\bullet A^\vee) \otimes_R \mathcal{U}^{k}(B) \arrow[r]
& \Gamma (\Lambda^\bullet A^\vee \otimes S^k B) \arrow[r] & 0
\end{tikzcd}
\]
commute.
\end{itemize}

\end{proof}

\subsection{Statement of main theorem}

The main result of this section is the following:

\begin{theorem}\label{theor:A_infty_MP}
Let $(A, B)$ be a matched pair of Lie algebroids and denote $L = A \bowtie B$.
Fix the canonical integrable splitting $j : B \to A \bowtie B$
of the short exact sequence $0 \to A \to A \bowtie B \to B \to 0$
of the Lie pair $(A, A \bowtie B)$. Then the $A_\infty$-algebra structure
on $\Gamma (\Lambda^\bullet A^\vee) \otimes_R \mathcal{U}_{L/A}$
given by Theorem~\ref{theor:main} agrees with the natural dga structure
of Proposition~\ref{theor:dga_MP} (in particular the operations of arity
greater than 3 all vanish), where
$\mathcal{U}_{L/A}$ is identified with $\mathcal{U}(B)$ via the $R$-module isomorphism $\Phi$
in Lemma~\ref{lem:Phi_B}.
\end{theorem}
Before proving Theorem~\ref{theor:A_infty_MP} we need a lemma.

\begin{lemma}\label{lem:ht_dga}
Let $(\mathcal{A},d_\mathcal{A})$ and $(\mathcal{B}, d_\mathcal{B})$ be dgas and let
\begin{equation}\label{eq:ht_dga}
\begin{tikzcd}
(\mathcal{A},d_\mathcal{A}) \arrow[r, "P", shift left] \arrow["H", loop left] &
(\mathcal{B},d_\mathcal{B}) \arrow[l, "I", shift left]
\end{tikzcd},
\end{equation}
be a contraction of cochain complexes with the additional property
that $I : (\mathcal{B},d_\mathcal{B}) \to (\mathcal{A},d_\mathcal{A})$ is
a morphism of dgas.
Then the $A_\infty$-algebra structure induced on $\mathcal{B}$ via homotopy transfer agrees with the preexisting dga structure, i.e.
\begin{enumerate}
\item the binary operation agrees with the preexisting associative product, and
\item all higher operations vanish.
\end{enumerate} 
\end{lemma}
\begin{proof}
Denote by $\star$ both the associative products in $\mathcal{A}$ and $\mathcal{B}$.
Recall, from the homotopy transfer theorem \cite{MR720689, MR2640649, MR3276839},
that the higher operations $\alpha_k : \mathcal{B}^{\otimes k} \to \mathcal{B}$
induced on $\mathcal{B}$ via homotopy transfer, are given by
\[
\alpha_1 = d_\mathcal{B}, \quad \alpha_k := P \circ \beta_k, \quad k \geq 2,
\]
where $\beta_k : \mathcal{B}^{\otimes k} \to \mathcal{A}$ are inductively
 defined by the following relations:
\[
\gamma_1 = -I, \quad \gamma_k = H \circ \beta_k,
\]
and 
\[
\beta_k (b_1, \cdots, b_k) = \sum_{i+j = k}(-1)^{\chi (i, j, b)} \gamma_i (b_1, \cdots, b_i) \star \gamma_j (b_{i+1}, \cdots, b_{i+j})
\]
$b_1, \cdots, b_k \in \mathcal{B}$, with $\chi (i,j,b) = i-1 + (j-1) (|b_1| + \cdots + |b_i|)$. Therefore
\[
\alpha_2 (b_1, b_2) = P \beta_2 (b_1, b_2) = P (\gamma_1 (b_1) \star \gamma_1 (b_2)) = P (I(b_1) \star I(b_2)) = (P \circ I)(b_1 \star b_2) = b_1 \star b_2.
\]
In order to see that the higher operations vanish, it suffices
to show that $\gamma_k = 0$ for all $k \geq 2$. This can be done by induction on $k$. For $k = 2$,
\[
\gamma_2 (b_1, b_2) = (H \circ \beta_2) (b_1, b_2) = (H \circ I) (b_1 \star b_2) = 0.
\] 
The rest is straightforward.
\end{proof}

We are now ready to prove Theorem~\ref{theor:A_infty_MP}.

\begin{proof}[Proof of Theorem~\ref{theor:A_infty_MP}]
Let $(A,B)$ be a matched pair of Lie algebroids,
and let $(A,L)$ be the associated Lie pair, where $L = A \bowtie B$.
According to Lemma~\ref{lem:Phi_B}, we have a
canonical isomorphism $\cU (B)\cong\cU_{L/A}$.
Choose $L$-connection $\nabla^A$, $\nabla^B$, and $\nabla^L$
on $A$, $B$ and $L$ and a linear
connection on $A$ as in Section~\ref{sec:A_infty_LP} and let 
\begin{equation}
\label{eq:Taipei}
\begin{tikzcd}
\Big(\cU(\pi^! L),D_\cU\Big)
\arrow[r, "P_\cU", shift left] \arrow["H_\cU", loop left] &
\Big(\Gamma(\Lambda^\bullet A^\vee)\otimes_R\cU (B),\dA\Big)
\arrow[l, "I_\cU", shift left]
\end{tikzcd}
\end{equation}
be the contraction from Theorem~\ref{thm:Genova} by
identifying $\cU_{L/A}$ with $\cU (B)$ as in Lemma~\ref{lem:Phi_B}.
Recall from the proof of Proposition~\ref{theor:dga_MP} that,
in the present matched pair case, $\pi^\ast B = A[1] \times_M B \to A[1]$
is a dg Lie algebroid and that
$\Gamma(\Lambda^\bullet A^\vee)\otimes_R\cU (B) \cong \mathcal{U}(\pi^\ast B)$.
It is simple to check that the morphism \eqref{eq:I0}:
$I_0 : A[1]\times_M B\to \pi^! L \cong T_{A[1]}\times_{T_M}L$
as defined by~\eqref{eq:I00} is a morphism of dg Lie algebroids
over $A[1]$.
Therefore, it induces a morphism of dgas between their universal enveloping
algebras:
\[ I : \Big(\Gamma(\Lambda^\bullet A^\vee)\otimes_R\cU (B),\dA\Big)
\to \Big(\cU(\pi^! L),D_\cU\Big) .\]

Our next goal is to prove that $I$ coincides with the injection
$I_\cU$ in~\eqref{eq:Taipei}.

It is clear that 
\[ P_\mathcal{U} \circ I = \operatorname{id} .\]
Next, we prove that 
\begin{equation}
\label{eq:HS}
H_S^{\mathsf{pbw}} \circ I = 0,
\end{equation}
where $H_S^{\mathsf{pbw}}$ is the morphism as in~\eqref{eq:Roma}.

Since both $H_S^{\mathsf{pbw}}$ and $I$ are $\Gamma (\Lambda^\bullet A^\vee)$-linear,
it suffices to check that $H_S^{\mathsf{pbw}} \circ I$ vanish on elements of the form
\[ \mathsf{pbw} (b_1 \odot \cdots \odot b_n), \quad b_i \in \Gamma (B) .\]
First of all, we prove, by induction on $n$, that 
with the above choice of connections, the diagram
\begin{equation}\label{diag:I_pbw}
\begin{tikzcd}
\mathcal{U}^{n}(\pi^\ast B) \arrow[d, "I"'] & \Gamma (S^n \pi^\ast B) \arrow[d, "I_S"'] \arrow[l, "\mathsf{pbw}"'] \\
\mathcal{U}^{n}(\pi^! L) & \Gamma (S^n \pi^! L) \arrow[l, "\mathsf{PBW}"']
\end{tikzcd}
\end{equation}
commutes.

It is obvious that it holds when $n=1$. Assume that it holds
for $n=k-1$.
We compute, for any $b_1, \cdots, b_k \in \Gamma (B)$,
\[
I \left(\mathsf{pbw} (b_1 \odot \cdots \odot b_k) \right) = 
\frac{1}{k} \sum_{i = 1}^k \Big(I (j (b_i) \cdot \mathsf{pbw} (b^{\{i\}}) - I \mathsf{pbw} (\nabla^B_{j(b_i)} b^{\{i\}})\Big).
\]
The action of $j(b_i)$ on $\mathsf{pbw} (b^{\{i\}})$ agrees with the multiplication by $b_i \cdot \mathsf{pbw} (b^{\{i\}})$ in $\mathcal{U}(B)$. Hence
\[
\begin{aligned}
I \left(\mathsf{pbw} (b_1 \odot \cdots \odot b_k) \right) & = 
\frac{1}{k} \sum_{i = 1}^k \Big(I(b_i \cdot \mathsf{pbw} (b^{\{i\}})) - I (\mathsf{pbw} (\nabla^B_{j(b_i)} b^{\{i\}}))\Big) \\
& = \frac{1}{k} \sum_{i = 1}^k \Big(I(b_i) \cdot I(\mathsf{pbw} (b^{\{i\}})) - \mathsf{PBW} (I_S(\nabla^B_{j(b_i)} b^{\{i\}}))\Big) \\
& = \frac{1}{k} \sum_{i = 1}^k \Big(I_0(b_i) \cdot \mathsf{PBW} (I_S(b^{\{i\}})) - \mathsf{PBW} (I_S(\nabla^B_{j(b_i)} b^{\{i\}}))\Big), 
\end{aligned}
\]
where we used the induction assumption. It remains to check that
\[
I_S(\nabla^B_{j(b_i)} b^{\{i\}}) = \nabla_{I_0 (b_i)} I_S(b^{\{i\}}).
\]
For this purpose, it indeed suffices to show that, for all $b, b' \in \Gamma (B)$
\[
I(\nabla^B_{j(b)} b') = \nabla_{I_0(b)} I_0(b').
\]
However, we have
\[
I(\nabla^B_{j(b)} b') = I_0 (\nabla^B_{j(b)} b') = \left(\Delta_{\nabla^B_{j(b)} b'}, j(\nabla^B_{j(b)} b') \circ \pi \right).
\]
It is now easy to see that
\[
\Delta_{\nabla^B_{j(b)} b'} = \nabla^{\mathsf{Bott}}_{\nabla^B_{j(b)}b'} = \nabla^A_{\nabla^L_{j(b)}j(b')},
\]
where we used the fact
that $\nabla^A$ extends the Bott $B$-connection
on $A$
and $\nabla^L$ reduces to $\nabla^B$ when being
restricted to $B$ (see
Lemma~\ref{lem:FRA}).
Similarly, we also have
\[
j(\nabla^B_{j(b)} b') = \nabla^L_{j(b)}j(b').
\]
Hence
\[
I(\nabla^B_{j(b)} b') = \left( \nabla^A_{\nabla^L_{j(b)}j(b')}, \  (\nabla^L_{j(b)}j(b')) \circ \pi \right) = \nabla_{(\nabla^A_{j(b)}, j(b) \circ \pi)} \left(\nabla^A_{j(b')}, j(b') \circ \pi \right),
\]
where, in the second step, we used Lemma~\ref{eq:BOS}.
Using again the fact that $\nabla^A$ extends
the Bott $B$-connection on $A$ and agrees with $\Delta_b$,
we have
\[
I(\nabla^B_{j(b)} b') = \nabla_{(\Delta_b, j(b) \circ \pi)} \left(\Delta_{b'}, j(b') \circ \pi \right) = \nabla_{I_0(b)}I_0(b')
\]
as desired. Therefore the diagram \eqref{diag:I_pbw} commutes.
It follows that, for any $b_1,\cdots,b_n\in\Gamma(B)$,
\[
(H_S^{\mathsf{pbw}} \circ I)(\mathsf{pbw} (b_1 \odot \cdots \odot b_n)) = (H_S^{\mathsf{pbw}} \circ \mathsf{PBW})I_S (b_1 \odot \cdots \odot b_n) = (\mathsf{PBW} \circ H_S \circ I_S) (b_1 \odot \cdots \odot b_n) = 0,
\]
where we used the fact that $H_S \circ I_S = 0$.
Thus we conclude that $H_S^{\mathsf{pbw}} \circ I = 0$.

In summary, we have proved that
$I : \Big(\Gamma(\Lambda^\bullet A^\vee)\otimes_R\cU (B),\dA\Big)
\to \Big(\cU(\pi^! L),D_\cU\Big)$ is a dga morphism such that
$P_{\cU} \circ I = \operatorname{id}$ and $H_S^{\mathsf{pbw}} \circ I = 0$.
Then, from~\eqref{eq:H^hp}, it follows that $H_S^{\mathsf{hp}} \circ I = 0$.
Moreover, we have
\[ I^{\mathsf{hp}}=I^{\mathsf{hp}}\circ P^{\mathsf{hp}} \circ I=
(\operatorname{id} - [H^{\mathsf{hp}}, D_{\cU}])\circ I
=I-H^{\mathsf{hp}}\circ D_{\cU}\circ I
=I-H^{\mathsf{hp}} \circ I \circ d_A
=I ,\]
where we also used the fact that $P^{\mathsf{hp}}_S = P_\cU$. The conclusion
 now follows from Lemma~\ref{lem:ht_dga}.
This concludes the proof of Theorem~\ref{theor:A_infty_MP}.
\end{proof}

\begin{remark}
Let $\mathfrak g$ be a finite dimensional Lie algebra, and let $M$ be a $\mathfrak g$-manifold. When $(A,B)$ is the matched pair $(\mathfrak g \ltimes M, TM)$ from Example \ref{exmpl:action}, then $\mathcal U_{L/A} \cong \mathcal U (B) = \mathcal U (TM)$ consists of differential operators on $M$, and $\Gamma (\Lambda^\bullet A^\vee) \otimes_R \mathcal U (B) \cong \Lambda^\bullet \mathfrak g^\vee \otimes \mathcal U (TM)$ is the cochain complex coming from the action of $\mathfrak g$ on differential operators. The $0$-th cohomology $H^0 (\Lambda^\bullet \mathfrak g^\vee \otimes \mathcal U (TM))$ consists of $\mathfrak g$-invariant differential operators, and it inherits an associative algebra structure from the dga structure on $\Lambda^\bullet \mathfrak g^\vee \otimes \mathcal U (TM)$ which actually agrees the usual associative algebra structure of (invariant) differential operators.
\end{remark}

\begin{remark}
Let $X$ be a complex manifold. When $(A, B)$ is the matched pair $(T^{0,1}X, T^{1,0}X)$ from Example \ref{exmpl:compl_man}, then $\mathcal U_{L/A} \cong \mathcal U (B) = \mathcal U (T^{1,0} X)$ consists of sections of the holomorphic vector bundle of differential operators on $X$, and $\Gamma (\Lambda^\bullet A^\vee) \otimes_R \mathcal U (B) = \Omega^{0,\bullet} (X) \otimes \mathcal U (T^{1,0} X)$ is the associated Dolbeault-like cochain complex. The $0$-cohomology $H^0 (\Omega^{0,\bullet} (X) \otimes \mathcal U (T^{1,0} X))$ consists of holomorphic differential operators on $X$, and it inherits an associative algebra structure from the dga structure on $\Omega^{0,\bullet} (X) \otimes \mathcal U (T^{1,0} X)$ which agrees with the usual associative algebra structure of (holomorphic) differential operators.
\end{remark}

\appendix

\section{Dg Lie algebroids associated with Lie algebroid morphisms}
\label{Luca}

The dg Lie algebroid $\pi^!L\xto{\breve{\varpi}}A[1]$
constructed from the morphism of Lie algebroids $\phi:A\to L$
in Section~\ref{doudou} is closely related to
the comma double Lie algebroid \cite{MR4322148},
which is the infinitesimal counterpart
of the comma double Lie groupoid \cite{MR1170713}.

\subsection{The double Lie groupoid \texorpdfstring{$\groupoid{D}$}{D} arising
from a morphism of Lie groupoids \texorpdfstring{$\varphi$}{ϕ}}

We recall a classical construction of double Lie groupoids, \emph{comma
double Lie groupoids},
due to Brown-Mackenzie \cite[Example~1.8]{MR1170713}
(see also \cite[Example~2.5]{MR1174393}) associated to a morphism of
Lie groupoids with the same base manifold.

Let $\varphi$ be a morphism from a Lie groupoid $\groupoid{A}\toto M$
to a Lie groupoid $\groupoid{L}\toto M$:
\[ \begin{tikzcd}
\groupoid{A} \arrow[d, shift left, color=gray] \arrow[d, shift right, color=gray]
\arrow[r, "\varphi"] &
\groupoid{L} \arrow[d, shift left, color=gray] \arrow[d, shift right, color=gray] \\
{\color{gray}M} \arrow[r, "\id_M"', color=gray] & {\color{gray}M}
\end{tikzcd} \]

Let $\groupoid{D}$ be the colimit (in the category of smooth manifolds) of the diagram
\[ \begin{tikzcd} \groupoid{A} \arrow[rd, "s"'] & & \groupoid{L} \arrow[ld, "s"]
\arrow[rd, "t"'] && \groupoid{A} \arrow[ld, "s"] \\
& M & & M & \end{tikzcd} \]
where the symbols $s$ and $t$ denote the source and target maps of the two groupoids at play.

We can think of an element in $\groupoid{D}$
as a triple $(a,\lambda,b)$ with $\lambda\in\groupoid{L}$ and $a,b\in\groupoid{A}$
such that $s(a)=s(\lambda)$ and $t(\lambda)=s(b)$ or
as a square diagram
\[ \begin{tikzcd}
s(a)=s(\lambda) \arrow[d, "a"'] \arrow[r, dashed, "\lambda"] & t(\lambda)=s(b) \arrow[d, "b"] \\
t(a) \arrow[r, dashed, gray, "\varphi(a^{-1})\cdot \lambda\cdot\varphi(b)"'] & t(b)
\end{tikzcd} \]
with horizontal sides labeled by elements of $\groupoid{L}$,
vertical sides labeled by elements of $\groupoid{A}$, and vertices labeled by points of $M$.

We define two partial multiplications on $\groupoid{D}$,
which we call horizontal composition ($\boxvert$):
\[ \begin{tikzcd}
\bullet \arrow[d, "a"'] \arrow[r, dashed, "\lambda"] & \bullet \arrow[d, "b"] \\
\bullet \arrow[r, dashed, gray, "\varphi(a^{-1})\cdot\lambda\cdot\varphi(b)"'] & \bullet
\end{tikzcd}
\quad\boxvert\quad
\begin{tikzcd}
\bullet \arrow[d, "b"'] \arrow[r, dashed, "\lambda'"] & \bullet \arrow[d, "c"] \\
\bullet \arrow[r, dashed, gray, "\varphi({b}^{-1})\cdot\lambda'\cdot\varphi(c)"'] & \bullet
\end{tikzcd}
\quad=\quad
\begin{tikzcd}
\bullet \arrow[d, "a"'] \arrow[r, dashed, "\lambda\cdot\lambda'"] & \bullet \arrow[d, "c"] \\
\bullet \arrow[r, dashed, gray, "\varphi({a}^{-1})\cdot\lambda\cdot\lambda'\cdot\varphi(c)"']
& \bullet
\end{tikzcd} \]
and vertical composition ($\boxminus$):
\[ \begin{tikzcd}
\bullet \arrow[d, "a"'] \arrow[r, dashed, "\lambda"] & \bullet \arrow[d, "b"] \\
\bullet \arrow[r, dashed, gray, "\varphi(a^{-1})\cdot\lambda\cdot\varphi(b)"'] & \bullet
\end{tikzcd}
\quad\boxminus\quad
\begin{tikzcd}
\bullet \arrow[d, "a'"'] \arrow[r, dashed, "\varphi(a^{-1})\cdot\lambda\cdot\varphi(b)"]
& \bullet \arrow[d, "b'"] \\
\bullet \arrow[r, dashed, gray,
"\varphi({a'}^{-1}\cdot a^{-1})\cdot\lambda\cdot\varphi(b\cdot b')"'] & \bullet
\end{tikzcd}
\quad=\quad
\begin{tikzcd}
\bullet \arrow[d, "a\cdot a'"'] \arrow[r, dashed, "\lambda"] & \bullet \arrow[d, "b\cdot b'"] \\
\bullet \arrow[r, dashed, gray,
"\varphi({a'}^{-1}\cdot a^{-1})\cdot\lambda\cdot\varphi(b\cdot b')"'] & \bullet
\end{tikzcd} \]

These two partial multiplications are compatible in the sense that,
for every pattern
\[ \begin{tikzcd}
\bullet \arrow[d] \arrow[r, dashed] \arrow[rd, phantom, "e" description] &
\bullet \arrow[d] \arrow[r, dashed] \arrow[rd, phantom, "\beta" description] &
\bullet \arrow[d] \\
\bullet \arrow[d] \arrow[r, dashed] \arrow[rd, phantom, "\gamma" description] &
\bullet \arrow[d] \arrow[r, dashed] \arrow[rd, phantom, "\delta" description] &
\bullet \arrow[d] \\
\bullet \arrow[r, dashed] & \bullet \arrow[r, dashed] & \bullet
\end{tikzcd} \]
in $\groupoid{D}$, we have
\[ (e\boxvert\beta)\boxminus(\gamma\boxvert\delta)
=(e\boxminus\gamma)\boxvert(\beta\boxminus\delta) .\]

The horizontal composition $\boxvert$ makes $\groupoid{D}$ a Lie groupoid
over $\groupoid{A}$ with source map
$(a,\lambda,b)\mapsto a$ and target map $(a,\lambda,b)\mapsto b$,
while the vertical composition $\boxminus$ makes $\groupoid{D}$ a Lie groupoid
over $\groupoid{L}$ with source map
$(a,\lambda,b)\mapsto\lambda$ and target map
$(a,\lambda,b)\mapsto\varphi(a^{-1})\cdot\lambda\cdot\varphi(b)$.

Indeed, we obtain a double Lie groupoid structure on $\groupoid{D}$
\cite{MR1170713,MR1174393,MR1780095,MR1650045}:
\begin{equation}\label{eq:doublegroupoid}
\begin{tikzcd}
\groupoid{D}
\arrow[r, shift left, "u"] \arrow[r, shift right, "d"']
\arrow[d, shift left, "r"] \arrow[d, shift right, "l"'] &
\groupoid{L}
\arrow[d, shift left, "t"] \arrow[d, shift right, "s"'] \\
\groupoid{A}
\arrow[r, shift left, "s"] \arrow[r, shift right, "t"'] & M
\end{tikzcd}
\end{equation}

We summarize the construction above in a more familiar terminology
in the following

\begin{proposition}[{see \cite[Example~2.5]{MR1174393}}]
\begin{enumerate}
\item As a smooth manifold, $\groupoid{D}$ is diffeomorphic to
$\groupoid{A}\times_{s, M, s}\groupoid{L}\times_{t, M, s}\groupoid{A}$.
\item The vertical groupoid $\groupoid{D}\toto \groupoid{A}$
is the pull back groupoid \cite{MR3453885} of
$\groupoid{L}\toto M$ via the surjective submersion $s: \groupoid{A}\to M$.
\item The horizontal groupoid $\groupoid{D}\toto \groupoid{L}$
is the action groupoid \cite{MR3453885} arising from the right action of
the Cartesian product groupoid
$\groupoid{A}\times \groupoid{A}\toto M\times M$
on $\groupoid{L}$ with the momentum map $J=(s, t): \groupoid{L} \to M\times M$,
with the action $\lambda \cdot (a, b)=\varphi(a^{-1})\cdot\lambda\cdot\varphi(b)$.
\end{enumerate}
\end{proposition}

\subsection{The LA groupoid \texorpdfstring{$\cD$}{D}
arising from the double Lie groupoid \texorpdfstring{$\groupoid{D}$}{D}}

There exists a well-known functor from the category of Lie groupoids and their morphisms
to the category of Lie algebroids and their morphisms,
which we call Lie functor.
According to Mackenzie~\cite{MR1174393,MR1780095},
two LA groupoids can be obtained from a double Lie groupoid
by applying the Lie functor horizontally and vertically, respectively.
Now let us apply the Lie functor to the double Lie groupoid
\eqref{eq:doublegroupoid} horizontally.
Applying the Lie functor to the Lie groupoid $\groupoid{A}\toto M$,
we obtain a Lie algebroid $A\xto{\pi}M$ with anchor $\rho:A\to T_M$.
Likewise, applying the Lie functor to the Lie groupoid $\groupoid{L}\toto M$,
we obtain a Lie algebroid $L\xto{\varpi}M$ with anchor $\varrho:L\to T_M$.
Furthermore, applying the Lie functor to the morphism of Lie groupoids
$\varphi:\groupoid{A}\to\groupoid{L}$, we obtain a morphism of Lie algebroids
$\phi:A\to L$.

Let $\cD$ be the colimit (in the category of smooth manifolds) of the diagram
\[ \begin{tikzcd} A \arrow[rd, "\pi"'] & & \groupoid{L} \arrow[ld, "s"]
\arrow[rd, "t"'] & & A \arrow[ld, "\pi"] \\ & M & & M & \end{tikzcd} \]

We can think of a point in $\cD$ as a triple $(v,\lambda,w)$
with $\lambda\in\groupoid{L}$ and $v,w\in A$ such that $\pi(v)=s(\lambda)$
and $t(\lambda)=\pi(w)$.

Applying the Lie functor to the Lie groupoid $\groupoid{D}\toto\groupoid{L}$
with multiplication $\boxminus$,
we obtain a Lie algebroid structure on $\cD$
with $\cD\xto{p_2}\groupoid{L}$ as underlying vector bundle ---
the projection $p_2$ maps $(v,\lambda,w)$ to $\lambda$.
Its anchor is the map
\begin{equation}\label{eq:Rome}
\cD\ni(v,\lambda,w)\longmapsto
L_{\lambda*}\big(\phi(w)\big) - R_{\lambda*}\big(\phi(v)\big)\in T_{\groupoid{L}} .
\end{equation}
Each pair of sections $v,w\in\Gamma(A\xto{\pi}M)$ determines a section
\[ \groupoid{L}\ni\lambda\longmapsto
\big(v\circ s(\lambda),\lambda,w\circ t(\lambda)\big)\in\cD \]
of $\cD\xto{p_2}\groupoid{L}$, which we denote
$(v\circ s,\id_{\groupoid{L}},w\circ t)$.
The Lie bracket on $\Gamma(\cD\xto{p_2}\groupoid{L})$
is determined by the Lie bracket on $\Gamma(A\xto{\pi}M)$
through the relation
\[ \big[(v\circ s,\id_{\groupoid{L}},w\circ t),
(v'\circ s,\id_{\groupoid{L}},w'\circ t)\big]
=\big(\lie{v}{v'}\circ s,\id_{\groupoid{L}},\lie{w}{w'}\circ t\big) .\]

On the other hand, the partial multiplication
\[ (u,\lambda,v)\boxvert(v,\lambda',w)=(u,\lambda\cdot\lambda',w) \]
on $\cD$ makes it a Lie groupoid over $A$ with
the projection $(v,\lambda,w)\mapsto v$ as its source map $p_1$
and the projection $(v,\lambda,w)\mapsto w$ as its target map $p_3$.

Indeed, $\cD$ admits an LA-groupoid structure
in the sense of Mackenzie~\cite{MR1174393}:
\[ \begin{tikzcd}
\cD \arrow[d, shift left, "p_1"] \arrow[d, shift right, "p_3"'] \arrow[r, "p_2"] &
\groupoid{L} \arrow[d, shift left, "s"] \arrow[d, shift right, "t"'] \\
A \arrow[r, "\pi"'] & M
\end{tikzcd} \]

In summary, as a smooth manifold, $\cD$ is diffeomorphic to
$A\times_{M, s}\groupoid{L}\times_{t, M}A$.
The vertical groupoid $\cD\toto A$
is the pull back groupoid \cite{MR3453885} of
$\groupoid{L}\toto M$ via the surjective submersion $\pi: A\to M$.
The horizontal Lie algebroid $\cD\to\groupoid{L}$
is the action Lie algebroid \cite{MR3453885} arising from the right action of
the Cartesian product Lie algebroid
$A\times A\to M\times M$
on $\groupoid{L}$ given by~\eqref{eq:Rome}.

Similarly, one obtains a second LA-groupoid by applying
the Lie functor to the vertical Lie groupoids on~\eqref{eq:doublegroupoid}:
\begin{equation}\label{eq:LA2}
\begin{tikzcd}
\mathcal{D'} \arrow[r, shift left, "p_1"] \arrow[r, shift right, "p_3"'] \arrow[d, "p_2"'] &
L \arrow[d, "\varpi"] \\
\groupoid{A} \arrow[r, shift left, "s"] \arrow[r, shift right, "t"'] & M
\end{tikzcd}
\end{equation}

As a manifold $\mathcal{D'}$ is diffeomorphic to
$T_\groupoid{A}\times_{s_*, T_M}{L}$.
The horizontal groupoid $\mathcal{D'}\toto L$
is the action groupoid \cite{MR3453885} corresponding to the right action
of the tangent Lie groupoid $T_\groupoid{A}\toto T_M$
on $L$ with the momentum map $\varrho: L\to T_M$ ---
see~\cite[Example~4.7]{MR1174393}.
The vertical Lie algebroid $\mathcal{D'}\to \groupoid{A}$
is the pull-back Lie algebroid \cite{MR3453885} of $L$
via the surjective submersion $s: \groupoid{A} \to M$.
See~\cite[Example~4.7]{MR1174393} for more details.

\subsection{The double Lie algebroid \texorpdfstring{$D$}{D} arising from the LA groupoid \texorpdfstring{$\cD$}{D}}

Note that the Lie groupoid $\cD\toto A$ is the colimit
(in the category of Lie groupoids and their morphisms) of the diagram
\begin{equation}\label{eq:Dessel} \begin{tikzcd}
\groupoid{L} \arrow[rd, "{(s,t)}"] \arrow[dd, shift left, "t", color=gray]
\arrow[dd, shift right, "s"', color=gray] &
& A\times A \arrow[ld, "\pi\times\pi"'] \arrow[dd, shift left, "p_2", color=gray]
\arrow[dd, shift right, "p_1"', color=gray] \\
& M\times M \arrow[dd, shift left, "p_2", color=gray]
\arrow[dd, shift right, "p_1"', color=gray] & \\
{\color{gray}M} \arrow[rd, "\id_M"', color=gray] & & {\color{gray}A}
\arrow[ld, "\pi", color=gray] \\
& {\color{gray}M} &
\end{tikzcd} \end{equation}

Consider the diagram in the category of Lie algebroids
\[ \begin{tikzcd} L \arrow[rd, "\varrho"'] && T_A \arrow[ld, "\pi_*"] \\
& T_M & \end{tikzcd} \]
obtained by applying the Lie functor to Diagram~\eqref{eq:Dessel}.
Its colimit is the Lie algebroid $D\xto{\breve{\varpi}}A$, the image
of the Lie groupoid $\cD\toto A$ (with multiplication $\boxvert$)
under the Lie functor.
We can think of an element in $D$ as a pair $(X, v)$ with $X\in T_A$
and $v\in L$ such that $\varrho(v)=\pi_*(X)$.
The projection $\breve{\varpi}$ is the canonical map $D\to T_A\to A$.

The anchor of this Lie algebroid structure on $D$ is the map
\[ D\ni(v,X)\longmapsto X\in T_A .\]
Each pair $(v,X)$ consisting of
(1) a section $v$ of the Lie algebroid $L\xto{\varpi}M$
and (2) a vector field $X\in\XX(A)$ that is $\pi_*$-projectable
onto the vector field $\varrho(v)\in\XX(M)$ determines a section
\[ A\ni a\longmapsto (X_a, v\circ\pi(a))\in D \]
of the Lie algebroid $D\xto{\breve{\varpi}}A$,
which we denote $(X,v\circ\pi)$.
The Lie bracket on $\Gamma(D\xto{\breve{\varpi}}A)$ is determined by the
Lie bracket on $\Gamma(L\xto{\varpi}M)$ through the relation
\[ \big[(X, v\circ\pi),( X', v'\circ\pi)\big]=
\big( \lie{X}{X'}, \lie{v}{v'}\circ\pi\big) .\]
We note that the Lie algebroid $D\xto{\breve{\varpi}}A$ is the Lie algebroid pullback
--- see~\cite{MR2157566,MR4322148} --- of the Lie algebroid $L\xto{\varpi}M$
through the surjective submersion $\pi:A\to M$.

On the other hand, $D$ inherits a second Lie algebroid structure
by applying the Lie functor to the
horizontal Lie groupoid $\mathcal{D'}\toto {L}$ in~\eqref{eq:LA2}.
It is the action Lie algebroid corresponding to the
right infinitesimal action of the tangent Lie
algebroid $T_A\to T_M$ on $\varrho_L: L\to T_M$
--- see~\cite[Section~4.1]{MR4322148}.

To define the action, note that $\Gamma(T_M,T_A)$ is generated,
over $C^\infty(T_M)$, by two types of sections: core sections $\Hat{X}$
and tangent sections $TX$, for all $X\in\Gamma(A)$
\cite{MR1262213,MR1617335}.

Recall that the core section $\Hat{X}\in\Gamma(T_M,T_A)$,
for any $X\in\Gamma(A)$, is defined as the map
\[ \Hat{X}: T_M\to T_A,
\quad \Hat{X} (v_m)=v_m +X|_m \in T_M|_m \oplus A|_m \cong T_A|_{0_m},
\quad \forall v_m\in T_M|_m .\]
Indeed their brackets completely determine the Lie bracket
on $\Gamma(T_M,T_A)$ \cite[Equation~(27)]{MR1262213}:
\[ \lie{TX}{TY}=T(\lie{X}{Y}) ,\qquad \lie{TX}{\Hat{Y}}=\hat{[X, Y]} ,\qquad \lie{\Hat{X}}{\Hat{Y}}=0 ,\]
for all $X,Y\in\Gamma(A)$.

Every section $X\in\Gamma(A)$ determines two vector fields on $A$:
the vertical lift vector field $X^\upa$
and the morphic vector field $\Tilde{X}$,
which are defined respectively by
\begin{align}
X^\upa(f\circ \pi) &= 0, &
X^\upa(\ell_\xi) &= \langle\xi, X\rangle\circ \pi , \label{eq:Jotz1} \\
\Tilde{X}(f\circ \pi) &= \rho (X)(f)\circ \pi, &
\Tilde{X}(\ell_\xi) &= \ell_{L_X(\xi)}, \label{eq:Jotz2}
\end{align}
for $f\in\cinf(M)$ and $\xi\in\Gamma(A^\vee)$ --- see~\cite{MR1617335}.
Here $\ell_\xi\in\cinf(A)$ is the fiberwise linear function determined by $\xi$.

Define $\Phi: \Gamma (T_M, T_A) \to \XX(L)$ by
\begin{equation}\label{eq:Jotz}
\Phi (TX)=\TTilde{\phi(X)}, \qquad\qquad \Phi (\Hat{X})=\phi(X)^\upa
\end{equation}
for any $X\in\Gamma(A)$.

One proves that Equation~\eqref{eq:Jotz} defines uniquely
an action of the tangent Lie algebroid $T_A\to T_M$
on $\varrho: L\to T_M$ \cite{MR4322148}. Thus
one can form the transformation Lie algebroid
$D=T_A \times_{\pi_*,T_M, \varrho} L \xto{\breve{\pi}}L$.

\begin{proposition}[\cite{MR4322148}]
\label{pro:Pune}
The double vector bundle $D$, equipped with the
Lie algebroid structures described above 
is a double Lie algebroid:
\begin{equation}\label{eq:Pune} \begin{tikzcd}
D \arrow[r, "\breve{\pi}"] \arrow[d, "\breve{\varpi}"'] & L \arrow[d, "\varpi"] \\
A \arrow[r, "\pi"'] & M
\end{tikzcd} \end{equation}
\end{proposition}
Such a double Lie algebroid is often called
\emph{comma double Lie algebroid} --- see~\cite{MR4322148}.
A similar construction is found in~\cite{MR4421028}.

\begin{remark}
Note that comma double Lie algebroids described above
always exist for any Lie algebroid  morphism $\phi: A\to L$ 
over a field $\KK$ of characteristic zero.
When $\KK$ is the real $\RR$, any Lie algebroid
morphism $\phi: A\to L$ corresponds to
a morphism of local Lie groupoids, so
the comma double Lie algebroid can be obtained
as the infinitesimal of a comma double local Lie groupoid.
However, when $\KK$ is the complex $\CC$, the groupoid
picture no longer exist. But there always exists
a comma double Lie algebroid, and 
all the structure maps of this comma double Lie algebroid
can be explicitly described in terms of
Lie algebroid data as discussed above.
\end{remark}

\subsection{The dg Lie algebroid \texorpdfstring{$D_{\breve{\pi}}[1]$}{Dpibreve[1]}
arising from the double Lie algebroid \texorpdfstring{$D$}{D}}

One important class of dg Lie algebroids arise
from Mackenzie's \emph{double Lie algebroids}
\cite{arXiv:math/9808081,MR1650045},
the infinitesimal counterparts of double Lie groupoids.
The following theorem is essentially due to Voronov~\cite{MR2971727}
and was reformulated in terms of representations
up to homotopy in~\cite{MR3802797}.

\begin{theorem}\label{Manaus}
A double vector bundle structure
\begin{equation}\label{eq:Pune1} \begin{tikzcd}
D \arrow[r, "\breve{\pi}"] \arrow[d, "\breve{\varpi}"'] & L \arrow[d, "\varpi"] \\
A \arrow[r, "\pi"'] & M
\end{tikzcd} \end{equation}
can be upgraded to a double Lie algebroid
in the sense of Mackenzie~\cite{MR1174393,MR1780095,MR1650045}
if and only if $D_{\breve{\pi}}[1]\to A_\pi [1]$ is a dg Lie algebroid.
Here $D_{\breve{\pi}}[1]$ and $A_{\pi}[1]$ denote the graded manifolds obtained by shifting the degree
of the fiberwisely linear functions on the total spaces $D$ and $A$
of the vector bundles $D\xto{\breve{\pi}}L$ and $A\xto{\pi}M$.
\end{theorem}

We are now ready to state the main result of this Appendix.

\begin{theorem}
\label{thm:Cortona}
Let $A \xto{} M$ and $L \xto{} M$ be Lie algebroids over $\KK$
with the same base manifold $M$. Let $\phi:A\to L$ be a Lie algebroid
morphism over the identity map $\id_M:M\to M$.
The dg Lie algebroid corresponding to the comma double Lie
algebroid \eqref{eq:Pune} in Proposition~\ref{pro:Pune} coincides with the one in Theorem~\ref{Zug}.
\end{theorem}

To prove Theorem~\ref{thm:Cortona}, note that
the Chevalley--Eilenberg differential for the Lie algebroid
$A\xto{\pi}M$ (with anchor $\rho:A\to T_M$)
is a homological vector field $Q=\dA$ on the $\ZZ$-graded manifold $A_{\pi}[1]$.
The algebra of functions on $A_{\pi}[1]$ is generated by
$\pi^*\big(C^{\infty}(M)\big)$ in degree $0$
and the $\pi$-fiberwisely linear functions on $A$, which are assigned degree $1$.
Hence $(A_{\pi}[1],Q)$ is a dg manifold.

Likewise, the Chevalley--Eilenberg differential for the Lie algebroid
$D\xto{\breve{\pi}}L$ (with anchor $\breve{\rho}:D\to T_L$)
is a homological vector field $\breve{Q}$ on the $\ZZ$-graded manifold $D_{\breve{\pi}}[1]$.
The algebra of functions on $D_{\breve{\pi}}[1]$ is generated by
$\breve{\pi}^*\big(C^{\infty}(L)\big)$ in degree $0$
and the $\breve{\pi}$-fiberwisely linear functions on $D$, which are assigned degree $1$.
Hence $(D_{\breve{\pi}}[1],\breve{Q})$ is a dg manifold.

As a consequence of Proposition~\ref{pro:Pune} and Theorem~\ref{Manaus},
we obtain a dg vector bundle
\begin{equation}\label{eq:Hull} \begin{tikzcd}
(D_{\breve{\pi}}[1],\breve{Q}) \arrow[d, "\breve{\varpi}"] \\ (A_{\pi}[1],Q)
\end{tikzcd} \end{equation}

The homological vector fields $\breve{Q}$ and $Q$ induce an operator $\cQ$ of degree $+1$
on $\Gamma(D_{\breve{\pi}}[1]\xto{\breve{\varpi}}A_{\pi}[1])$ satisfying $\cQ^2=0$
through the relation
\begin{equation}
\label{eq:SKG}
Q\big(\duality{\varepsilon}{e}\big)
=\duality{\breve{Q}(\varepsilon)}{e}+(-1)^{\degree{\varepsilon}}
\duality{\varepsilon}{\cQ(e)}
\end{equation}
for $e\in\Gamma(D_{\breve{\pi}}[1]\xto{\breve{\varpi}}A_{\pi}[1])$ and
$\varepsilon\in\Gamma((D_{\breve{\pi}}[1])^\vee
\xto{\breve{\varpi}}A_{\pi}[1])\subset C^{\infty}_{\text{linear}}(D_{\breve{\pi}}[1])$.
Note that $\breve{Q}$ is a linear vector field \cite{MR2709144}
with respect to the vector bundle $D_{\breve{\pi}}[1]\xto{\breve{\varpi}}A_{\pi}[1]$.

For all $f\in C^\infty(A_{\pi}[1])$ and
$e\in\Gamma(D_{\breve{\pi}}[1]\xto{\breve{\varpi}}A_{\pi}[1])$, we have
\[ \cQ(f\cdot e)=Q(f)\cdot e+(-1)^{\degree{f}}f\cdot\cQ(e) .\]
Thus, $\big(\Gamma(D_{\breve{\pi}}[1]\xto{\breve{\varpi}}A_{\pi}[1]),\cQ\big)$
is a dg module over the \dga\ $\big(C^\infty(A_{\pi}[1]),Q\big)$.

The double Lie algebroid structure in 
Proposition~\ref{pro:Pune} guarantees that the Lie algebroid structure on~\eqref{eq:Hull}
is compatible with the dg vector bundle structure
encoded in the homological vector field $Q=\dA$ and the operator $\cQ$.
Therefore, the morphism of Lie algebroids $\phi:A\to L$ 
determines a dg Lie algebroid structure on $D_{\breve{\pi}}[1]\xto{\breve{\varpi}}A_\pi[1]$.
It is much easier to describe $\cQ$ than $\breve{Q}$.
Since the graded manifold $D_{\breve{\pi}}[1]$ is the fibered product
\[ \begin{tikzcd} D_{\breve{\pi}}[1] \arrow[d, dashed] \arrow[r, dashed]
& T_{A_{\pi}[1]} \arrow[d, "\pi_*"] \\
L \arrow[r, "\varrho"'] & T_M \end{tikzcd} ,\]
a section $s$ of the Lie algebroid
\begin{equation}\label{eq:Newcastle} \begin{tikzcd}
D_{\breve{\pi}}[1] \arrow[d, "\breve{\varpi}"'] \\
A_{\pi}[1] \arrow[u, bend right, color=gray, "s_\phi"']
\end{tikzcd} \end{equation}
is a pair $s=(X,\mu)$ consisting of a map $\mu:A_{\pi}[1]\to L$
and a vector field $X\in\XX(A_{\pi}[1])$ such that
$\varrho\circ\mu=\pi_*\circ X$.
The Lie algebroid \eqref{eq:Newcastle} admits a canonical section:
\[ s_{\phi}=( Q=\dA, A_{\varpi}[1]\to A\xto{\phi}L) .\]

Thus to prove Theorem~\ref{thm:Cortona},
it suffices to prove the following

\begin{proposition}
\label{pro:Naples}
The operator $\cQ$ acting on sections of the dg vector bundle \eqref{eq:Hull}
admits the following very neat description:
\begin{equation}
\label{eq:Porto}
\cQ(s)=\lie{s_\phi}{s} ,
\end{equation}
where $\lie{\argument}{\argument}$ denotes the Lie bracket
on sections of the Lie algebroid \eqref{eq:Newcastle}.
\end{proposition}

\subsection{Proof of Proposition~\ref{pro:Naples}}

This subsection is devoted to the proof of
Proposition~\ref{pro:Naples}.
Choose local coordinates $(x^1,\cdots,x^n)$ on $M$,
a local frame $(e_1,\cdots,e_m)$ for $A$, and a local frame
$(f_1, \cdots, f_t)$ for $L$.
They induce a local chart $(x^1,\cdots,x^n; \zeta^1, \cdots,
\zeta^t; \xi^1,\cdots,\xi^m, \eta^1, \cdots, \eta^m)$ on
\begin{equation}
\label{eq:DLAA1}
D_{\breve{\pi}}[1] \xto{\cong} L\times_M A[1] \times_M A[1], 
\end{equation}
where the middle component is the base of the vector bundle
$D_{\breve{\pi}}[1]\to A[1]$.

Recall that under this local chart, $d_A$ is
given by Equation~\eqref{eq:dA}, while $X_i\in\XX(A[1])$
is given by Equation~\eqref{eq:Xi}.
By $\ef{i}$, $i=1,\cdots,t$, and $\ee{i}$, $i=1,\cdots,m$,
we denote the (local) section of $D_{\breve{\pi}}[1]\to A[1]$ 
corresponding to the pullback sections $f_i \in \Gamma (L)$
and $e_i \in \Gamma (A[1])$, respectively,
under the identification \eqref{eq:DLAA1}.
Here, by abuse of notation, we use the same symbol
to denote a section of $A$ and the corresponding section of $A[1]$.

We need a few lemmas first.
The following lemma can be easily proved by a direct 
verification using Equation~\eqref{eq:sphi}.

\begin{lemma}
\label{lem:A5}
We have
\begin{enumerate}
\item $\lie{s_{\boldphi}}{\ef{l}}
= \sum_i \xi^i \big(\lie{X_i}{\rho (f_l)}, \ \lie{\phi(e_i)}{f_l}\circ\pi \big), \ \ \ \forall l=1, \cdots , t;$
\item
$\lie{s_{\boldphi}}{\ee{l}}
= \sum_i \xi^i (\lie{X_i}{e_l^{\uparrow}}, 0) -
\sum_i \pairing{\xi^i}{e_l}(X_i, \phi(e_i)\circ\pi), \ \ \ \forall l=1, \cdots , m.$
\end{enumerate}
\end{lemma}

\begin{lemma}
For all $i=1, \cdots, t$, we have
\begin{equation}
\label{eq:Naples}
\breve{Q} (\zeta^i)=
\sum_j \big(\liederivative{\phi (e_j)}\zeta^i) \xi^j+ \duality{\zeta^i}{\phi (e_j)}
\eta^j\big).
\end{equation}
\end{lemma}
\begin{proof}
We have
\begin{align*}
\duality{\breve{Q}(\zeta^i)}{Te_j} &= \Phi(Te_j)(\zeta^i)
&& \text{by Equation~\eqref{eq:Jotz}} \\
&= \TTilde{\phi(e_j)}(\zeta^i)
&& \text{by Equation~\eqref{eq:Jotz2}} \\
&= \liederivative{\phi(e_j)}\zeta^i
\in\Gamma (L^\vee)\cong C^{\infty}_{\text{linear}}(L) . &&
\end{align*}
Here we identify $\Gamma(L^\vee)$ with the space of linear functions on $L$.

On the other hand, according to Equations~\eqref{eq:Jotz1}-\eqref{eq:Jotz}, we have
\[ \duality{\breve{Q}(\zeta^i)}{\Hat{e_j}} = \Phi(\Hat{e_j})(\zeta^i)
= \phi( e_j)^\upa(\zeta^i) = \duality{\zeta^i}{\phi(e_j )} . \]

Equation~\eqref{eq:Naples} thus follows immediately.
\end{proof}

\begin{lemma}
\label{lem:4.7}
For any $i=1,\cdots,m$, we have
\[ \breve{Q} (\eta^i)=
-\half \sum_{s, j, \lambda}\frac{\partial c^{i}_{j \lambda} (x)}{\partial x^s}
\xi^j \xi^{\lambda} y^s
-\sum_{s, \lambda} c_{s \lambda}^i (x)\xi^{\lambda}\eta^s , \]
where $(y^1,\cdots,y^n)$ are the induced vertical local 
coordinates on $T_M$ corresponding to the local
coordinates $(x^1,\cdots,x^n)$ on $M$.
\end{lemma}

We need the following lemma concerning complete (tangent) lift vector fields first.

\begin{lemma}
\label{lem:lift}
Let $\cM$ be a graded manifold with local chart $(z^1,\cdots,z^n)$
and $(z^1, \cdots , z^n; \mu^1, \cdots , \mu^n)$ be
its corresponding coordinates on $T_\cM$.
For any vector field $X=\sum_i V^i (z) \frac{\partial}{\partial z^i}$, we have
\[ \widetilde{X}= X+\sum_{i, s} \frac{\partial V^i(z)}{\partial z^s}
\mu^s \frac{\partial}{\partial \mu^i} .\]
\end{lemma}

\begin{proof}[Proof of Lemma~\ref{lem:4.7}]
By $\widetilde{\dA}$,
we denote the complete (tangent) lift of the homological
vector field $\dA$ on $T_{A[1]}$. By construction, we have
\[ \breve{Q} (\eta^i)=\pr^* ( \widetilde{\dA} (\eta^i)) ,\]
where $\pr: T_{A[1]}\times_{T_M}L\to T_{A[1]}$ is the
natural projection.

Thus, by Lemma~\ref{lem:lift}, we have
\begin{align*}
\breve{Q} (\eta^i) &= \pr^*(\widetilde{\dA}(\eta^i)) \\
&= \sum_i \big( \widetilde{ -\frac{1}{2}\sum_{i,j,k} c^k_{ij}(x) \xi^i\xi^j
\frac{\partial}{\partial \xi^k} } \big) (\eta^i) \\
&= \sum_s \frac{\partial}{ \partial x^s} \big(-\half \sum_{j \lambda}
c_{j \lambda}^i (x) \xi^j \xi^{\lambda}\big) y^s
+\sum_s \frac{\partial}{ \partial \xi^s} \big(-\half \sum_{j \lambda}
c_{j \lambda}^i (x) \xi^j \xi^{\lambda}\big)\eta^s \\
&= -\half \sum_{s, j, \lambda} \frac{\partial c_{j \lambda}^i (x)}{ \partial x^s}
\xi^j \xi^\lambda y^s-\sum_{s \lambda}c_{s\lambda}^i (x) \xi^\lambda\eta^s
.\qedhere\end{align*}
\end{proof}

\begin{proof}[Proof of Proposition~\ref{pro:Naples}]
Let $\breve{Q}$ be the operator as in Equation~\eqref{eq:Porto}.
It suffices to prove that Equation~\eqref{eq:SKG} must hold for
any sections
$e\in\Gamma(D_{\breve{\pi}}[1]\xto{\breve{\varpi}}A_{\pi}[1])$ and
$\varepsilon\in\Gamma((D_{\breve{\pi}}[1])^\vee
\xto{\breve{\varpi}}A_{\pi}[1])$. It is simple to check that,
if Equation~\eqref{eq:SKG} holds for a pair of sections
$(e,\varepsilon)$, it must hold for $(f e, g \varepsilon)$,
where $f,g\in C^\infty(A_{\pi}[1])$ are any functions.
As a consequence, it suffices to prove that
Equation~\eqref{eq:SKG} holds for any local frames of the vector bundles
$D_{\breve{\pi}}[1]\xto{\breve{\varpi}}A_{\pi}[1]$ and
$(D_{\breve{\pi}}[1])^\vee\xto{\breve{\varpi}}A_{\pi}[1]$. 

By Equation~\eqref{eq:Naples}, we have
\begin{equation*}
\pairing{\breve{Q}(\zeta^i)}{\ef{l}}=
\sum_j \pairing{\liederivative{\phi(e_j)}(\zeta^i)}{\ef{l}} \xi^j
=-\sum_j \pairing{\zeta^i}{\lie{\phi(e_i)}{f_l}}\xi^j .
\end{equation*}
On the other hand,
by Lemma~\ref{lem:A5}, we have 
\begin{equation*}
\pairing{\zeta^i}{\cQ (\ef{l})}=
\pairing{\zeta^i}{\lie{s_{\boldphi}}{\ef{l}}}=
\sum_j \xi^j\pairing{\zeta^i}{\lie{\phi(e_i)}{f_l}} .
\end{equation*}
Hence, it follows that
\[ \pairing{\breve{Q}(\zeta^i)}{\ef{l}}+\pairing{\zeta^i}{\cQ (\ef{l})}=0 .\]

Now, according to Lemma~\ref{lem:A5}, we have
\begin{equation*}
\pairing{\zeta^i}{\cQ (\ee{l})}=
\pairing{\zeta^i}{\lie{s_{\boldphi}}{\ee{l}}}
=-\sum_j \pairing{\xi^i}{e_l}\pairing{\zeta^i}{\phi (e_j)}.
\end{equation*}

On the other hand, Equation~\eqref{eq:Naples} implies that
\begin{equation*}
\pairing{\breve{Q}(\zeta^i)}{\ee{l}}= \sum_j
\pairing{\zeta^i}{ \phi (e_j)} \pairing{\xi^i}{ e_l}
=\sum_j \pairing{\xi^i}{e_l}\pairing{\zeta^i}{\phi (e_j)} .
\end{equation*}
Hence, it follow that 
\[ \pairing{\breve{Q}(\zeta^i)}{\ee{l}}+\pairing{\zeta^i}{\cQ (\ee{l})}=0 .\]

According to Lemma~\ref{lem:4.7}, we have
\begin{equation*}
\pairing{\breve{Q}(\eta^i)}{\ef{l}}=
-\half \sum_{ j \lambda} \rho (f_l) (c^{i}_{j \lambda} (x))
\xi^j \xi^{\lambda} .
\end{equation*}

By Lemma~\ref{lem:A5}, we have
\begin{equation*}
\pairing{\eta^i}{\cQ (\ef{l})}=
\pairing{\eta^i}{\lie{s_{\boldphi}}{\ef{l}} }= 
\sum_j \xi^j \lie{X_j}{\rho (f_l)}(\xi^i)
=\half \sum_{ j \lambda} \rho (f_l) (c^{i}_{j \lambda} (x))
\xi^j \xi^{\lambda}. 
\end{equation*}

Hence it follows that
\[ \pairing{\breve{Q}(\eta^i)}{\ef{l}}+\pairing{\eta^i}{\cQ (\ef{l})}=0 .\]

By Lemma~\ref{lem:4.7}, we have
\begin{equation*}
\pairing{\breve{Q}(\eta^i)}{\ee{l}}=
-\sum_j c_{l j}^i(x) \xi^j =\sum_j c_{ jl}^i(x) \xi^j.
\end{equation*}

According to Lemma~\ref{lem:A5}, we have
\begin{align*}
\pairing{\eta^i}{\cQ (\ee{l})}
&= \pairing{\eta^i}{\lie{s_{\boldphi}}{\ee{l}}} \\
&= \sum_j \big(\xi^j \lie{X_j}{e_l^\upa} (\xi^i)
-\pairing{\xi^j}{e_l}X_j (\xi^i)\big) \\
&= -\half \sum_j \xi^j c_{jl}^i (x)- \sum_j \delta_l^j \cdot \sum_\lambda
\half c_{j \lambda}^i (x) \xi^\lambda \\
&= -\sum_j c_{jl}^i (x) \xi^j
.\end{align*}

Hence, it follows that
\[ \pairing{\breve{Q}(\eta^i)}{\ee{l}}+\pairing{\eta^i}{\cQ (\ee{l})}=0 .\]

This concludes the proof of Proposition~\ref{pro:Naples}. 
\end{proof}

\section{A useful formula and its proof}
\label{trenitalia}

Let $A\to M$ be a Lie algebroid.
Let $\interior{a}$ and $\liederivative{a}$ denote the interior product and the Lie derivative w.r.t.\ $a\in\Gamma(A)$ respectively, and let $\dA$ denote the
Chevalley--Eilenberg differential.

Choose any pair of dual (local) frames $\{e_k\}_{k=1,\cdots,n}$ and $\{\xi^k\}_{k=1,\cdots,n}$ for $A$ and $A^\vee$ respectively.
Then $\sum_{k=1}^n \xi^k\cdot\interior{e_k}$ is a (local) derivation
of the graded algebra $\Gamma(\Lambda^\bullet A^\vee)$, i.e.\ a (locally defined)
vector field on the graded manifold $A[1]$.

The purpose of this appendix is to prove the following
useful formula.

\begin{proposition}\label{siccita}
We have
\begin{equation}\label{eq:italo}
\dA=\sum_{k=1}^n \big( \xi_k\cdot\liederivative{e_k}
- (\dA\xi_k)\cdot\interior{e_k} \big)
.\end{equation}
\end{proposition}

The following lemma is immediate.

\begin{lemma}\label{rats}
For all $\omega\in\Gamma(\Lambda^p A^\vee)$, we have
\[ \sum_{k=1}^n \xi^k\cdot\interior{e_k} (\omega) = p\cdot\omega .\]
\end{lemma}

\begin{proof}[Proof of Proposition~\ref{siccita}]
Suppose as above that $\omega\in\Gamma(\Lambda^p A^\vee)$.
It follows from Lemma~\ref{rats} that
\[ \begin{split}
p\cdot \dA\omega &= \sum_k \dA(\xi_k\cdot\interior{e_k}\omega) \\
&= \sum_k \big( (\dA\xi^k)\cdot\interior{e_k}\omega
-\xi_k\cdot \dA(\interior{e_k}\omega) \big) \\
&= \sum_k \big( (\dA\xi^k)\cdot\interior{e_k}\omega
-\xi_k\cdot (\liederivative{e_k}\omega-\interior{e_k}\dA\omega) \big) \\
&= \sum_k \big( (\dA\xi^k)\cdot\interior{e_k}\omega
-\xi_k\cdot \liederivative{e_k}\omega \big)
+\sum_k \xi_k\cdot\interior{e_k}(\dA\omega) \\
&= \sum_k \big( (\dA\xi^k)\cdot\interior{e_k}\omega
-\xi_k\cdot \liederivative{e_k}\omega \big)
+(p+1)\cdot \dA\omega
.\end{split} \]
Hence, we obtain
\[ p\cdot \dA\omega
=\sum_k \big( (\dA\xi^k)\cdot\interior{e_k}\omega
-\xi_k\cdot \liederivative{e_k}\omega \big)
+(p+1)\cdot \dA\omega \]
or
\[ \dA\omega
=\sum_k \big( \xi_k\cdot \liederivative{e_k}\omega
-(\dA\xi^k)\cdot\interior{e_k}\omega \big)
.\qedhere\]
\end{proof}

\section*{Acknowledgments}

We wish to thank
Ruggero Bandiera,
Zhuo Chen,
Madeleine Jotz,
Camille Laurent-Gengoux,
Hsuan-Yi Liao,
Rajan Mehta,
and Maosong Xiang
for inspiring discussions.

We dedicate this paper to the memory of Kirill Mackenzie
whose pioneering study of double structures inspired our present work.
We trust that, as time passes, Kirill's work will receive the
recognition it deserves.

\printbibliography
\end{document}